\documentclass[11pt]{amsart}

\usepackage{tikz}
\usetikzlibrary{arrows, positioning}

\usepackage{geometry}                
\usepackage{graphicx}
\usepackage{amsfonts,amsthm,amsmath,amssymb,latexsym, amscd, euscript}
\usepackage{graphicx,color}
\usepackage[all]{xy}
\usepackage{epsfig}
\usepackage{labelfig}
\usepackage{mathrsfs}

\usepackage{amssymb}
\theoremstyle{plain}
\newtheorem{thm}{Theorem}[section]
\newtheorem{cor}[thm]{Corollary}
\newtheorem{prop}[thm]{Proposition}
\newtheorem{lem}[thm]{Lemma}

\newtheorem{rem}[thm]{Remark}

\newtheorem{definition}[thm]{Definition}

\title[Quadratic differentials]{Quadratic differentials and function theory on Riemann surfaces}

\thanks{The author was supported in part by a PSC-CUNY grant.}

\author{Dragomir \v Sari\' c}

\address{PhD Program in Mathematics, The Graduate Center, CUNY \\ 365 Fifth Ave., N.Y., N.Y., 10016 and\newline Department of Mathematics, Queens College, CUNY\\ 65--30 Kissena Blvd., Flushing, NY 11367, USA.}
\email{Dragomir.Saric@qc.cuny.edu}

\begin{document}

\begin{abstract} A finite-area holomorphic quadratic differentials on an arbitrary Riemann surface $X=\mathbb{H}/\Gamma$ is uniquely determined by its horizontal measured foliation. By extending our prior result for $\Gamma$ of the first kind to arbitrary Fuchsian group $\Gamma$, we obtain that a measured foliation $\mathcal{F}$ is realized by the horizontal foliation of a finite-area holomorphic quadratic differential on $X$ if and only if $\mathcal{F}$ has finite Dirichlet integral. 

We determine the image of this correspondence when the infinite Riemann surface has bounded geometry- an extension of the realization result of Hubbard and Masur for compact surfaces. A corollary is that a planar surface $X$ with bounded pants decomposition and with (at most) countably many ends is parabolic, i.e., does not support Green's function, in notation $X\in O_G$ where $G$ is Green's function. 

The class of harmonic functions with finite Dirichlet integral is denoted by $HD$. We give a geometric proof that the class $O_{HD}$ of the Riemann surfaces (that do not support non-constant $HD$-functions) is invariant under quasiconformal maps. Lyons proved that the $O_{HB}$ class (surfaces that do not support non-constant bounded harmonic functions) is not invariant under quasiconformal maps, and it is well-known that the $O_G$ class is invariant. Therefore, the noninvariant class $O_{HB}$ is between two invariant classes: $O_G\subset O_{HB}\subset O_{HD}$.
 \end{abstract}


\maketitle

\section{Introduction}

Let $X$ be a Riemann surface holomorphically covered by the hyperbolic plane $\mathbb{H}$. The covering group $\Gamma$ acts discretely and properly discontinuously by isometries on the universal covering $\mathbb{H}$. The quotient $\mathbb{H}/\Gamma$ is conformal to $X$, and $X$ is said to be {\it infinite} if $\Gamma$ is not finitely generated. 

The hyperbolic area of an infinite Riemann surface $X$ is infinite, and the following dichotomy holds: almost every geodesic on $X$ is recurrent, or almost every geodesic is transient (see Hopf \cite{hopf}). In the former case \cite{hopf}, the geodesic flow on $T^1X$ is ergodic with respect to the natural Liouville measure (which is of infinite volume). The Hopf-Tsuji-Sullivan theorem (see Sullivan \cite{Sullivan} and also \cite{Nicholls1}) states that the geodesic flow is ergodic if and only if the Poincar\'e series for $\Gamma$ is divergent if and only if the Brownian motion is recurrent.

From the complex analysis point of view, fundamental objects for the study of Riemann surfaces are holomorphic quadratic differentials.
A finite-area holomorphic quadratic differential on $X$ gives Euclidean charts $\{ w_i\}$ of finite total area with transitions $w_1=\pm w_2+const$ that are preserving the horizontal and vertical directions. 
While the Euclidean and the hyperbolic metrics are conformal, their relationship is not explicit. In a prior work \cite{Saric23}, the author established that the (hyperbolic) geodesic flow on $T^1X$ is ergodic if and only if almost every horizontal trajectory of every finite-area holomorphic quadratic differential is recurrent. 

Let $A(X)$ be the space of all finite-area holomorphic quadratic differentials on a Riemann surface $X$. The map which assigns to each $\varphi\in A(X)$ the homotopy class of its horizontal foliation is injective for any Riemann surface $X=\mathbb{H}/\Gamma$ (see \cite{MardenStrebel}, \cite{Saric-heights}). 
An important result of Hubbard and Masur \cite{HubbardMasur} states that when $X$ is a compact Riemann surface, each measured foliation on $X$ (with $k$-pronged singularities) is homotopic to a horizontal foliation of a unique finite-area holomorphic quadratic differential on $X$. This can be interpreted as a realization of any measured foliation by the horizontal trajectories of a holomorphic quadratic differential.

By Thurston \cite{Thurston} (see also \cite{Levitt} and \cite{Sarictt}), a foliation on a Riemann surface can be straightened to a measured (geodesic) lamination in its conformal hyperbolic metric. Therefore, we can rephrase Hubbard-Masur result: each measured geodesic lamination on a compact surface can be realized by the horizontal foliation of a unique holomorphic quadratic differential.

When $X=\mathbb{H}/\Gamma$ is an infinite Riemann surface with $\Gamma$ of the first kind, it is known that not every measured lamination (even not every bounded one) can be realized by the horizontal foliation of a finite-area holomorphic quadratic differential (see \cite{Saric-heights}).
The author proved that the image $ML_{int}(X)$ of $A(X)$ in the space of all measured laminations 
consists of laminations that can be realized by partial measured foliations whose Dirichlet integrals are finite (see \cite[Theorem 1.6]{Saric23}). We extend this result to all Fuchsian groups $\Gamma$, including the possibility of $\Gamma =\{ id\}$, i.e., $X=\mathbb{H}$ (see Appendix for the proof).

\begin{thm}
\label{thm:realization-arbitrary}
Let $X=\mathbb{H}/\Gamma$ be an arbitrary Riemann surface. Then, a measured geodesic lamination $\mu$  on $X$ can be realized by the horizontal foliation of a unique finite-area holomorphic quadratic differential if and only if   
$\mu$ is realized by a partial measured foliation with finite Dirichlet integral.
\end{thm}

In other words, the space $A(X)$ is in a one-to-one correspondence with the space $ML_{int}(X)$ of measured laminations homotopic to partial measured foliations with finite Dirichlet integral. We explicitly determine when a measured lamination on $X$ is homotopic to a partial measured foliation with finite Dirichlet integral in terms of the hyperbolic geometry invariants of the surface $X$. 

Given a simple closed geodesic $\gamma$, the {\it intersection number} $i(\gamma ,\mu )$ is the total transverse measure of $\gamma$ with respect to the measured lamination $\mu$. 
Our next result is an explicit characterization of the image $ML_{int}(X)$ of $A(X)$ in terms of the hyperbolic geometry of $X$. Assume $X$ is a Riemann surface with bounded geometry. Then there exists a decomposition of $X$ into geodesic pairs of pants and right-angled hexagons whose sizes are uniformly bounded (see \cite{Kinjo} and \S \ref{sec:foliations}). We prove (see Proposition \ref{prop:int-numbers} and Theorem \ref{thm:holq-ml-bdd})

\begin{thm}
\label{thm:main}
Let $X$ be a Riemann surface with bounded geometry and let $\{\alpha_n\}_n$, $\{\beta_n\}_n$ and $\{\beta_s\}_s$ be three countable families of simple closed geodesics that partition $X$ into bounded shapes as in \S \ref{sec:const-rect}. Then a measured lamination $\mu$ on $X$ is realized by the horizontal foliation of a (unique) finite-area holomorphic quadratic differential $\varphi\in A(X)$ if and only if
$$
\sum_n [i(\alpha_n,\mu )^2+i(\beta_n,\mu )^2]+\sum_s i(\beta_s,\mu )^2 <\infty .
$$
\end{thm}

The above condition on $\mu$ is purely topological. 
The sufficiency of the above condition on a measured lamination $\mu$ is proved by a direct construction of a partial measured foliation from the data of the transverse measure of $\mu$. This construction uses the decomposition of the surface $X$ into countably many pieces with bounded geometry and constructing a partial measured foliation on each piece such that two partial measured foliations agree on the intersections of the pieces (see \S \ref{sec:sufficient-cond}).

In particular, when $X$ has a bounded pants decomposition, we define a train track $\tau$ by adding finitely many edges connecting cuffs in each pair of pants such that it weakly carries $\mu$ (see \cite{Sarictt} and also \S \ref{sec:const-rect}). Then a measured lamination $\mu$ defines a weight function $w_{\mu}:E(\tau )\to\mathbb{R}$ that satisfies switch conditions at vertices, where $E(\tau )$ is the set of edges of $\tau$. Conversely, every weight function $w:E(\tau )\to\mathbb{R}$ defines a measured lamination weakly carried by $\tau$. We prove (see Proposition \ref{prop:int-weights} and Corollary \ref{cor:holq-ml-bdd})

\begin{thm}
\label{thm:weights}
Let $X$ be a Riemann surface with a bounded pants decomposition, and let $\tau$ be a train track constructed by adding finitely many edges to each pair of pants. Then the weight function
$$
w:E(\tau )\to\mathbb{R}
$$
corresponds to a measured lamination $\mu$ that is realized (in the homotopy) by the horizontal foliation of a finite-area holomorphic quadratic differential $\varphi\in A(X)$ if and only if
$$
\sum_{e\in E(\tau )}[w(e)]^2<\infty .
$$
\end{thm}

Given a fixed pants decomposition of a Riemann surface $X$ with covering Fuchsian group $\Gamma$ of the first kind and a measured lamination $\mu$ on $X$, there exists a train track $\tau$ on $X$ obtained by adding finitely many edges connecting cuffs of the pants decomposition that carries the lamination $\mu$ (see \cite{Sarictt}). Therefore, the above two theorems provide parametrizations of the spaces of finite-area holomorphic quadratic differentials on infinite Riemann surfaces with bounded geometry or with bounded pants decompositions. 

\vskip .2 cm

From now on, $X$ is an arbitrary infinite Riemann surface without any assumption on the geometry. The function theory on Riemann surfaces defines various classes of Riemann surfaces based on the non-existence of certain classes of harmonic (non-constant) functions defined on the surfaces (see Ahlfors-Sario \cite {AhlforsSario}, Lyons-Sullivan \cite{LyonsSullivan}). A Green's function $g:X\setminus\{ x_0\}\to\mathbb{R}$ is a harmonic function with a logarithmic singularity at $x_0\in X$ that vanishes at the ideal boundary of $X$. The class of Riemann surfaces that does not support Green's function is denoted by $O_G$, and such surfaces are called {\it parabolic} (see \cite{AhlforsSario}). 
A Riemann surface is parabolic if and only if the geodesic flow on $T^1X$ is ergodic (see \cite{AhlforsSario} and \cite{Sullivan}). 
The class of bounded harmonic function $u:X\to\mathbb{R}$ is denoted by $HB$. The class of Riemann surfaces that do not support a non-constant $HB$-function is denoted by $O_{HB}$ (see \cite{AhlforsSario}). Kaimanovich \cite{Kaim} proved that $X\in O_{HB}$ if and only if the horocyclic flow on $T^1X$ is ergodic. 

The non-existence of Green's function is equivalent to the fact that the extremal distance between a compact part of $X$ and its ideal boundary is infinite (see \cite{AhlforsSario}). The extremal distance is a quasiconformal quasi-invariant and $O_G$ is invariant under quasiconformal maps. Surprisingly, Lyons \cite{Lyons} proved that the class $O_{HB}$ is not invariant under quasiconformal maps. The class $O_{HB}$ is said to satisfy the weak Liouville property. A related class $O_{HD}$ consists of surfaces that do not support a non-constant harmonic function with finite Dirichlet integral. The following inclusions (see \cite{AhlforsSario})
$$
O_G\subset O_{HB}\subset O_{HD}
$$
are proper for non-planar Riemann surfaces. 

In Sario-Nakai \cite{SarioNakai}, Chapter III is devoted to studying Royden's compactification of a Riemann surface and, as an application, it was proved that $O_{HD}$ is invariant under quasiconformal map.  We give a new proof of this fact using the geometric structures induced by quadratic differentials and realization theorem (see Theorem \ref{thm:O_HD)}).

\begin{thm}
\label{thm:inv-HD}
The class $O_{HD}$ is invariant under quasiconformal maps.
\end{thm}

 A harmonic function on a Riemann surface defines an Abelian differential with purely imaginary periods, and conversely, an Abelian differential with purely imaginary periods defines a harmonic function. Thus, we can replace the questions of the existence of harmonic functions with the existence of holomorphic quadratic differentials with special properties. If the square of the Abelian differential is of finite area, then a quasiconformal map will preserve the finite area property (see Theorem \ref{thm:realization-arbitrary},  \cite[Theorem 1.6]{Saric23} and \S \ref{sec:harmonic-classes}).

\vskip .2 cm

The {\it classification problem} is deciding whether an explicitly defined Riemann surface $X$ is parabolic or not, i.e., whether $X\in O_G$ or $X\notin O_G$. The classification problem was considered by many authors when considering specific constructions coming from complex analysis such as gluing slit planes or some covering constructions (see \cite{AhlforsSario}, \cite{GM}, \cite{McKS}, \cite{Merenkov}, \cite{Milnor}, \cite{Nevanlinna:criterion} and references therein). 

A more recent and somewhat different setup for the classification problem is when the surface is assigned in terms of hyperbolic geometry invariants. One such question is to determine when $X\in O_G$ based on its Fenchel-Nielsen parameters (see \cite{BHS}, \cite{HPS}, \cite{Pandazis} and \cite{PandazisSaric}). The classification of surfaces with bounded geometry can be effectively studied in terms of hyperbolic invariants.

The following theorem could (potentially) be deduced by a theorem of Kanai \cite{kanai} by carefully studying the nets on Riemann surfaces with bounded geometry. We give (see Theorem \ref{thm:bounded-pants-graph-O_G}) an independent proof in the Appendix using our Theorem \ref{thm:main} and \cite[Theorem 1.1]{Saric23}.

\begin{thm}
\label{thm:dual-graph}
Let $X$ be a Riemann surface with bounded geometry. Let $\mathcal{G}$ be the graph dual to the decomposition of $X$ into bounded pairs of pants and bounded hexagons. Then
$$
X\in O_G
$$
if and only if the simple random walk on $\mathcal{G}$ is recurrent.
\end{thm}

The above theorem implies certain characterizations for planar surfaces (see Theorems \ref{thm:planar-bounded-finite-ends} and \ref{thm:countable-ends-planar}). Namely,

\begin{thm}
\label{thm:planar-ends}
If $X$ is a planar surface with bounded pants decomposition and at most countably many topological ends, then $X\in O_G$.
\end{thm}

Lyons-McKean \cite{LMcK} and McKean-Sullivan \cite{McKS} proved that the class-surface covering of the thrice-punctured sphere is not a parabolic Riemann surface. The class-surface has one topological end, and it contains fins that have an unbounded injectivity radius, which implies that the surface does not have bounded geometry.

We find examples (described in terms of shears on ideal triangulations) of infinite flute surfaces with bounded geometry that are not parabolic and have covering Fuchsian group of the first kind, answering a question of Sullivan (see Theorem \ref{thm:nonparabolic-bdd-geom}).
\vskip .2 cm

\noindent
{\bf Acknowledgements.} We are grateful to Dennis Sullivan for many inspirational conversations. We are also grateful to Ara Basmajian, Vladimir Markovi\' c, Sergiy Merenkov, Enrique Pujals, and Nick Vlamis for the various discussions. 

\section{Preliminaries}
\label{sec:foliations}

Let $X=\mathbb{H}/\Gamma$ be a Riemann surface, where $\mathbb{H}$ is the hyperbolic plane and $\Gamma$ is a Fuchsian group. We assume that $X$ is {\it infinite}, i.e., the covering group $\Gamma$ is not finitely generated. 

A {\it holomorphic quadratic differential} $\varphi$ on a Riemann surface $X$ is an assignment of a holomorphic function $\varphi (z)$ in each coordinate chart $z$ such that $\varphi (z)dz^2$ remains invariant under the coordinate change. A holomorphic  quadratic differential $\varphi$ is of {\it finite-area} if $$\int_X|\varphi |<\infty,$$ where the absolute value $|\varphi (z)||dz^2|$ is an area form.

Points of $X$ where $\varphi \neq 0$ are called {\it regular points} of $\varphi$. 
A holomorphic quadratic differential $\varphi$ on a hyperbolic Riemann surface $X$ defines a natural parameter in a neighborhood of its regular points as follows. Consider a local chart $z$ of a neighborhood of a regular point $P\in X$, where $z_0$ corresponds to $P$ and $\varphi (z)\neq 0$ in the whole chart. Then $\sqrt{\varphi (z)}$ is a well-defined (up to a sign) holomorphic function in the whole chart, and the {\it natural parameter} is given by
$$
w(z)=u+iv=\int_{z_0}^z\sqrt{\varphi (z)}dz .
$$
Different choices of the local chart $z$, the base point $P$, and the square root give a different natural parameter $w_1$. However, we have $w_1=\pm w+const$ on the intersection of the charts.  Therefore, the horizontal and vertical lines in the $w$-parameter are mapped onto the horizontal and vertical lines in the $w_1$-parameter. A {\it horizontal arc} of $\varphi$ is an arc on $X$ that is the preimage of a horizontal line in the natural parameter of $\varphi$. The image of the natural parameter in a neighborhood of a point where $\varphi$ is not zero is foliated by horizontal arcs.

A {\it horizontal trajectory} of $\varphi$ is a maximal horizontal arc. A horizontal trajectory can either be closed or open. If it is open, then it can accumulate to a zero of $\varphi$ in either direction. Analogously, a {\it vertical arc} is the preimage of a vertical line in the natural parameter, and a {\it vertical trajectory} is a maximal vertical arc.

The {\it horizontal foliation} of $\varphi$ on $X$ is a foliation whose leaves are horizontal trajectories and the transverse measure is given by $\int_{\alpha}|Im(\sqrt{\varphi (z)}dz)|=\int_{\alpha}|dv|$, where $\alpha$ is a differentiable arc transverse to the horizontal leaves. At a zero of $\varphi$ of order $k>0$, called a {\it singular point} of $\varphi$, the horizontal foliation is strictly speaking not a foliation but instead has a well-known structure of $(k+2)$-pronged singularity (see \cite{Strebel}).

We recall a notion of a partial measured foliation from \cite{Saric23} whose definition is motivated by the paper of Gardiner and Lakic (see \cite{GardinerLakic}). For the purposes of this paper we need a slightly less general definition than in \cite{Saric23}.

\begin{definition}
Let $X=\mathbb{H}/\Gamma$ be an infinite Riemann surface such that $\Gamma$ is of the first kind.
A {\it partial measured foliation} $\mathcal{F}$ on $X$ is an assignment of a collection of sets $\{ U_i\}_i$ of $X$ (which do not have to cover the whole surface $X$) and differentiable real-valued functions 
$$
v_i:U_i\to\mathbb{R}
$$
with surjective tangent map at each point. 
The sets $U_i$ are closed Jordan domains with piecewise differentiable Jordan curves on their boundaries.
By the Implicit Function Theorem, the pre-image $v_i^{-1}(c)$ for $c\in\mathbb{R}$ is either a connected differentiable arc (possibly with endpoints) or empty, and 
\begin{equation}
\label{eq:coord-change}
v_i=\pm v_j+const
\end{equation} 
on $U_i\cap U_j$. 
\end{definition}

\begin{rem}
When $U_i\cap U_j$ has non-empty interior, the condition (\ref{eq:coord-change}) is a standard condition as in the case of the transition maps for the natural parameters of holomorphic quadratic differentials. When $U_i\cap U_j=\gamma$ is a differentiable arc, the condition (\ref{eq:coord-change}) imposes the equality of the transverse measures on $\gamma$ coming from $U_i$ and $U_j$. However, we do not require that $v_i:U_i\to\mathbb{R}$ extends to a differentiable function in an open neighborhood of $\gamma$. In fact, the functions $\{v_i\}_i$ glue to a piecewise differentiable function around $\gamma$. 
\end{rem}

For a partial measured foliation $\mathcal{F}$, a {\it horizontal arc} is a curve in $X$ that is a connected union of $v_i^{-1}(c_i)$ for some finite or infinite choice of indices $i$ and real numbers $c_i$. A {\it horizontal trajectory} of $\mathcal{F}$ is a maximal horizontal arc, including the possibility of a closed trajectory.

Given a closed curve $\gamma$ on $X$, the {\it height} of (the homotopy class) of $\gamma$ with respect to a partial foliation $\mathcal{F}$ is given by
$$
h_{\mathcal{F}}(\gamma )=\inf_{\gamma_1} \int_{\gamma_1}|dv|
$$
where the infimum is over all closed curves $\gamma_1$ in $X$ homotopic to $\gamma$. The integration of the form $|dv|$ is independent of the chart by (\ref{eq:coord-change}) and if a subarc $\gamma_2$ of $\gamma_1$  does not intersect the support of $\mathcal{F}$ we set $\int_{\gamma_2}|dv|=0$. 

If $Y$ is a subsurface of $X$ and $\gamma\subset Y$, then we define
$$
h_{\mathcal{F},Y}(\gamma )=\inf_{\gamma_1}\int_{\gamma_1}|dv|
$$
where the infimum is over all closed curves $\gamma_1$ in $Y$ that are homotopic to $\gamma$. 

\begin{definition}
A horizontal trajectory of $\mathcal{F}$ that limits to a point in $X$ is called a singular trajectory.
A partial measured foliation $\mathcal{F}$ on $X$ is said to be proper if, except for countably many singular trajectories, the lift to the universal cover $\mathbb{H}$ of each horizontal trajectory accumulates to an ideal endpoint on $\partial\mathbb{H}$ on each end and the two points of the accumulation are distinct. \end{definition}

If $\mathcal{F}$ is a proper partial foliation of $X$, then the lift $\tilde{\mathcal{F}}$ to $\mathbb{H}$ is a proper partial measured foliation of $\mathbb{H}$. Its set of horizontal leaves is invariant under $\Gamma$.
Each non-singular horizontal leaf of $\tilde{\mathcal{F}}$ has precisely two endpoints, and it can be homotoped to a hyperbolic geodesic of $\mathbb{H}$ modulo ideal endpoints on $\partial\mathbb{H}$. 
We denote by $G(\tilde{\mathcal{F}})$ the closure of the set of geodesics obtained by replacing the non-singular horizontal trajectories of $\tilde{\mathcal{F}}$ with the hyperbolic geodesics that share the same endpoints on $\partial\mathbb{H}$. Note that $G(\tilde{\mathcal{F}})$ is a geodesic lamination in the hyperbolic plane $\mathbb{H}$.

By repeating the arguments in \cite[\S 3.2]{Saric-heights}, the transverse measure to $\mathcal{F}$  induces a transverse measure to $\tilde{\mathcal{F}}$ which in turn induces a transverse measure on $G(\tilde{\mathcal{F}})$. Denote by $\mu_{\tilde{\mathcal{F}}}$ the induced measured lamination in $\mathbb{H}$ and note that it is invariant under the action of $\Gamma$. Therefore the measured lamination $\mu_{\tilde{\mathcal{F}}}$ induces a measured lamination $\mu_{{\mathcal{F}}}$ on $X$.

We assume that the collection  $\{ U_i\}_i$ defining the partial measured foliation $\mathcal{F}$ is locally finite in $\cup_iU_i$, i.e., every compact set in $\cup_iU_i$ intersects only finitely many sets of the collection. 

The Dirichlet integral of $v_i$ is given by $\int_{U_i}[(\frac{\partial v_i}{\partial x})^2+ (\frac{\partial v_i}{\partial y})^2]dxdy$ and by (\ref{eq:coord-change}) we have
$$
\int_{U_i\cap U_j}\Big{[}\Big{(}\frac{\partial v_i}{\partial x}\Big{)}^2+ \Big{(}\frac{\partial v_i}{\partial y}\Big{)}^2\Big{]}dxdy=\int_{U_i\cap U_j}\Big{[}\Big{(}\frac{\partial v_j}{\partial x}\Big{)}^2+ \Big{(}\frac{\partial v_j}{\partial y}\Big{)}^2\Big{]}dxdy.
$$
Using the partition of unity on $X$, the Dirichlet integral $D(\mathcal{F})$ of $\mathcal{F}$ over $X$ is well-defined. If the integration is over a subsurface $Y$ of $X$, denote the corresponding Dirichlet integral by $D_Y(\mathcal{F})$. 

\begin{definition}
A proper partial measured foliation $\mathcal{F}$ on $X$ is called an integrable foliation if $D(\mathcal{F})=D_X(\mathcal{F})<\infty$.
\end{definition}

Consider a finite-area holomorphic quadratic differential $\varphi$ on a hyperbolic Riemann surface $X=\mathbb{H}/\Gamma$  where $\Gamma$ is an arbitrary Fuchsian group (including the possibility of a trivial group, i.e. $\Gamma =\{ id\}$). If $w=u+iv$ is a natural parameter of $\varphi$ with the coordinate chart $U$, then $v:U\to\mathbb{R}$ defines the horizontal foliation $\mathcal{F}_{\varphi}$ of the differential $\varphi$. The horizontal foliation $\mathcal{F}_{\varphi}$ is a proper partial foliation because each non-singular horizontal trajectory of $\varphi$ when lifted to $\mathbb{H}$ accumulates to a single point in $\partial \mathbb{H}$ in each direction. The two limit points are different (see \cite{MardenStrebel}, \cite{Strebel}).

Since $\int_U|\varphi (z)|dxdy=\int_{w(U)}dudv=\int_{w(U)}[(\frac{\partial v}{\partial u})^2+ (\frac{\partial v}{\partial v})^2]dudv$ and by the invariance of the Dirichlet integral under conformal maps, we get
$$
\int_U|\varphi (z)|dxdy=\int_{U}\Big{[}\Big{(}\frac{\partial v}{\partial x}\Big{)}^2+ \Big{(}\frac{\partial v}{\partial y}\Big{)}^2\Big{]}dxdy =D_U(\mathcal{F}_{\varphi}).
$$
Therefore if $\int_X|\varphi (z)|dxdy<\infty$ we have $D(\mathcal{F}_{\varphi})<\infty$, i.e. $\mathcal{F}_{\varphi}$ is an integrable partial foliation.
 
 We proved in \cite{Saric23} that the converse of the above statement is true in the following sense when $\Gamma$ is of the first kind, i.e., when the limit set of the action of $\Gamma$ on $\mathbb{H}$ is $\partial\mathbb{H}$. 
 
 \begin{definition}
 A measured geodesic lamination $\mu$ on a Riemann surface $X$ is said to be realizable by an integrable partial foliation if there is an integrable partial foliation $\mathcal{F}$ such that $\mu_{\mathcal{F}}=\mu$.
 \end{definition}
 
Let ${ML}_{int}(X)$ be the space of all measured geodesic laminations on the Riemann surface $X$ that are realizable by integrable partial foliations. Denote by $A (X)$  the space of all finite-area holomorphic quadratic differentials on $X$. 

\begin{thm}[see \cite{Saric23}]
\label{thm:main}
Let $X=\mathbb{H}/\Gamma$ be an infinite Riemann surface, where $\Gamma$ is a Fuchsian group of the first kind. The map
$
A(X)\to {ML}_{int}(X)
$ defined by
$$\varphi\mapsto \mu_{\varphi}$$ obtained by straightening the trajectories of the horizontal foliation of $\varphi$ is a bijection.
\end{thm}

A Riemann surface $X=\mathbb{H}/\Gamma$ is said to be {\it parabolic} if it does not support a Green's function-i.e., there is no positive harmonic function on $X\setminus\{\zeta\}$ that has a singularity of the form $-\log |z-\zeta |$ in a neighborhood of $\zeta\in X$ expressed in a local chart and tend to $0$ as $z\to\partial_{\infty}X$ (see Ahlfors-Sario \cite{AhlforsSario}). This gives a natural classification of infinite Riemann surfaces: $X$ is of parabolic type, i.e. $X\in O_G$, or $X$ is not of parabolic type, i.e. $X\notin O_G$.
We proved

\begin{thm}[see \cite{Saric23}]
\label{thm:par-ch-cross-cuts}
Let $X=\mathbb{H}/\Gamma$ be a Riemann surface with $\Gamma$ of the first kind. Then $X$ is of parabolic type if and only if the set of horizontal trajectories of each finite-area holomorphic quadratic differential on $X$ that are cross-cuts is of zero measure for the area form induced by the quadratic differential.
\end{thm}

The above theorem is equivalent to

\begin{thm}[see \cite{Saric23}]
\label{thm:non-parabolic-criteria}
Let $X=\mathbb{H}/\Gamma$ be a Riemann surface with $\Gamma$ of the first kind. Then $X$ is not of parabolic type if and only if there is a measured lamination $\mu\in {ML}_{int}(X)$ such that the $\mu$-measure of the geodesics of the support of $\mu$ leaving every compact subset of $X$ is positive.
\end{thm}

\section{Realizing measured geodesic laminations by foliations of unions of annuli and rectangles} 
\label{sec:const-rect}

A Riemann surface $X$ is of {\it bounded geometry} if the injectivity radius is bounded between two positive constants at all points outside of the horodisk neighborhoods of all punctures with boundary horocycle of length $1$ (see \cite{Kinjo}, \cite{BPV}). Kinjo \cite{Kinjo} proved that any Riemann surface $X$ with bounded geometry contains countably many pairwise disjoint simple closed geodesics $\{\alpha_n\}_n$, called {\it cuffs}, whose lengths are between two positive constants such that the connected components of $X\setminus\cup_n\alpha_n$ are either geodesic pairs of pants $\{ P_j\}_j$ with cuffs either in $\{\alpha_n\}_n$ or punctures, and planar components that are decomposed into right-angled hyperbolic hexagons $\{ H_k\}_k$ whose three sides are orthogonal to the cuffs and other three sides are on the cuffs. Moreover, the lengths of the sides of the hexagons are between two positive constants. We note that a cuff may not be on the boundary of a pair of pants; it may connect two planar parts that are decomposed into hexagons. From now on, a {\it planar part} of $X$ refers to the component of $X\setminus\cup_n\alpha_n$ that is decomposed into the hexagons with bounded side lengths.

\begin{figure}[h]
\leavevmode \SetLabels
\endSetLabels
\begin{center}
\AffixLabels{\centerline{\epsfig{file =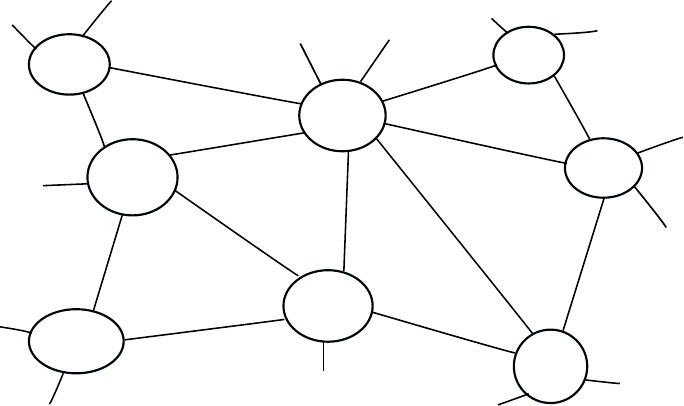,width=10cm,angle=0}}}
\vspace{-20pt}
\end{center}
\label{pants}
\caption{The planar part; closed curves are cuffs, and hexagons have every other side orthogonal to cuffs.} 
\end{figure}

Each cuff can meet only finitely many hexagons, and there is an upper bound on the number of hexagons it can meet since cuffs and sides of the hexagons have lengths between two positive constants. The boundary of each planar component consists of the cuffs of the decomposition (see Figure 1). A planar component is of finite type if its boundary consists of finitely many cuffs or of infinite type if its boundary consists of infinitely many cuffs.  
A pair of pants or another planar part is attached to each cuff on the boundary of a planar part of the surface in the above decomposition.

Let $\mu\in ML_{int}(X)$. 
We form the union of annuli and rectangles such that the geodesics of the support of $\mu$ can be homotoped to the interior of the union. 
Let $m>0$ and $m^*>0$ be fixed and small. For each cuff, we consider an $m$-neighborhood where $m$ is smaller than the collar constant for the infimum of the lengths of the cuffs. These $m$-neighborhoods form the (pairwise disjoint) annuli around each cuff. In each pair of pants, the geodesics of the support of $\mu$ connect pairs of cuffs, including the possibility that a cuff is connected to itself. Every time two cuffs are connected, we draw the unique common orthogonal between the cuffs in the pair of pants. There is a positive lower bound on the distance between any two feet of any two orthogonals on each cuff because of the bounds on the lengths of sides of the hexagons.
We take the $ m^*$-neighborhood of each orthogonal and erase the common intersections with the annuli. The obtained region is a rectangle. 
Since all involved lengths are bounded, we can choose $m$ and $ m^*$ small enough such that the interiors of the rectangles and annuli are disjoint. 
The geodesic sublamination of the support of $\mu$ that consists of geodesics entirely contained in the union of the pairs of pants $\{ P_j\}_j$ can be homotoped to a (non-geodesic) lamination in the 
union of annuli and cuffs (see \cite{Sarictt}). 

In addition, we introduce annuli and rectangles in the planar parts of $X$. First, we extend the annuli around the cuffs on the boundary of planar parts of $X$. We take the one-sided $2m$-neighborhoods of the cuffs in the planar parts. In other words, we added an annuli homotopic to each cuff with the inner boundary on the distance $m$ from the cuff and the outer boundary on the distance $2m$. We take the $m^*$-neighborhoods of each side of the hexagons in the planar parts of $X$ that connect the boundary cuffs. The rectangles are obtained by deleting the parts of the neighborhoods of the sides that lie in the annuli around the boundary cuffs. Since the support of any measured lamination is nowhere dense, we can homotope the geodesics of the support of $\mu$ that are not contained in the union of the pairs of pants to a (non-geodesic) lamination contained in the union of the rectangles and annuli. To perform the homotopy in each hexagon, we choose a component $\Delta$ of the complement of the support of $\mu$ inside the hexagon, as in Figure 2.

\begin{figure}[h]
\leavevmode \SetLabels
\endSetLabels
\begin{center}
\AffixLabels{\centerline{\epsfig{file =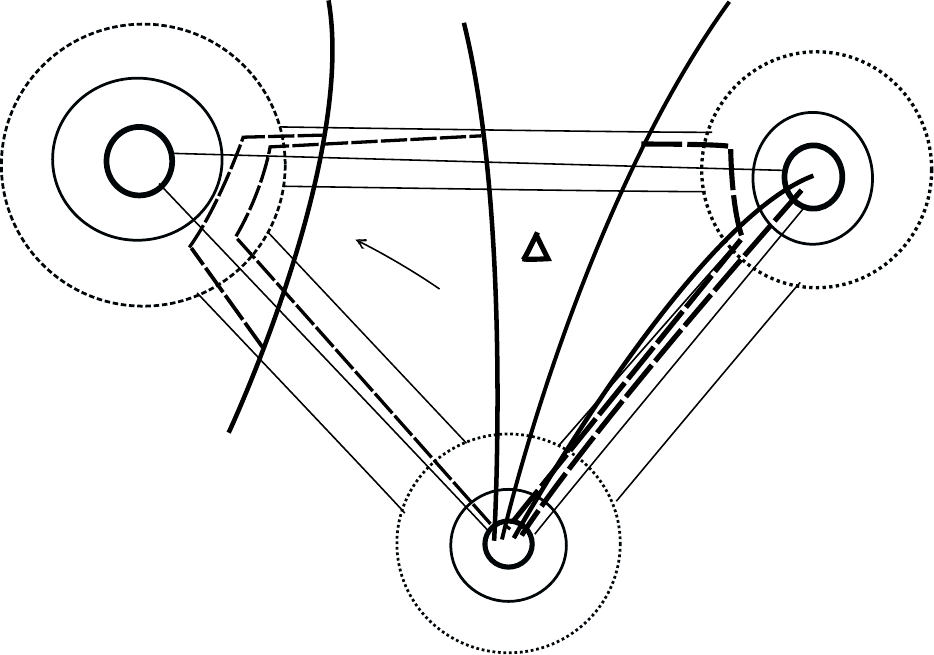,bb= 0 0 350 350,width=8cm,angle=0}}}
\vspace{-20pt}
\end{center}
\label{pants}
\caption{The homotopy moves the intersections of the support of $\mu$ with the hexagons such that the new (non-geodesic) lamination is contained in the union of annuli and rectangles.} 
\end{figure}

The component $\Delta$ is chosen such that if a component of the support of $\mu$ connects two sides of the hexagon that connect cuffs, then $\Delta$ does not separate them. The homotopy pushes the components of the intersection of $\mu$ with the hexagon to the union of rectangles that contain the boundary sides of the hexagon and the outer annuli (distance between $m$ and $2m$ from the cuffs and) concentric to the cuffs. In particular, if a geodesic $g$ of the support of $\mu$ stays in a planar part of $X$, then its image under the homotopy stays in the union of rectangles and outer annuli and does not enter the $m$-neighborhood of the cuffs. If $g$ intersects an edge of a hexagon, then the image $g'$ (of $g$ under the homotopy) either crosses the rectangle that contains the edge or enters the rectangle on one side, which shares with $2m$-neighborhood of a cuff and exists the rectangle from the same side. In the latter case, we homotope $g'$ so it does not enter the rectangle. In the former case, $g$ intersects the opposite sides of the union of two adjacent hexagons, as in Figure 3.

\begin{figure}[h]
\leavevmode \SetLabels
\L(.61*.1) $g$\\
\L(.4*.2) $g'$\\
\endSetLabels
\begin{center}
\AffixLabels{\centerline{\epsfig{file =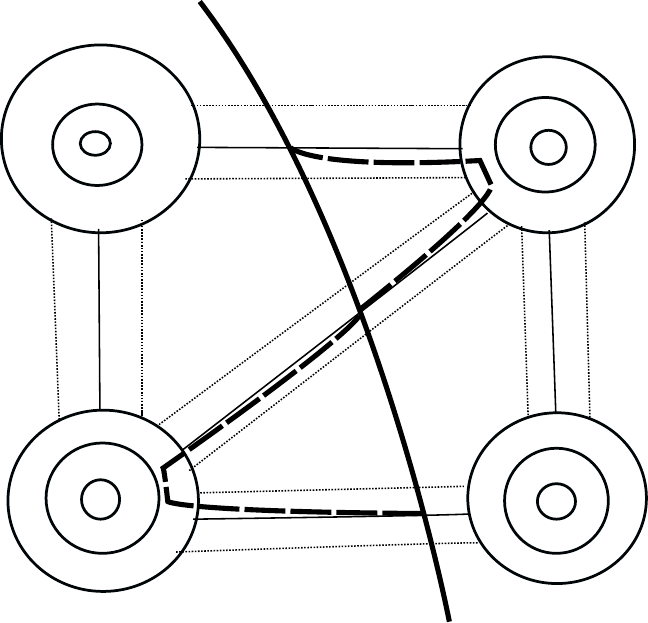,bb= 0 0 300 300,width=10cm,angle=0}}}
\vspace{-20pt}
\end{center}
\label{pants}
\caption{The case when geodesic $g$ crosses a rectangle.} 
\end{figure}

If a geodesic $g$ of the support of $\mu$ enters a planar part through a cuff and never leaves, the same is true for its homotopy image. Finally, if an arc of $g$ connects two cuffs inside the planar part, then this arc of $g$ is homotoped to stay in the union of rectangles and $2m$-neighborhoods of the two cuffs. Again, if $g'$ enters and leaves a rectangle through the same side, we homotope further so that it does not enter the rectangle. 

The geodesic lamination $|\mu |$ inside the union of the pairs of pants $\cup_jP_j$ is homotoped inside the $m^*$-neighborhoods of orthogeodesics to the cuffs and the $m$-neighborhoods of the cuffs according to the construction of the standard Dehn-Thurston train track adjusted to the pairs of pants decomposition for infinite surfaces as in \cite{Sarictt}.

We established

\begin{prop}
\label{prop:ml-support}
Let $X$ be a Riemann surface with bounded geometry and $\mu$ a measured (geodesic) lamination on $X$. Let $\{ P_j\}_j$ and $\{ H_k\}_k$ be the decomposition of $X$ into bounded geodesic pairs of pants and bounded hexagons as above. In each pair of pants, we draw three disjoint orthogeodesics between the cuffs based on the combinatorial position of the support $|\mu |$ of the lamination $\mu$. Denote by $\mathcal{T}$ the union of the $2m$-neighborhoods of cuffs and the $m^*$-neighborhoods of orthogeodesics and sides of hexagons that do not lie on cuffs, where $m>m^*>0$. Then, there is a choice of $m$ and $m^*$ such that the geodesic lamination $|\mu |$ can be homotoped to a lamination whose support is in $\mathcal{T}$. The neighborhoods of the arcs are called rectangles, and a homotoped leaf follows a rectangle if and only if the corresponding geodesic is in the position as in Figure 3. Suppose a geodesic of $|\mu |$ that intersects a planar part does not cross a boundary cuff (of the planar part). In that case, the homotoped geodesic can only enter the concentric annulus around that cuff whose two boundaries are $2m$ and $m$ distance away from the cuff as in Figures 2 and 3. 
\end{prop}

The homotopy of the geodesic lamination $|\mu |$ into the union of annuli and rectangles depends on the choice of $\Delta$ in the hexagons in the planar parts (see Figure 2). Still, it is independent of any choice in the union of the pairs of pants. 

Each annulus of a cuff on the boundary of a planar part is divided into the inside annuli (the $m$-neighborhoods of the cuff) and the outside annulus, which is concentric to the inside annulus and consists of the points on the distance between $m$ and $2m$ from the cuff on the planar side of the cuff. If a cuff is on the boundary of two planar parts, then the annulus around the cuff has two outside annuli, and if a cuff is on the boundary of a single planar part, then the annulus has one outside annulus.
In particular, if $X$ has a bounded pants decomposition (i.e., no planar parts), then the homotopy is unique, and we only consider the annuli that are $m$-neighborhoods of cuffs. 

This homotopy of the support geodesics of the measured lamination $\mu$ into parts of the surface $X$ that are rectangles and annuli is useful when constructing a realization of a measured lamination by a partial measured foliation. The measured lamination $\mu$ assigns a non-negative number (weight) to each rectangle and inside and outside annulus by taking the transverse measure of geodesics of $|\mu |$ that are homotoped in the corresponding piece. We also want to relate these numbers (weights) to invariant quantities (independent of the choices of the homotopies).

We define another family of simple closed geodesics starting from the family of bounded cuffs $\{\alpha_n\}_n$ and sides of the hexagons $\{ H_k\}_k$. If $\alpha_n$ is not on the boundary of a planar part of $X$, then either
there exist two pairs of pants $P_1$ and $P_2$ in $\{ P_j\}_j$ such that  $P_1\cup \alpha_n\cup P_2$ contains $\alpha_n$ in its interior, or there exists a pair of pants $P$ such that $\alpha_n$ is in the interior of $P\cup\alpha_n$. In the former case, let $\beta_n$ be a simple closed geodesic in $P_1\cup \alpha_n\cup P_2$ that intersects $\alpha_n$ in two points and has minimal length among all such geodesics.  In the latter case, let $\beta_n$ be a simple closed geodesic in $P\cup\alpha_n$ that intersects $\alpha_n$ in a single point and has minimal length among all such geodesics.

Next, we consider a side $s$ of a hexagon $H$ that connects two cuffs $\alpha_1$ and $\alpha_2$.
Let $X_s$ be the planar part of $X$ that contains $s$. Then $X\setminus X_s$ is connected to $X_s$ via the cuffs $\alpha_1$ and $\alpha_2$. We choose two simple closed curves in $X\setminus X_s$ that pass through the points $\alpha_1\cap s$ and $\alpha_2\cap s$, are homotopically non-trivial modulo $\alpha_1\cap s$ and $\alpha_2\cap s$, and have the minimal intersections with the cuffs and hexagons of $X\setminus X_s$. There is an upper bound on the number of intersections, which depends on the number of hexagons that meet a single cuff. There are finitely many different topological positions that all give a finite intersection. We form a closed curve by taking the first closed curve in $X\setminus X_s$ attached to $\alpha_1$ at $\alpha_1\cap s$ followed by $s$ then followed by the second closed curve attached to $\alpha_2$ at $\alpha_2\cap s$ and going back to the starting point along $s$. The obtained closed curve is homotopically non-trivial, and the geodesic $\beta_s$ in its homotopy class is simple. In addition, $\beta_s$ intersects both $\alpha_1$ and $\alpha_2$, and it intersects any simple geodesic that also intersects $s$.

\begin{prop}
\label{prop:weights-int-num}
Let $X$ be an infinite Riemann surface $X$ with bounded geometry, and let $\mu$ be a measured lamination on $X$. Divide $X$ into bounded pairs of pants $\{P_j\}_j$ and bounded hexagons $\{ H_k\}_k$, and define the rectangles and annuli as above. Then, the weights on the rectangles and inner and outer annuli are in the $\ell^2$-space of all functions from the rectangles and annuli to the real numbers if and only if
\begin{equation}
\label{eq:sum^2-int}
\sum_n [i(\alpha_n,\mu )^2+i(\beta_n,\mu )^2]+\sum_si(\beta_s,\mu )^2<\infty ,
\end{equation}
where $i(\gamma ,\mu )$ is the total tranverse measure of the geodesics of $|\mu |$ that cross a closed geodesic $\gamma$.
\end{prop}

\begin{proof}
Assume that the weights on the rectangles and inner and outer annuli are in the $\ell^2$-space. Each closed geodesic in the families $\{ \alpha_n\}_n$, $\{\beta_n\}_n$ and $\{\beta_s\}_s$ intersects at most finitely many rectangles and annuli. Thus, (\ref{eq:sum^2-int}) holds since the intersection numbers are bounded by the total weights on the rectangles and annuli that intersect given closed geodesic.

Conversely, assume that (\ref{eq:sum^2-int}) holds. The set of geodesics corresponding to either a rectangle or an annulus intersects at least one closed geodesic in the families $\{ \alpha_n\}_n$, $\{\beta_n\}_n$ and $\{\beta_s\}_s$. If we pick the closest closed geodesic in the families $\{ \alpha_n\}_n$, $\{\beta_n\}_n$ and $\{\beta_s\}_s$ to the rectangle or annulus under consideration, then the weight on each rectangle and each annulus is bounded above by the finite sum of intersections with nearby closed geodesics. We conclude that the weights are in $\ell^2$-space using the Cauchy-Schwarz inequality.
\end{proof}

\subsection{The case of bounded geodesic pants decomposition}
\label{sec:tt}

In this subsection, we assume that $X$ has a bounded pants decomposition with cuffs $\{\alpha_n\}_n$.  
The Dehn-Thurston train track $\Theta$ corresponding to the geodesics $\{\alpha_n\}_n$ is obtained by adding finitely many edges in each pair of pants (see \cite{Sarictt}). In fact, given any $\mu\in ML(X)$ there is a Dehn-Thurston train track $\Theta$ that weakly carries $\mu$ (see \cite{Sarictt}). The measured lamination $\mu$ induces a weight function 
$w:E(\Theta )\to\mathbb{R}$, where $E(\Theta )$ is the set of edges of $\Theta$. We find an equivalent condition to (\ref{eq:sum-intersections}) in terms of the weights on the train track $\Theta$. 

\begin{figure}[h]
\leavevmode \SetLabels
\endSetLabels
\begin{center}
\AffixLabels{\centerline{\epsfig{file =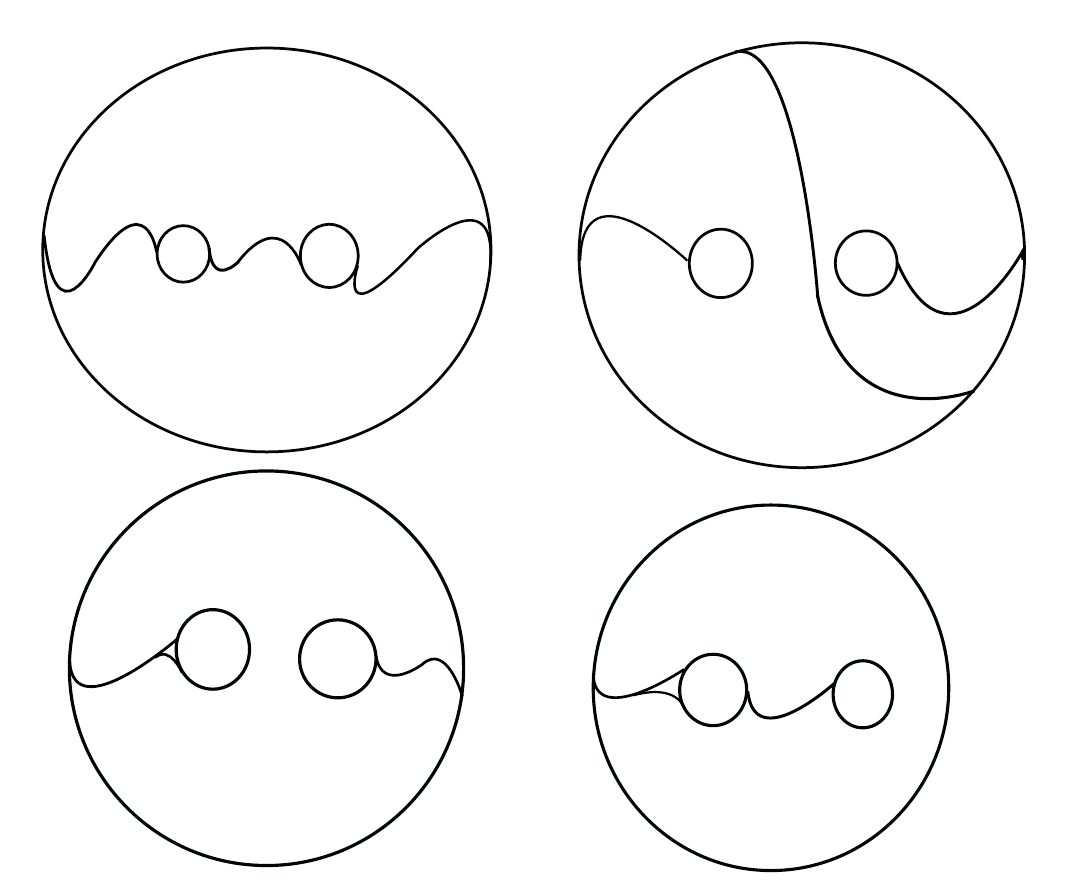,bb= 0 0 400 400,width=10cm,angle=0}}}
\vspace{-20pt}
\end{center}
\label{pants}
\caption{The standard train tracks on pairs of pants with $3$ cuffs.} 
\end{figure}

\begin{figure}[h]
\leavevmode \SetLabels
\endSetLabels
\begin{center}
\AffixLabels{\centerline{\epsfig{file =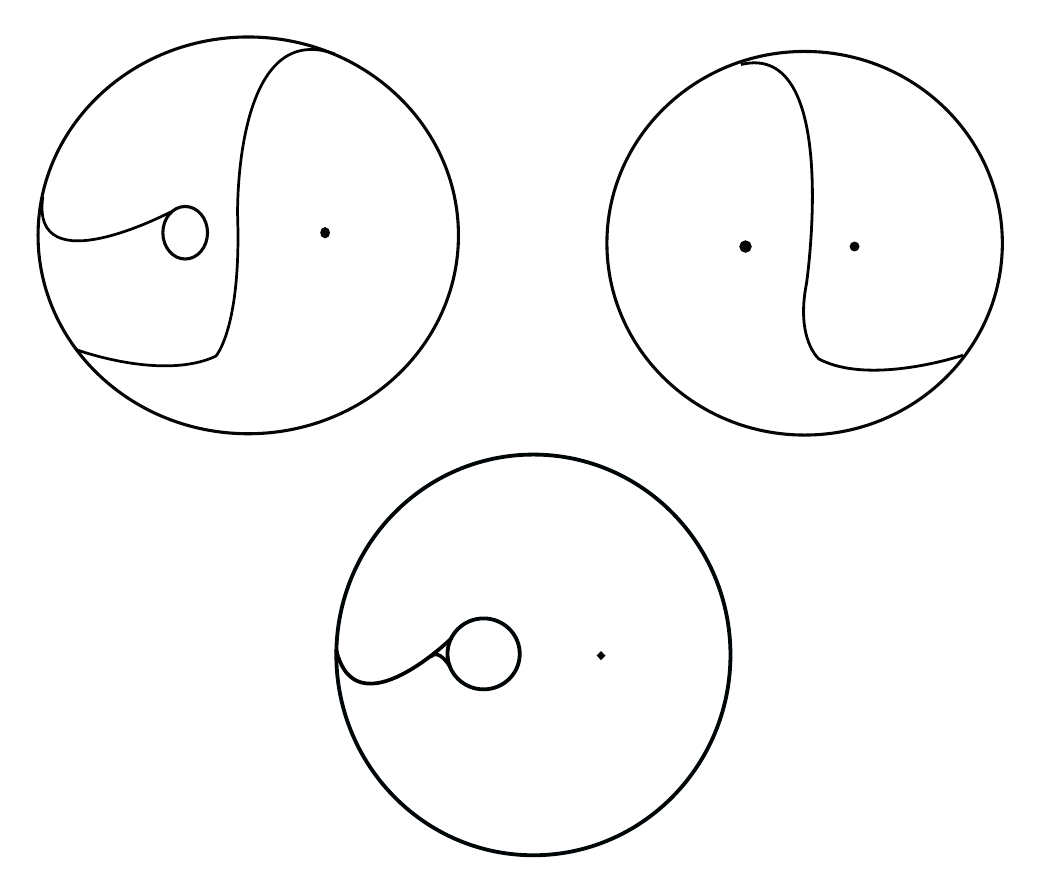,bb= 0 0 500 500,width=10cm,angle=0}}}
\vspace{-20pt}
\end{center}
\label{pants}
\caption{The standard train tracks on pairs of pants with $2$ and $1$ cuffs.} 
\end{figure}

\begin{prop}
\label{prop:int-weights}
Let $\mu\in {ML}(X)$ and $w:E(\Theta )\to\mathbb{R}$ be the corresponding weights on the Dehn-Thurston train tracks $\Theta$ that weakly carries $\mu$. Then
$$
\sum_n [i(\alpha_n,\mu )^2+i(\beta_n,\mu )^2]<\infty
$$
is equivalent to
\begin{equation}
\label{eq:l2-weights}
\sum_{a\in E(\Theta )}w(a)^2<\infty .
\end{equation}
\end{prop}

\begin{proof}
Consider a train track $\Theta$ weakly carrying $\mu$ constructed using the boundary geodesics $\{\alpha_n\}_n$ of the pants decomposition by slightly deforming the standard train tracks in \cite{Sarictt} as in Figures 4 and 5. The train track $\Theta$ is constructed by adding finitely many edges to each pair of pants(see Figures 4 and 5). If the weights on the train track satisfy (\ref{eq:l2-weights}), then it follows that 
$\sum_n [i(\alpha_n,\mu )^2+i(\beta_n,\mu )^2]<\infty$ since there is an upper bound on the number of edges of $\Theta$ that have vertices on each $\alpha_n$ and $\beta_n$.

Conversely, assume that $
\sum_n [i(\alpha_n,\mu )^2+i(\beta_n,\mu )^2]<\infty
$. Then, all the edges on one side of $\alpha_n$ correspond to the leaves that intersect $\alpha_n$. Therefore, the weight of each edge meeting $\alpha_n$ is less than or equal to $i(\alpha_n,\mu )$. Since $\beta_n$ intersects $\alpha_n$, the weight of one edge on $\alpha_n$ of the train track $\Theta_0$ is less than or equal to $i(\beta_n ,\mu)$. Any other edge on $\alpha_n$ has weight bounded by the sum of  $i(\beta_n ,\mu)$ and weights of edges meeting $\alpha_n$ from one side (because of the switch relations). Therefore each edge on $\alpha_n$ has weight bounded by a constant times  $ i(\alpha_n ,\mu)+i(\beta_n ,\mu)$. The condition (\ref{eq:l2-weights}) follows.
\end{proof}

\section{A necessary condition for measured laminations to represent horizontal foliations of finite-area
differentials}

In this section, $X=\mathbb{H}/\Gamma$ is an infinite Riemann surface with bounded geometry. Then, the fundamental group $\Gamma$ of $X$ is of the first kind. The surface $X$ has countably many simple closed mutually disjoint geodesics $\{\alpha_n\}_n$ called {\it cuffs} that satisfy
$$
\frac{1}{C}\leq \ell_X(\alpha_n)\leq C
$$
for a fixed $C>0$ and all $n$. The cuffs decompose $X$ into a countable union of the geodesic pairs of pants $\{ P_j\}_j$ with boundary geodesics $\{\alpha_n\}_n$, and at most countably many planar parts decomposed into the union of right-angled hexagons $\{ H_k\}_k$ with three nonconsecutive sides on the cuffs and the other three sides orthogonal to cuffs and all side having their lengths between $1/C$ and $C$. 

To each cuff $\alpha_n$ that is in the interior of the union of two pairs of pants or a single pair of pants from $\{ P_j\}_j$, there is associated shortest simple closed geodesic $\beta_n$ with bounded length intersecting $\alpha_n$ in one or two points in the corresponding pants. For each side $s$ of a hexagon in $\{ H_k\}_k$ that does not lie on a cuff, there is associated simple closed geodesic $\beta_s$. The geodesics $\beta_s$ for all $s$ have bounded lengths, and they intersect the cuffs that $s$ connects.

Let $\varphi$ be a finite-area holomorphic quadratic differential on $X$ and let $\mu_{\varphi}\in {ML}_{int}(X)$ be the measured lamination corresponding to the vertical foliation of $\varphi$ as in \cite{Saric23}. 
We prove that the intersection numbers of any $\mu_{\varphi}$ with the family $\{ \alpha_n\}\cup\{\beta_n\}\cup\{\beta_s\}$ satisfies an $\ell^2$-summability condition. 

\begin{prop}
\label{prop:int-numbers}
Let $X$ be an infinite Riemann surface with a bounded geometry and cuffs $\{\alpha_n\}_n$ and transverse geodesic families $\{\beta_n\}_n$ and $\{\beta_s\}_s$ defined above. For any $\varphi\in A(X)$, we have
\begin{equation}
\label{eq:sum-intersections}
\sum_n [i(\alpha_n,\mu_{\varphi})^2+i(\beta_n,\mu_{\varphi})^2]+\sum_s i(\beta_s,\mu_{\varphi})^2 <\infty .
\end{equation}
\end{prop}

\begin{proof}
Let $\mathcal{C}_n$ be the standard collar around $\alpha_n$, that is, the set of points in $X$ that are on the distance at most $\sinh^{-1}\frac{1}{\sinh \ell_X(\alpha_n)/2}$ from the geodesic $\alpha_n$. By the Collar Lemma \cite{Buser}, the collars $\mathcal{C}_n$ are mutually disjoint. Let $\mathcal{B}_n$ be the subleaves of the horizontal foliation $\mathcal{F}_{\varphi}$ of $\varphi$ that connect the two boundary components of $\mathcal{C}_n$. Note that by \cite[Lemma 6.1]{Saric-heights}  we have
$$
i(\alpha_n,\mu_{\varphi})\leq i(\alpha_n,\mathcal{B}_n).
$$

Let $I_n$ be the union of countably many vertical arcs that intersect only the leaves of $\mathcal{B}_n$, and at most countably many leaves of $\mathcal{B}_n$ can intersect $I_n$ more than once. 
Let $w:=u+iv=\int_{*}\sqrt{\varphi (z)}dz$ be the natural parameter of $\varphi$. Then we have 
$$i(\alpha_n,\mathcal{B}_n)=\int_{I_n}dv$$
by the definition of the intersection number, where the intersection number is with respect to the family of curves in $\mathcal{C}_n$ homotopic to $\alpha_n$ .

The Cauchy-Schwarz inequality and the above give
\begin{equation}
\label{eq:intersection_cuff}
i(\alpha_n,\mu_{\varphi})^2\leq \Big{(}\int_{I_n}dv\Big{)}^2\leq\int_{I_n}l_n(w)dv\int_{I_n}\frac{1}{l_n(w)}dv
\end{equation}
where $l_n(w)$ is the length of the leaf of $\mathcal{B}_n$ through $w$ in the $\varphi$-metric. 

Note that by the Fubini's theorem $$\int_{I_n}l_n(w)dv=\int_{\widehat{\mathcal{B}}_n}dudv\leq\int_{\mathcal{C}_n}|\varphi (z)|dxdy,$$ where $\widehat{\mathcal{B}}_n$ stands for the region in $\mathcal{C}_n$ which is the union of the leaves of $\mathcal{B}_n$.

In the natural parameter $w=u+iv$ of $\varphi$, the leaves of $\mathcal{B}_n$ are horizontal arcs that are disjoint and parallel. Given $w\in\widehat{\mathcal{B}}_n$, define $\frac{1}{l_n(w)}|dw|$ a conformal metric supported on $\mathcal{C}_n$ which is zero on $\mathcal{C}_n\setminus \widehat{\mathcal{B}}_n$.  This metric is extremal for the family of curves $\mathcal{B}_n$ (see \cite{HakobyanSaric}) and since $|dw|=du$ on the curves in $\mathcal{B}_n$, we obtain
$$
\int_{I_n}\frac{1}{l_n(w)}dv=\mathrm{mod}\mathcal{B}_n
$$
where $\mathrm{mod}\mathcal{B}_n$ is the modulus of the family of curves $\mathcal{B}_n$. 
Further
$$
\mathrm{mod}\mathcal{B}_n\leq\mathrm{mod} \mathcal{C}_n\leq C'
$$
where $\mathrm{mod} \mathcal{C}_n$ is the modulus of all curves in $\mathcal{C}_n$ that connect the two boundary components of $\mathcal{C}_n$. The inequality $\mathrm{mod} \mathcal{C}_n\leq C'$ holds for all $n$ by the lower bound on the lengths $\ell_X(\alpha_n)$ (see \cite{Maskit}, \cite{BHS} or \cite{Saric-heights} for details).

Therefore from (\ref{eq:intersection_cuff}) we obtain
\begin{equation}
\label{eq:int_local}
i(\alpha_n,\mu_{\varphi})^2\leq C'\int_{\mathcal{C}_n}|\varphi (z)|dxdy
\end{equation}
and summing over all $n$ gives
\begin{equation}
\label{eq:int-sum-cuffs}
\sum_n i(\alpha_n,\mu_{\varphi})^2\leq C'\|\varphi\|_{L^1(X)}.
\end{equation}

Consider the family of simple closed geodesics $\{\beta_n\}_n$ associated to $\{\alpha_n\}_n$ as above. Note that $\ell_X(\beta_n)$ is also bounded between two positive constants because $\ell_X(\alpha_n)$ is bounded and $\beta_n$ is chosen to be of minimal length among all closed geodesics intersecting $\alpha_n$ in a minimal number of points inside $P_1\cup\alpha_n\cup P_2$ or $P\cup\alpha_n$. In addition, each $\beta_n$ can intersect at most four other geodesics from the family $\{\beta_n\}_n$. Thus, each standard collar around $\beta_n$ can intersect at most four standard collars around other geodesics in $\{\beta_n\}_n$. Applying the inequality (\ref{eq:int_local}) to the family $\{\beta_n\}_n$ gives
$$
\sum_n i(\beta_n,\mu_{\varphi})^2\leq 4C'\|\varphi\|_{L^1(X)}.
$$

Consider the family of simple closed geodesics $\{\beta_s\}_s$ associated with the sides of the hexagons $s$ not lying on the cuffs. Note that $\ell_X(\beta_s)$ is also bounded between two positive constants because $\ell_X(\alpha_n)$ is bounded, the hexagons are bounded, and $\beta_s$ is chosen to be homotopic to the union of four arcs each of bounded length. In addition, each $\beta_s$ can intersect at most finitely many (uniformly over all $s$) other geodesics from the family $\{\beta_s\}_s$. Thus
$$
\sum_s i(\beta_s,\mu_{\varphi})^2\leq C''\|\varphi\|_{L^1(X)}
$$
and the proposition is proved. 
\end{proof}

\section{A sufficient condition for measured laminations to represent horizontal foliations of finite-area differentials}
\label{sec:sufficient-cond}

In this section, we assume that $X$ is an infinite Riemann surface with bounded geometry. Given $\mu\in ML_{int}(X)$, our goal is to construct a partial measured foliation $\mathcal{F}$ on $X$ whose support homotopes to $|\mu |$ such that its transverse measure agrees with $\mu$ under the push-forward and $D_X(\mathcal{F})<\infty$. Then there exists $\varphi\in A(X)$ such that $\mu_{\varphi}=\mu$. 

To construct $\mathcal{F}$, we will construct measured foliations of the rectangles and annuli defined with respect to $\mu$ in \S\ref{sec:const-rect}. Our first lemma constructs a horizontal foliation of a rectangle whose transverse measure is the Euclidean measure on the vertical segments.

\begin{lem}
\label{lem:vert-rect-fol}
Let $R=[a,b]\times [c,d]$ be a rectangle in the complex plane $\mathbb{C}$. Then the function
$$
v(x,y)=y
$$
defines a foliation of $R$ by leaves $y=const$ and satisfies
$$
D_R(v)=area(R)=(b-a)(d-c).
$$
\end{lem}

\begin{proof}
Note that $(\partial u/\partial x)^2+(\partial u/\partial y)^2=1$. 
\end{proof}

The following lemma provides the connection between two foliations when one foliation is horizontal, and the other is vertical for the canonical Euclidean structure on the rectangles. 

\begin{lem}
\label{lem:L-fol}
Let $R=[a,b]\times [c,d]$ be a rectangle in the complex plane $\mathbb{C}$ with $\frac{d-c}{b-a}\leq C$. Then, there exists a partial measured foliation on $R$ given by
$$
v:R\to\mathbb{R}
$$
such that its leaves connect the top side to the left side of $R$ (see Figure 6), the transverse measure to the left side is the Euclidean measure, the transverse measure to the top side is proportional to the Euclidean measure, and
$$
D_R(v)=C'\cdot area(R),
$$
where $C'=1+C^2$.
\end{lem}

\begin{proof}
We place the bottom left vertex of $R$ to the origin by translating such that $R=[0,b-a]\times [0,d-c]$. Let $t\in [0,d-c]$ and note that the line though point $(0,t)$ of the $y$-axis and slope $\frac{d-c}{b-a}$ has equation
$$
y-t=\frac{d-c}{b-a}x.
$$
The function
$$
v(x,y)=t=y-\frac{d-c}{b-a}x
$$
defines a partial measured foliation of $R$ whose leaves connect the top side to the left side of $R$ and whose transverse measure on the left side is the Euclidean measure of mass $d-c$. The transverse measure of the top side of $R$ is the Euclidean measure multiplied by $\frac{d-c}{b-a}$. Finally, we get
$$
D_R(v)=\Big{[}1+\Big{(}\frac{d-c}{b-a}\Big{)}^2\Big{]}(d-c)(b-a)
$$
\end{proof}

\begin{figure}[h]
\leavevmode \SetLabels
\endSetLabels
\begin{center}
\AffixLabels{\centerline{\epsfig{file =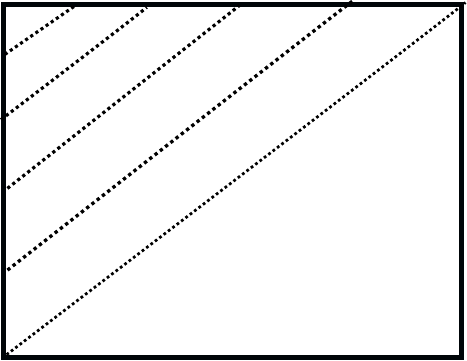,bb= 0 0 200 200,width=6cm,angle=0}}}
\vspace{-20pt}
\end{center}
\label{pants}
\caption{The corner foliation of a rectangle that connects the vertical and horizontal sides.} 
\end{figure}

We also need a lemma that constructs a partial measured foliation of a parallelogram.

\begin{lem}
\label{lem:parallelogram-fol}
Let $Q$ be a parallelogram in the complex plane $\mathbb{C}$ with vertices $(0,0)$, $(a,0)$, $(b,c)$
and $(b+a,c)$. 
Then, there exists a partial measured foliation on $Q$ given by
$$
v:Q\to\mathbb{R}
$$
such that its leaves connect the left side to the right side of $Q$, the transverse measure to the left and the right sides is the Euclidean measure of their projections to a vertical line, and
$$
D_R(v)=area(Q).
$$
\end{lem}

\begin{proof}
The function
$$
v(x,y)=y
$$
defines a partial measured foliation of $Q$ whose leaves connect the left side to the right side of $Q$ and whose transverse measure on the sides is the Euclidean measure of their projections to the vertical lines whose total mass is $c$. Since $(\partial u/\partial x)^2+(\partial u/\partial y)^2=1$, we get
$$
D_Q(v)=area(Q)=ac.
$$
\end{proof}

In addition, we will need partial measured foliations of trapezoids whose one side is orthogonal both base sides. 

\begin{lem}
\label{lem:trapezoid-foliation}
Let $Q$ be a trapezoid whose bases are vertical segments of lengths $a$ and $a_1$, and one side is orthogonal to both bases with length at least $D>0$. Then there exists a partial measured lamination of $Q$ whose leaves are Euclidean lines connecting the two base sides and whose transverse measure on one base side is the Euclidean measure and on the other base side is a multiple of the Euclidean measure. Moreover, the Dirichlet integral is, at most, a constant multiple of the area of $Q$ when $D>0$ is fixed.
\end{lem}

\begin{proof}
We position $Q$ to have one base on the positive $y$-axis starting at the origin and the other base starting at point $D>0$ on the $x$-axis and orthogonal to the $x$-axis while staying in the upper half-plane. Let $0\leq t\leq 1$. We connect the point $ta$ on the left base side of $Q$ to the point at height $ta_1$ on the right base side by Euclidean line. The equation of the line is
$$
y-ta=\frac{ta_1-ta}{D}x.
$$
From this formula, we get $t=y/(\frac{a_1-a}{D}x+a)$. Define
$$
v(x,y)=at=\frac{ay}{\frac{a_1-a}{D}x+a}.
$$
A straightforward differentiation and integration gives
$$
D_Q(v)=\frac{a^2D}{a_1-a}\log\frac{a_1}{a}+\frac{a^2(a_1-a)}{3D}\log\frac{a_1}{a}
$$
which implies the uniform bound when $D>0$ and $a_1\geq a>0$ are bounded from the above.
\end{proof}

We will use the above four lemmas to construct a partial measured foliation on $X$ representing $\mu$. Since we used rectangles and parallelograms to define the foliations above, we will conformally represent the rectangles and annuli from \S\ref{sec:const-rect} as Euclidean rectangles and parallelograms in the complex plane.

\begin{figure}[h]
\leavevmode \SetLabels
\L(.3*.2) $A$\\
\L(.47*.38) $i$\\
\L(.44*.85) $e^{\ell}i$\\
\L(.5*.58) $\ell$\\
\L(.63*.39) $m$\\
\L(.5*.18) $\theta_t$\\
\L(.53*.37) $t$\\
\endSetLabels
\begin{center}
\AffixLabels{\centerline{\epsfig{file =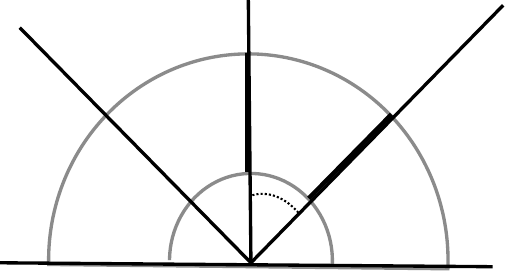,bb= 0 0 240 130,width=8.0cm,angle=0} }}
\vspace{-20pt}
\end{center}
\caption{The lengths of the corresponding arcs.} 
\end{figure}
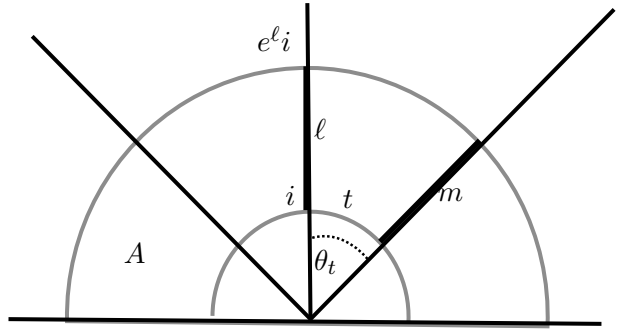

\begin{lem}
 \label{lem:dist-cuff}
Let $A\subset\mathbb{H}$ be the set between the two hyperbolic geodesics orthogonal to the $y$-axis at points $i$ and $e^{\ell}i$. Let $r$ be the sets of points to the right of $y$ on the hyperbolic distance $t>0$ from the $y$-axis. 
 Let $0<\theta_t <\frac{\pi}{2}$ be the Euclidean angle that ray $r$ subtends with the $y$-axis at $0$. Then 
 $$ 
 \cos\theta_t =\frac{1}{\cosh t}
 $$
 and the hyperbolic length $m$ of $r\cap A$ satisfies
 $$
 m=\frac{\ell}{\cos \theta_t}.
 $$
 \end{lem}
 
 \begin{proof}
 The first formula is given in Beardon \cite[\S 7.20]{Bear}. To prove the second formula, note that $z=se^{(\frac{\pi}{2}-\theta )i}$ for $1\leq s\leq e^{\ell}$ parametrizes $r\cap A$. Since $Im(se^{(\frac{\pi}{2}-\theta )i})=s\cos\theta$, a direct integration gives the second formula (see Figure 7).
 \end{proof}

We prove that, when $X$ has bounded geometry, the necessary condition (\ref{eq:sum-intersections}) for $\mu\in {ML}_{int}(X)$ is also sufficient which is the main result of this section.

\begin{thm}
\label{thm:holq-ml-bdd}
Let $X$ be an infinite Riemann surface with bounded geometry and $\mu\in {ML}(X)$. If 
$$
\sum_n [i(\alpha_n,\mu )^2+i(\beta_n,\mu )^2]+\sum_s i(\beta_s,\mu )^2 <\infty
$$
then $\mu\in {ML}_{int}(X)$. 
\end{thm}
 
\begin{proof}
We construct a partial foliation with finite Dirichlet integral that represents $\mu\in ML_{int}(X)$. To do so, recall that there exists a union of annuli around the cuffs and rectangles around the orthogonal arcs between the cuffs (induced by the combinatorics of $|\mu |$) such that $|\mu |$ can be homotoped into the union (see Proposition \ref{prop:ml-support}). The homotopy is not unique, but the weights on the annuli and rectangles determine the measured lamination $\mu$. 

We first consider a rectangle $R_{i,j}$, corresponding to an orthogeodesic arc $o_{i,j}$ between two cuffs $\alpha_i$ and $\alpha_j$. We allow that $i=j$, i.e., $\alpha_i=\alpha_j$, which happens if a geodesic of $|\mu |$ connects the cuff $\alpha_i$ to itself inside a pair of pants.
The rectangle $R_{i,j}$ corresponding to the orthogeodesic $o_{i,j}$ is foliated by lines equidistant to $o_{i,j}$, which will become the leaves of the partial foliation (see Figure 8). The orthogeodesic $o_{i,j}$ is not completely contained in $R_{i,j}$, but we can assume that the length of $o_{i,j}\cap R_{i,j}$ is bounded below by a positive constant by taking $m>0$ sufficiently small.
Let $C_i\subset U$ and $C_j\subset U$ be the $m$-neighborhoods of the cuffs $\alpha_i$ and $\alpha_j$ that $R_{i,j}$ connect if they are not on the boundary of a planar part, and take them to be $2m$-neighborhoods on the sides where they meet planar parts, if any.

Consider two components $\tilde{C}_i$ and $\tilde{C}_j$ of the lifts $C_i$ and $C_j$ to the upper half-plane $\mathbb{H}$ such that the lifts $\tilde{\alpha}_i$ and $\tilde{\alpha}_j$ of $\alpha_i$ and $\alpha_j$ are orthogonal to the $y$-axis and one lift the orthogeodesic between $\alpha_i$ and $\alpha_j$ is connecting $\tilde{\alpha}_i$ and $\tilde{\alpha}_j$ (see Figure 8). 

\begin{figure}[h]
\leavevmode \SetLabels
\L(.52*.25) $\tilde{R}_{i,j}$\\
\L(.27*.5) $\tilde{\alpha}_{i}$\\
\L(.36*.05) $\tilde{\alpha}_{j}$\\
\L(.505*.87) $\phi_0$\\
\L(.57*.13) $\theta$\\
\endSetLabels
\begin{center}
\AffixLabels{\centerline{\epsfig{file =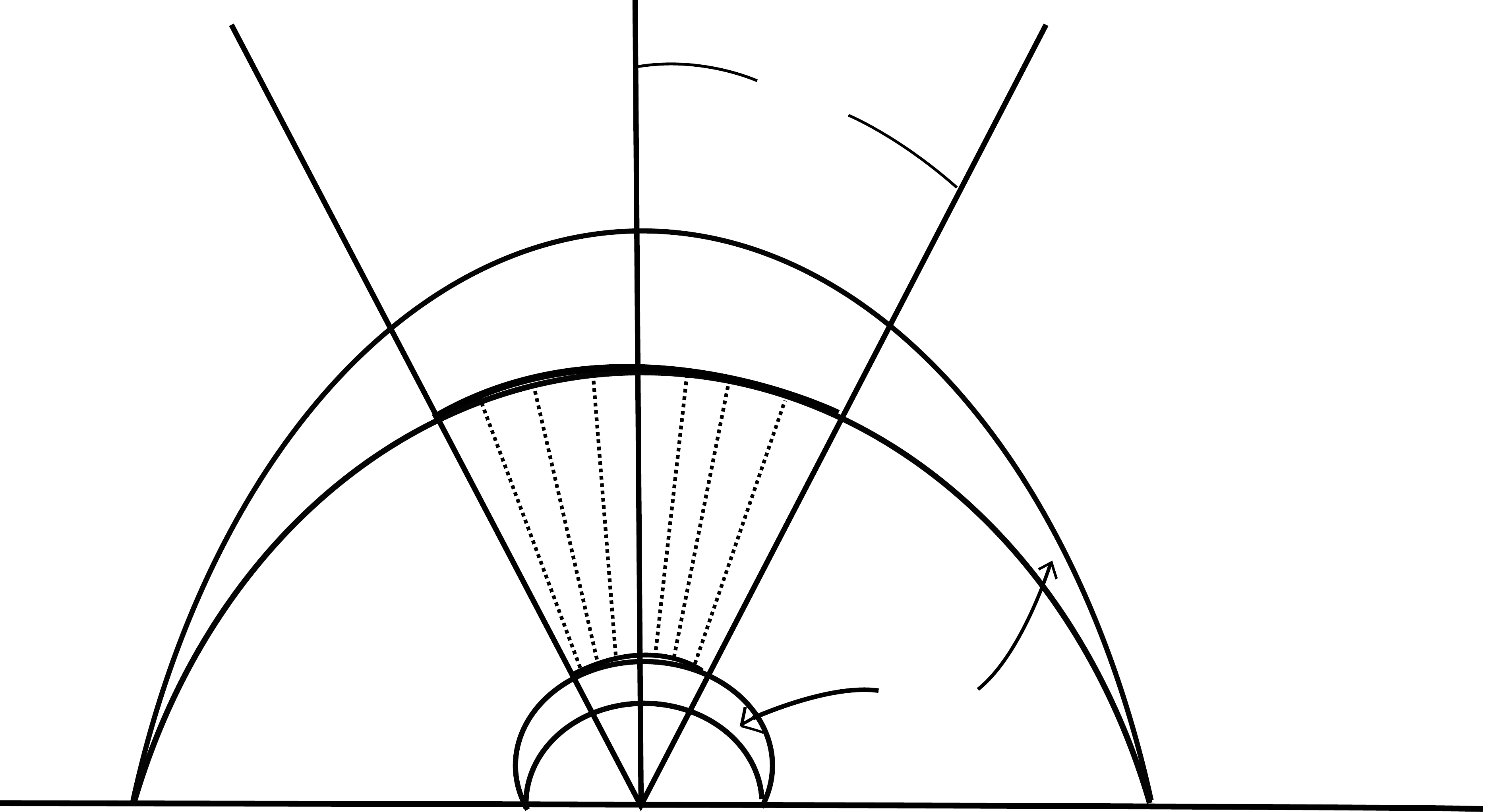,bb= 0 0 1800 1000,width=8.0cm,angle=0} }}
\vspace{-20pt}
\end{center}
\caption{The foliation of $\tilde{R}_{i,j}$ by the arcs of equidistant points to the $y$-axis (dotted lines).} 
\end{figure}
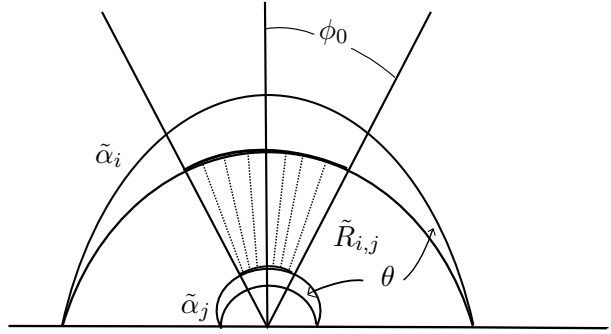

The component $\tilde{R}_{i,j}$ connecting $\tilde{C}_i$ and $\tilde{C}_j$ is between two Euclidean rays starting at $0$ that subtend angles $\phi_0$ with the $y$-axis such that (see Beardon \cite[\S7.20]{Bear})
$$
\sin\phi_0 =\tanh m^*.
$$
Let $\theta$ be the angle between $\partial \tilde{C}_i$ and $\tilde{\alpha}_i$ and thus it satisfies $\sin\theta =\tanh m$. We introduce a partial measured foliation of
 $\tilde{R}_{i,j}$ whose arcs are equidistant points to the $y$-axis (see Figure 8) and whose transverse measure on any geodesic arc orthogonal to the $y$-axis is the hyperbolic length on the arc.
 
Let $\log z=s+{it}$ be the polar coordinates for $z\in \tilde{R}_{i,j}$, where $\frac{\pi}{2}-\phi_0\leq t\leq\frac{\pi}{2}+\phi_0$ and $-\ell (o_{i,j})/2<s<\ell (o_{i,j})/2$. The inequalities for $s$ are strict since the image of $\tilde{R}_{i,j}$ does not meet the vertical sides of the rectangle  $[\frac{\pi}{2}-\phi_0, \frac{\pi}{2}+\phi_0]\times [-\ell (o_{i,j})/2,\ell (o_{i,j})/2]$. If $m(z)$ is the hyperbolic distance from $z$ to the $y$-axis, then the above formula gives
$$
|\cos\phi |=\tanh m(z).
$$
The function
$$
f(s+{\bf i} t)=s+{\bf i}\cdot\mathrm{sgn} (t-\pi /2)\cdot\tanh^{-1} [\cos (t)]
$$ 
is quasiconformal on $[\frac{\pi}{2}-\phi_0, \frac{\pi}{2}+\phi_0]\times [-\ell (o_{i,j})/2,\ell (o_{i,j})/2]$, where $\mathrm{sgn} (x)$ is the sign of $x$ and the quasiconformal constant continuously depends on $\phi_0>0$ (which in turn depends on $m^*$). Finally, the composition of $\log z$ with $f(s+{\bf i}t)$ is an isometry on the geodesic arcs orthogonal to the $y$-axis. Therefore, the image of $\tilde{R}_{i,j}$ is in a rectangle $R$ with the height $2m^*$ and the length $\ell (o_{i,j})$. We define a partial foliation of $R$ by
$$
{v}(s+{\bf i}t)=t.
$$
The leaves are horizontal lines, the Dirichlet integral $D_R(v)\leq 2m^*\ell (o_{i,j})$ and the pull-back of $\tilde{v}_{i,j}$ to $\tilde{R}_{i,j}$ gives a partial measured foliation with leaves as in Figure 8 and the Dirichlet integral bounded above by a constant multiple of the hyperbolic area on $\tilde{R}_{i,j}$ (the Dirichlet integral of the pre-composition of a function with a $K$-quasiconformal map is at most $K$ times the Dirichlet integral of the original function). The transverse measure of the foliation is given by the hyperbolic length on the arc orthogonal to the $y$-axis. We denote by $v_{i,j}:R_{i,j}\to\mathbb{R}$ the partial measured foliation on the rectangle $R_{i,j}\subset X$ with the desired properties.

We form a partial foliation of $C_i$. 
There is a uniform bound on the number of orthogeodesics meeting $\alpha_i$ from each of its sides and the distance between the foots of the orthogeodesics on each sides are bounded below and above by two positive constants. 
Let $\tilde{C}_i$ be a single component lift of $C_i$ to the upper half-plane such that $\tilde{\alpha}_i$ is the $y$-axis. Take the fundamental region for the covering $\tilde{C}_i$ of $C_i$ to be between the semicircles $|z|=1$ and $|z|=e^{{\bf i}\ell_X(\alpha_i)}$ in the upper half-plane $\mathbb{H}$. The boundary of the fundamental region consists of two subarcs of $|z|=1$ and $|z|=e^{{\bf i}\ell_X(\alpha_i)}$, and two subarcs of two Euclidean rays starting at $0$ that subtend angles $\theta$ on both sides of the $y$-axis. The lifts $\tilde{R}_{i,j}$ meet $\tilde{C}_i$ at the boundary arcs on the Euclidean rays from $0$ (see Figure 9). 
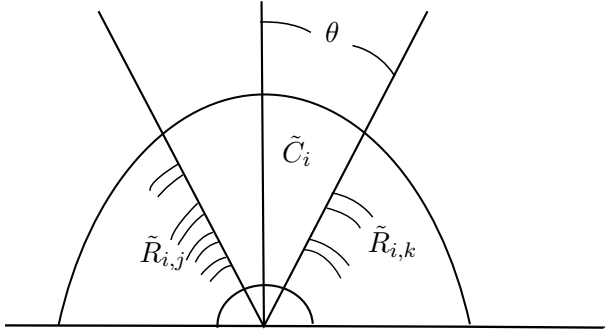
\begin{figure}[h]
\leavevmode \SetLabels
\L(.32*.2) $\tilde{R}_{i,j}$\\
\L(.56*.23) $\tilde{R}_{i,k}$\\
\L(.47*.5) $\tilde{C}_i$\\
\L(.515*.86) $\theta$\\
\endSetLabels
\label{fig9}
\begin{center}
\AffixLabels{\centerline{\epsfig{file =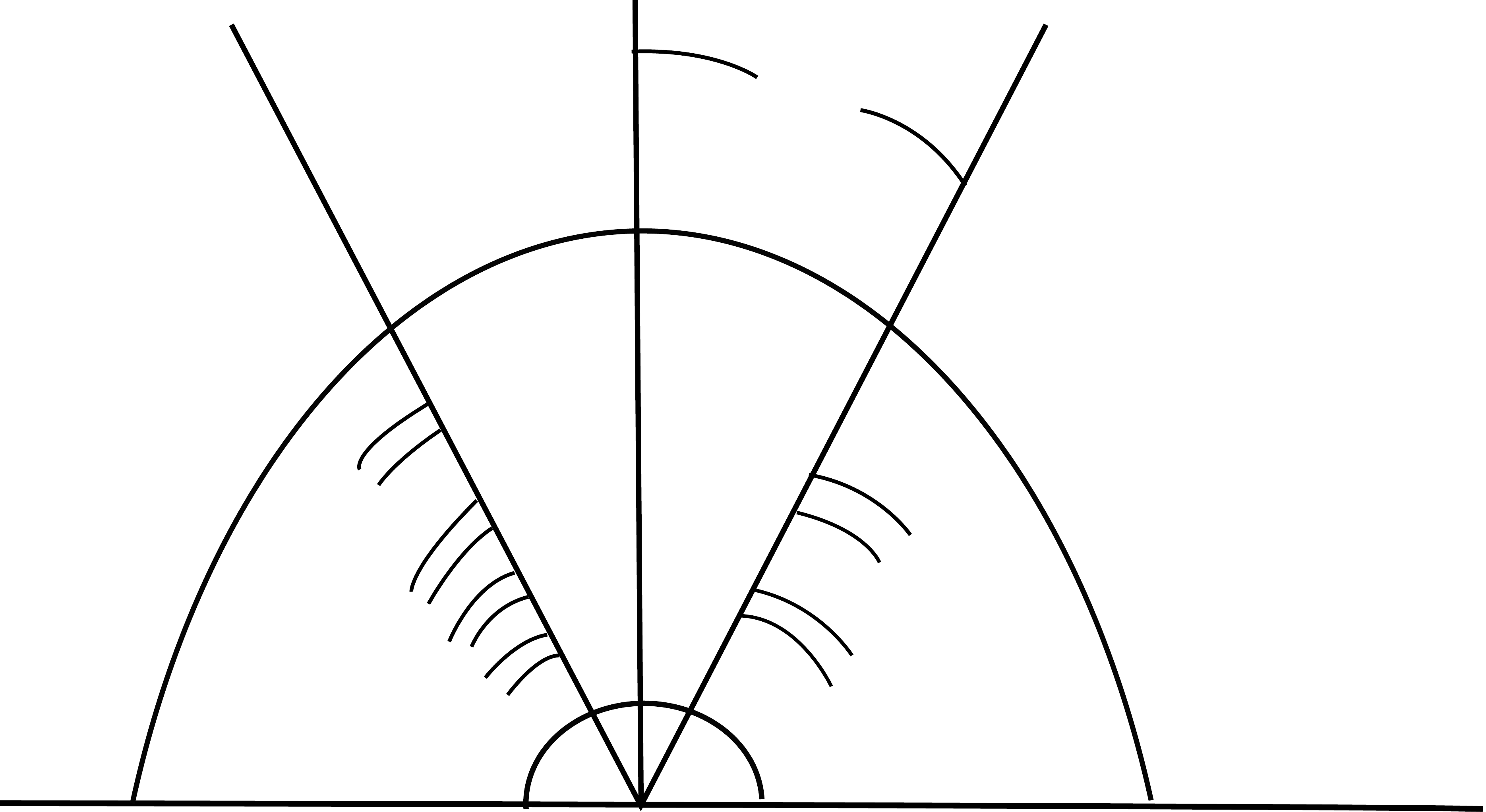,bb= 0 0 1800 1000,width=8.0cm,angle=0} }}
\vspace{-20pt}
\end{center}
\caption{The collar ($m$-neighborhood) around $\alpha_i$ is lifted to $\tilde{C}_i$.} 
\end{figure}

Let $Q$ be the image of the fundamental region under the conformal map $z\mapsto \log z$, where the argument of $z\in \mathbb{H}$ is taken in $(0,\pi )$ (see Figure 11). 
Then $Q=[0,\ell_X(\alpha_i)]\times [\frac{\pi}{2} -\theta ,\frac{\pi}{2}+\theta ]$ and the identification by the Euclidean translation $z\mapsto z+1$ of the vertical sides of $Q$ is conformal to the collar $C_i$. The image of $\tilde{R}_{i,j}\cap \tilde{C}_i$ consists of finitely many intervals on the top and the bottom side of $Q$ 
that have length $2m^*$ and are separated by intervals whose Euclidean length is bounded below by a positive constant depending on $m^*>0$. The intervals can be as small as needed by choosing $m^*>0$ small enough. Since the hyperbolic distance between ${\bf i}y_1\in\mathbb{H}$ and ${\bf i}y_2\in\mathbb{H}$ is $\log y_2/y_1$, it follows that the Euclidean measure of the image of $\tilde{R}_{i,j}\cap \tilde{C}_i$ on the horizontal boundary of $Q$ (under the map $\log z$) is given by the hyperbolic length of the projection of $R_{i,j}\cap C_i$ onto $\alpha_i$ (whose mass is $2m^*$).


We first deal with the case when some leaves of $|\mu |$ are homotoped to enter the outer annulus of $C_i$ through a rectangle $R_{i,j}$ and leave it on the same side through a rectangle $R_{i,k}$ different from $R_{i,j}$. We consider the lifts $\tilde{C}_i$, $\tilde{R}_{i,j}$ and $\tilde{R}_{i,k}$ to the universal cover and their images under $\log z$. In the $\log z$ coordinates, the distance between $\tilde{R}_{i,j}$ and $\tilde{R}_{i,k}$ is bounded below by a universal constant $D$. The $\mu$-measure of the geodesics in $R_{i,j}$ homotoped into the outer annulus around $\alpha_i$ that leaves through $R_{i,k}$ is denoted by $s_{j,k}^i$. The ratios $a_j=s_{j,k}^i/\mu (R_{i,j})$ and $a_k=s_{j,k}^i/\mu (R_{i,k})$ can be different. The proportion of the vertical lines of $\tilde{R}_{i,j}$ that leaves through $\tilde{R}_{i,k}$ is $2m^*a_j$ and the proportion of the vertical lines of $\tilde{R}_{i,k}$ that leaves through $\tilde{R}_{i,j}$ is $2m^*a_k$. We extend these parts of $\tilde{R}_{i,j}$ and $\tilde{R}_{i,k}$ into squares with top sides $\tilde{R}_{i,j}\cap\tilde{C}_i$ and $\tilde{R}_{i,k}\cap\tilde{C}_i$, and the side lengths $2m^*a_j$ and $2m^*a_k$. The foliation is extended vertically, going down until it hits the main diagonal, and then extended horizontally until it hits the right vertical side (see Figure 10). The total Dirichlet integral of this part of the foliation is the Euclidean area of these squares by Lemma \ref{lem:vert-rect-fol}. It remains to connect the sides of the squares by a foliation. We make a trapezoid $T_{j,k}$ whose parallel sides are the vertical sides of the squares, as in Figure 10. Using Lemma \ref{lem:trapezoid-foliation}, we form a partial measured foliation of $T_{j,k}$ such that its leaves are the Euclidean lines in $T_{j,k}$ connecting the parallel sides whose transverse measure on the side corresponding to $R_{i,j}$ is equal to the Euclidean measure and the transverse measure on the side corresponding to $R_{i,k}$ is proportional to the Euclidean measure. 

\begin{figure}[h]
\leavevmode \SetLabels
\L(.26*.9) $\tilde{R}_{i,j}$\\
\L(.63*.9) $\tilde{R}_{i,k}$\\
\endSetLabels
\begin{center}
\AffixLabels{\centerline{\epsfig{file =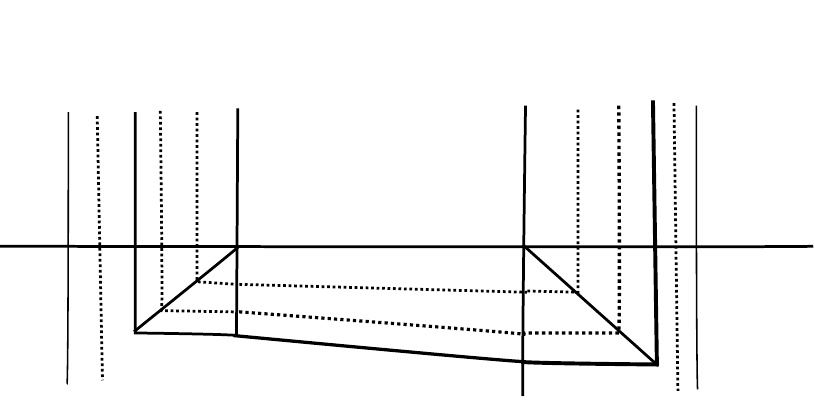,bb= 0 0 390 170,width=8.0cm,angle=0} }}
\vspace{-20pt}
\end{center}
\label{fig10}
\caption{The foliation in $Q_i$ that corresponds to the leaves of $|\mu |$ entering only the outer annulus.} 
\end{figure}

We next consider the leaves of $|\mu |$ that cross the cuff $\alpha_i$, which implies that the homotoped leaves are entering the inner annulus of $\alpha_i$ and exiting on the other side. We introduce a partial measured foliation of $C_i$ corresponding to this case. 
To simplify the construction, we can assume that each homotoped leaf corresponding to a geodesic of $|\mu |$ essentially intersects each geodesic arc in $C_i$ orthogonal to $\alpha_i$ at least twice. If this is not the case, we can apply the square of the Dehn twist around each $\alpha_i$ where this is not the case. Since the composition of these Dehn twists is a quasiconformal map, it follows that the partial measured foliation represnting the new (twisted) lamination is mapped to a partial measured foliation (by the inverse of the quasiconformal map) representing the original lamination with a bounded Dirichlet integral. 

Therefore, the total transverse measure across an arc orthogonal to $\alpha_i$ that connects the two boundaries of $C_i$ is at least three times the transverse measure of the leaves entering the inner annulus of $\alpha_i$ from one side (which is equal to the transverse measure of the leaves entering the annulus from the other side).

Let $\tilde{C}_i$ be a fundamental domain of the lift of $C_i$ to the upper half-plane $\mathbb{H}$ such that $\alpha_i$ is covered by the $y$-axis. Let $Q_i$ be the image of $\tilde{C}_i$ under $\log z$. The height of $Q_i$ is $2\theta$ and the length is $\ell_X(\alpha_i)$. 
When the vertical sides of $Q_i$ are identified by the Euclidean translation, we obtain the cylinder conformal to $C_i$
We divide each $Q_i$ into several subdomains and define the partial foliation in each subdomain such that it extends the partial foliation already defined on the rectangles and outer annuli, if any. 

\begin{figure}[h]
\leavevmode \SetLabels
\L(.12*.5) $2\theta$\\
\L(.255*.95) $b_1$\\
\L(.42*.95) $b_2$\\
\L(.56*.95) $b_3$\\
\L(.71*.95) $b_4$\\
\L(.28*.02) $c_1$\\
\L(.455*.02) $c_2$\\
\L(.645*.02) $c_3$\\
\endSetLabels
\begin{center}
\AffixLabels{\centerline{\epsfig{file =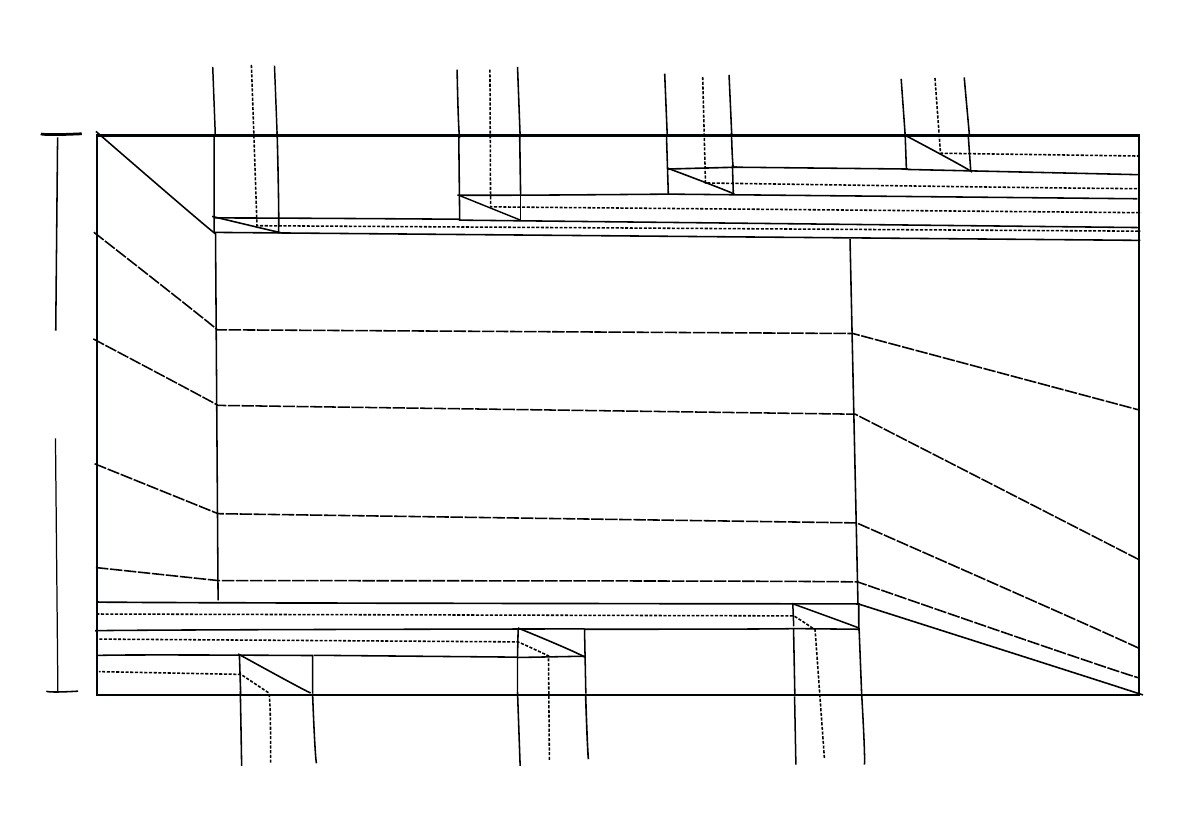,,bb= 0 0 540 400,width=10.0cm,angle=0} }}
\vspace{-20pt}
\end{center}
\label{fig11}
\caption{The partial measured foliation of $Q_i$.} 
\end{figure}
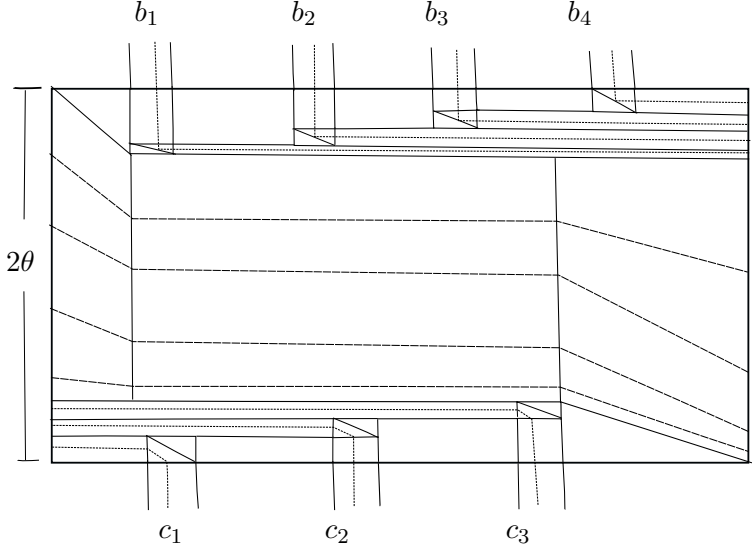


Let $a$ be the difference of the supremum of the intersection numbers of geodesic arcs orthogonal to $\alpha_i$ connecting the two boundaries of $C_i$ and the measured lamination $\mu$, and the sum of the weights of all the rectangles meeting $C_i$ from one side.  
We define a partial foliation of $Q_i$, which depends on the weights of the rectangles meeting $C_i$ (note that we did not use the weights for defining a partial foliation of $R_{i,j}$). Let $b_j$ be the weights of the rectangles $\tilde{R}_{i,j}$ incoming to the top of $Q_i$, and we assume that the twisting is such that the leaves in $\tilde{R}_{i,j}$ are incoming from the right (the other case is analogous). Then, the leaves in all the edges $R_{i,k}$ incoming to the bottom of $Q_i$ are from the right and denote their weights by $c_k$ (see Figure 11). Then $\sum_{l}b_l=\sum_{h}c_h$.

Let $b_j'=\frac{ 2\theta b_j}{a+\sum_lb_l}$ and $c_k'=\frac{ 2\theta c_k}{a+\sum_hc_h}$, where $2\theta$ is the length of the vertical  boundary sides of $Q$. 
It follows that $\sum_lb_l'$ and $\sum_hc_h'$ are each less than or equal to the length of the vertical side of $Q$. 

We extend each $\tilde{R}_{i,j}$ by starting from the leftmost by adding a rectangle whose top side is on $Q_i\cap\tilde{R}_{i,j}$ and the height is $b_j'$. The foliation of $\tilde{R}_{i,j}$ is extended by foliating the part above the main diagonal, as in Figure 6. We point out that this extension is different than in the previous case. Since the ratio $b_j'/b_j$ is bounded, the Dirichlet integral of the extension is also bounded by Lemma \ref{lem:L-fol}. We extend this rectangle by attaching a rectangle whose right vertical side is equal to the left vertical side of the above rectangle, and the left vertical side lies on the left vertical side of $Q_i$. The foliation by the horizontal leaves and Lemma \ref{lem:vert-rect-fol} the Dirichlet integral is bounded. We continue to extend all other rectangles $\tilde{R}_{i,j}$ from the left side and obtain foliations whose Dirichlet integrals are bounded (see Figure 11). Analogously, we extend the rectangles $\tilde{R}_{i,k}$ that meet $Q_i$ on its bottom side. The two partial foliations meet the left and the right vertical side of $Q_i$, and they need to be connected to form a partial foliation of the whole surface of $X$. We make two trapezoids and one rectangle to complete the foliation of $Q_i$. The sides of the trapezoids that are slanted have slopes $\sum b_j'/D'$, where $D'$ is the distance between the left vertical side of $Q_i$ and the leftmost $\tilde{R}_{i,j}$. The constant $D'$ is bounded by the half of the shortest distance between the rectangle incoming to the top side of $Q_i$ when the cut arc of $C_i$ defining $Q_i$ is appropriately chosen. Therefore, the foliation of two trapezoids chosen to have transverse measure equal to the Euclidean measure on the base sides belonging to the vertical sides of $Q_i$ have bounded Dirichlet integrals. The natural extension to the rectangle between them also has a bounded Dirichlet integral.

The leaves of partial foliations of the above rectangles and trapezoids glue together to make a proper partial unmeasured foliation of $X$, whose each leaf is homotopic to a geodesic of $|\mu |$. Conversely, every geodesic of $|\mu |$ is homotopic to a unique leaf of the proper unmeasured partial measured foliation. 
We multiply functions defining the partial measured foliation on various rectangles and trapezoids by appropriate positive numbers to make a measured proper partial foliation on $X$ such that the weights $w$ are achieved. For each $v_{i,j}:R_{i,j}\to\mathbb{R}$, we define $v_{i,j}^*=\frac{w(R_{i,j})}{2m^*}v_{i,j}$. The weight of $R_{i,j}$ for $v_{i,j}^*$ is $w(R_{i,j})$ and $D_{R_{i,j}}(v_{i,j}^*)\leq C\cdot [w(R_{i,j})]^2$, where $C>0$ is a universal constant. 

Assume that $R_{i,j}$ connects to a square $S$ inside $C_i$ corresponding to the geodesics of $|\mu |$ that leave $C_i$ on the same side of $C_i$. Then the function $v:S\to\mathbb{R}$ defines a partial measured foliation on the rectangle with leaves making an $L$-shape as in Figure 10 induces a transverse measure on the top and right vertical side equal to the Euclidean measure of total mass $2m^*a_j$. The Dirichlet integral $D_S(v)$ is the Euclidean area of $S$. We define $v^*=\frac{w(R_{i,j})}{2m^*}v$ and note that the transverse measure on the top side agrees with the transverse measure from $v_{i,j}^*$ on the portion where they meet. In addition, $D_S(v^*)\leq C [w(R_{i,j})]^2$ for a universal constant $C>0$.

We next adjust the transverse measure induced by the function $v_{j,k}:T_{j,k}\to\mathbb{R}$ on the trapezoid $T_{j,k}$ that connects two squares attached to $R_{i,j}$ and $R_{i,k}$ (see Figure 10). We can assume that the transverse measure to the foliation of $T_{j,k}$ is equal to the Euclidean measure of the right vertical side of the square attached to $R_{i,j}$ (the other case is symmetric). We define $v_{j,k}^*=\frac{w(R_{i,j})}{2m^*}v_{j,k}$. The partial foliation of the trapezoid $T_{j,k}$ agrees on both of its vertical sides with the partial foliation induced on the attached squares and $D_{T_{j,k}}(v_{j,k}^*)\leq C [w(R_{i,j})]^2$ for a universal constant $C>0$.

It remains to consider the foliation of the part of $Q_i$ that represents the part of $|\mu |$ that crosses $\alpha_i$. For each $R_{i,j}$, there is (a possibly empty) rectangle entering $Q_i$ that is foliated by vertical lines, followed by a rectangle where either the part above or the part below the diagonal is foliated by lines parallel to the diagonal, and then followed by another rectangle that ends on a vertical side of $Q_i$. By Lemmas \ref{lem:vert-rect-fol}, \ref{lem:L-fol} and \ref{lem:trapezoid-foliation}, the total Dirichlet integral over these pieces is uniformly bounded. The transverse measure of the partial foliation of the rectangle with the foliation by vertical lines is the Euclidean measure on the horizontal lines whose mass is $2m^*$. We multiply the function defining the foliation by $\frac{a+\sum_lb_l}{ 2\theta}$ so that the new measure agrees with the measure from $R_{i,j}$ and the Dirichlet integral is bounded above by $C(a+\sum_lb_l)^2\leq C'[a^2+\sum_lb_l^2]$, where the last inequality follows by the Cauchy-Schwarz inequality. The other parts are treated similarly, and we obtain the same inequality for the Dirichlet integrals. The two trapezoids and the rectangle that separate the parts where $R_{i,j}$ and $R_{i,k}$ enter $Q_i$ are easily estimated using the same lemmas and multiplying by the appropriate constants so that the transverse measures of the pieces agree on their common sides. The proper partial foliation has the desired weights on the rectangles and annuli, representing the geodesic lamination $\mu$.  

In conclusion, the Dirichlet integral of the measured foliation is bounded above by a constant multiple of the sum of the square of the weights on the rectangles $R_{i,j}$ connecting the cuffs and the sum of $[a^2+\sum_lb_l^2]$ for each cuff. Since $a$ is less than $i(\beta_n,\mu )$, we conclude that the Dirichlet integral is less than a constant multiple of the sum of the squares of the intersection numbers. 
\end{proof}

By the above Theorem and Proposition \ref{prop:int-weights} we obtain
\begin{cor}
\label{cor:holq-ml-bdd}
Let $X$ be an infinite Riemann surface with a bounded geodesic pants decomposition and $\mu\in {ML}(X)$. Let $\Theta$ be the standard train track that carries $\mu$. If the induced weights $w_{\mu}=w:E(\Theta )\to\mathbb{R}$ satisfy 
$$
\sum_{e\in E(\Theta )}[w(e)]^2<\infty
$$
then $\mu\in {ML}_{int}(X)$. 
\end{cor}

\section{A class of harmonic functions on Riemann surfaces}
\label{sec:harmonic-classes}

In this section, $X$ is an arbitrary Riemann surface. 
Recall that $X\in O_G$ if does not support Green's function. A Green's function is a real-valued function $u:X\to\mathbb{R}$ that is harmonic in $X\setminus\{ z_0 \}$, $u(z)\asymp -\log |z-z_0|$ near $z_0$ and $u(z)\to 0$ as $z\to\partial_{\infty}X$, where $\partial_{\infty}X$ is the ideal boundary of $X$. If $X$ supports Green's function with the singularity at some $z_0$, then $X$ supports Green's function with the singularity at any point of $X$.

In addition to Green's function, a Riemann surface $X$ can support a harmonic function $u:X\to\mathbb{R}$ with additional properties (see Ahlfors-Sario \cite{AhlforsSario}). The class of bounded harmonic function $u:X\to\mathbb{R}$ is denoted by $HB$. The class of surfaces that do not support a non-constant bounded harmonic function are denoted by $O_{HB}$ (see \cite{AhlforsSario}). Kaimanovich \cite{Kaim} proved that $X\in O_{HB}$ if and only if the horocyclic flow on $X$ is ergodic. Another interesting class of maps are harmonic maps $u:X\to\mathbb{R}$ that have finite Dirichlet integrals $D_X(u)<\infty$, denoted by $HD$. The class of Riemann surfaces without non-constant harmonic function with finite Dirichlet integrals is denoted by $O_{HD}$. The class of bounded harmonic functions with finite Dirichlet integrals is denoted by $HBD$, and the class of Riemann surfaces that does not have non-constant $HBD$-functions is denoted by $O_{HBD}$. It turns out that if $X$ supports a non-constant $HD$-function then it also supports a non-constant $HBD$-function (see \cite[page 203, Theorem 5C.]{AhlforsSario}). We have the following inclusions
$$
O_G\subset O_{HB}\subset O_{HD}=O_{HBD}.
$$
The above inclusions are equalities when $X$ is a planar surface. When $X$ is not planar, Lyons-Sullivan \cite{LyonsSullivan}, McKean-Sullivan \cite{McKS}, Ahlfors \cite{AhlforsSario}, \cite{Ahlfors1} and T\^oki \cite{Toki} showed that the proper inclusions hold by producing examples of the Riemann surfaces. The examples given are infinite Loch-Ness monsters.

The class $O_G$ is invariant under quasiconformal maps \cite{AhlforsSario}. Lyons \cite{Lyons} proved that $O_{HB}$ is not invariant under quasiconformal maps by giving an example of two infinite Riemann surfaces with a Cantor set of ends that are quasiconformal with one in $O_{HB}$ and the other not in $O_{HB}$. 
Sario-Nakai \cite[\S 3]{SarioNakai} proved that the class $O_{HD}$ is invariant under quasiconformal mappings as an application of the (deep) properties of Royden's compactification of Riemann surfaces. 
We prove this fact using a more geometric approach by studying horizontal foliations of quadratic differentials. 

\begin{thm}
\label{thm:O_HD)}
The class $O_{HD}=O_{HBD}$ of Riemann surfaces is invariant under quasiconformal maps.
\end{thm} 

\begin{proof}
Let $X\notin O_{HBD}$ and $f:X\to Y$ be a quasiconformal map. Let $u:X\to\mathbb{R}$ be a non-constant $HBD$-function. Let $\phi =d(u+{\bf i}u^*)$ be an Abelian differential on $X$, where $u^*$ is a local harmonic conjugate of $u$. Since $u$ is a globally defined function, it follows that all the periods of $\phi$ are purely imaginary. The Dirichlet integral $D_X(u)$ is finite because $u$ is an $HBD$-function. Therefore
$$
\int_X |\phi ^2|=D_X(u)<\infty 
$$
and $\phi^2\in A(X)$. The vertical leaves of $\phi^2$ are given by $u(z)=c$ for $c\in\mathbb{R}$; the transverse measure to the vertical foliation is given by $\int_{*} |du|$. Denote by $\mathcal{V}_{\phi^2}$ the vertical foliation of $\phi^2$.

By Theorem \ref{thm:realization-arbitrary}, there exists a finite-area holomorphic quadratic differential $\psi$ on $Y$ such that its vertical measured foliation $\mathcal{V}_{\psi}$ is equivalent to the push-forward measured foliation $f^{*}(\mathcal{V}_{\phi^2})$ of the foliation $\mathcal{V}_{\phi^2}$ by $f$. The leaves of the measured foliation $f^{*}(\mathcal{V}_{\phi^2})$ are globally oriented, which induces a global orientation of the vertical leaves of $\psi$. It follows that the square root of $\psi$ is globally defined Abelian differential $\phi_Y$ on the Riemann surface $Y$. Indeed, if $w$ is a natural parameter for $\psi$ on $Y$, then we have two choices $\pm w$ for the natural parameter. We choose the sign in front of $w$ so that the imaginary part of the choice made increases as we traverse the vertical arc in the positive direction of the orientation. This makes a unique choice of the natural parameter for every point of $Y$, and the transition maps satisfy $w_2=w_1+const$. The differential $dw$ in these local charts defines an Abelian differential $\phi_Y$ on $Y$, which is the square root of $\psi$ in the $w$-coordinate.

Let $\alpha$ be a homotopically non-trivial closed differentiable curve on $Y$ and fix $\epsilon >0$. We replace $\alpha$ by a homotopic step curve for $\phi_Y^2$, denoted by $\alpha$ again, that consists of horizontal and vertical arcs such that 
$$
h_{\phi^2_Y}(\alpha )\geq \int_{\alpha}|Re(d\phi_Y )|-\epsilon .
$$

We need to compute $\int_{\alpha}Re(d\phi_Y)$ in order to show that $\phi_Y$ induces an $HD$-function on $Y$.  
The steps of $\alpha$ are horizontal arcs $h_j$ and vertical arcs $v_k$  for $\phi^2_Y$. Then 
 $$
 \int_{v_k}Re(d\phi_Y) =0
 $$
 and 
 $$
 \int_{h_j}Re(d\phi_Y) =\pm \int_{h_j}|Re(d\phi_Y)|,
 $$
 where the sign is $"+"$ if $Re(d\phi_Y)>0$ along $h_j$ for the orientation induced by the orientation of $\alpha$, and the sign is $"-"$ if $Re(d\phi_Y)<0$ along $h_j$. We have
$$
Re\int_{\alpha}d\phi_Y =\sum_j(\pm \int_{h_j}|Re(d\phi_Y)|).
$$
 
Let $\tilde{\phi}_Y$ be the lift of $\phi_Y$ to the universal covering of $Y$. Let $\tilde{\alpha}$ be a connected closed arc in the universal covering of $Y$ that maps injectively to $\alpha$ when its two endpoints $a_1$ and $a_2$ are identified. Let $\tilde{h}_j$ be the lifts of $h_j$ on $\tilde{\alpha}$, and $\tilde{v}_k$ be the lifts of $v_k$ on $\tilde{\alpha}$. We consider the strips of vertical trajectories between the vertical trajectories $v_{a_1}$ and $v_{a_2}$ through $a_1$ and $a_2$. If $\tilde{v}_k$ is a vertical arc such that it meets two horizontal arcs $\tilde{h}_{j_1}$ and $\tilde{h}_{j_2}$ on the same side, then there is a strip $S_r$ that contains $\tilde{v}_k$ and does not separate $v_{a_1}$ and $v_{a_2}$. Then
$$
\int_{S_r\cap \tilde{h}_{j_1}}Re(d\tilde{\phi}_Y)+\int_{S_r\cap \tilde{h}_{j_2}}Re(d\tilde{\phi}_Y)=0
$$ 
as $S_r\cap \tilde{h}_{j_1}$ and $S_r\cap \tilde{h}_{j_1}$ are homotopic modulo the vertical trajectories of $\tilde{\phi}^2_Y$ and have opposite orientations (see \cite[Figure 1]{Saric-heights}). 
We modify $\tilde{h}_k$ to $\tilde{h}_k':=\tilde{h}_k\setminus\cup_rS_r$. 
The above discussion gives that the integral
$$
\int_{\tilde{\alpha}}Re (d\tilde{\phi}_Y)
$$
is equal to the sum of the integral of $Re (d\tilde{\phi}_Y)$ over all horizontal subarcs $\tilde{h}_k'$ of $\tilde{\alpha}$. 

Let $\tilde{f}:\tilde{X}\to\tilde{Y}$ be the lift of $f:X\to Y$. 
The vertical leaves that intersect $\tilde{h}'_k$ are mapped by $\tilde{f}^{-1}$ to vertical leaves that separate $\tilde{f}^{-1}(v_{a_1})$ and $\tilde{f}^{-1}(v_{a_2})$. Moreover, the set of vertical trajectories that separate $v_{a_1}$ and $v_{a_2}$ is in a bijection with the set of the vertical trajectories that separate $\tilde{f}^{-1}(v_{a_1})$ and $\tilde{f}^{-1}(v_{a_2})$. Let $\tilde{\beta}$ be a step curve for $\phi$ connecting $\tilde{f}^{-1}(v_{a_1})$ and $\tilde{f}^{-1}(v_{a_2})$.

Since $\mathcal{V}_{\phi^2}$ and $\mathcal{V}_{\phi^2_Y}$ correspond under $\tilde{f}$, the transverse measures of any two sets of corresponding vertical trajectories are equal. 
 The integrals over horizontal trajectories are positive if the direction of the horizontal subarc of $\tilde{\beta}$ and the directions of the vertical trajectories subtend an angle $\pi /2$ in the positive direction for the orientation of the universal covering; otherwise, the integral is negative. This property is invariant under $\tilde{f}$ and we conclude that 
$$
\int_{\tilde{\alpha}}Re(d\phi_Y )=\int_{\tilde{\beta}}Re(d\phi )=0.
$$
Therefore, the real part of $d\phi_Y$ gives a globally defined harmonic function $u_Y:Y\to\mathbb{R}$, which is of class $HD$. Thus $Y\notin O_{HD}=O_{HBD}$. The invariance is established.
\end{proof}

\begin{rem}
The above theorem implies that if $X$ supports a non-constant $HB$-function and a quasiconformally equivalent Riemann surface $Y$ does not, then the $HB$-function has an infinite Dirichlet integral.
\end{rem}

\section{The type problem for Riemann surfaces with bounded geometry}
\label{sec:type-graph}

In this section, $X$ stands for an infinite Riemann surface with bounded geometry. By a result of Kinjo \cite{Kinjo}, there exists $M>0$ and a decomposition of the surface $X$ into pairs of pants $\{ P_j\}_j$ and hexagons $\{ H_k\}_k$ such that the cuff lengths and side lengths of the hexagons are between $1/M$ and $M$. We form a graph $\mathcal{G}$ by taking its vertices to be the pairs of pants and hexagons in the fixed decomposition of $X$ as above. Two vertices of $\mathcal{G}$ are connected by an edge if they share a part of their boundaries. Therefore, $\mathcal{G}$ is an infinite graph with bounded valence. 

Indeed, a vertex corresponding to a pair of pants $P_j$ has a degree at most three times the maximum of the attached pieces to each cuff (from the other side). If the attached piece is another pair of pants, the number of attached pieces to a cuff is one. Otherwise, the bounded geometry gives the number of attached hexagons to a cuff at most $M^2$. For a vertex corresponding to a hexagon $H_k$, each side of the hexagon that connects two cuffs gives one edge at the vertex. A side $a_k$ of a hexagon on a cuff $\alpha$ can give one edge if there is a pair of pants on the other side of the cuff $\alpha$ or if only one hexagon on the other side of $\alpha$ contains $a_k$ in its boundary. Otherwise, $a_k$ can be contained in an at most $M^2$ boundary sides of the hexagons on the other side of $\alpha$, and this corresponds to an at most $M^2$ edges at the vertex $H_k$ corresponding to the side $a_k$. Therefore, at most, $3M^2+3$ edges are at the vertex $H_k$.
A vertex of $\mathcal{G}$ can have an edge connecting it to itself (i.e., a loop), which happens when a cuff is on the boundary of a single pair of pants. 
 Denote by $V=V(\mathcal{G})$ and $E=E(\mathcal{G})$ the sets of vertices and edges of $\mathcal{G}$, respectively.

We consider the simple random walk on $\mathcal{G}$ with the probability of moving from one vertex $x$ to its neighbor $y$ (i.e., vertices $x$ and $y$ are connected by an edge) is given by $1/deg(x)$, where $deg(x)$ of the degree of the vertex $x$. For any two vertices $x$ and $y$ of $\mathcal{G}$, let $p(x,y)$ denote the probability of the random walk moving from $x$ to $y$ in one step. Then $p(x,y)=1/deg(x)$ if $x$ and $y$ are connected by an edge, and $p(x,y)=0$ otherwise. Note that $p(x,y)$ may not equal $p(y,x)$ since the degrees of $x$ and $y$ may differ. For $n>0$, denote by $p^{(n)}(x,y)$ the probability that the random walk will move from $x$ to $y$ in $n$ steps. The Green's function $G(x,y)$ of the random walk is defined by (for example, see \cite{Woess})
$$
G(x,y)=\sum_{n>0} p^{(n)}(x,y).
$$
By the definition, the random walk on the graph $\mathcal{G}$ is {\it recurrent} if $G(x,y)=\infty$ for some pair of vertices $x,y$
(which is equivalent to all pairs of vertices). The random walk is {\it transient} if $G(x,y)<\infty$.

We choose an arbitrary orientation of each edge $e\in E$. 
Denote by $e^-$ and $e^+$ the initial and the end point of $e$. In the language of Reversible Markov Chains as in \cite[I.2.A]{Woess}, we have the total conductance $m(x)=deg(x)$, the conductance between $x$ and $y$ to be $a(x,y)=1$ and the resistance $r(e)=1$ along $e$ when $x=e^-$ and $y=e^+$.
Let $\ell^2(V,m)$ and $\ell^2(E)$ be the spaces of $\ell^2$ functions on the sets of vertices and edges of $\mathcal{G}$, where the first space is weighted by the function $m(x)=deg(x)$ and the second space is unweighted. Note that  $\ell^2(V,m)=\ell^2(V)$ because $m(x)>0$ is bounded above, which implies that the norms are quasi-isometric.

Define
$$
\nabla :\ell^2(V,m)\to\ell^2(E)
$$
by
$$
\nabla (f)(e)=f(e^+)-f(e^-).
$$
The adjoint map
$$\nabla^*:\ell^2(E)\to\ell^2(V,m)$$ is given by
$$
\nabla^*(g)(x)=\frac{1}{m(x)}\Big{[}\sum_{e:e^+=x}g(e)-\sum_{e:e^-=x}g(e)\Big{]}. 
$$

By \cite[page 19, Theorem (2.12)]{Woess}, the random walk on $\mathcal{G}$ is transient if and only if there exists a function $u\in \ell^2(E)$ not identically equal to zero such that
\begin{equation}
\label{eq:transient_walk}
\nabla^*u(y)=-\frac{i_0}{m(x)}\delta_x(y),
\end{equation}
where $\delta_x(\cdot )$ is the Dirac measure with support $\{ x\}$ and $i_0$ is a non-zero constant. 

We give a complete characterization when $X\in O_G$ for $X$ with a bounded geometry. The proof could potentially be deduced from the result of Kanai \cite{kanai} by carefully analyzing the relationship between the graph $\mathcal{G}$ and nets. We give an independent and a direct proof in the Appendix.

\begin{thm}
\label{thm:bounded-pants-graph-O_G}
Let $X$ be an infinite Riemann surface with a bounded geometry. Then $X\in O_G$ if and only if the simple random walk on the graph $\mathcal{G}$ is recurrent.
\end{thm}

\subsection{The difference between planar and non-planar surfaces with bounded pants decompositions} We consider $X$ a planar infinite Riemann surface with a bounded pants decomposition. An important observation is that the pants graph $\mathcal{G}$ is a tree since every closed curve on $X$ is separating. 

\begin{thm}
\label{thm:planar-bounded-finite-ends}
Let $X$ be a planar Riemann surface with bounded pants decomposition and finitely many infinite ends. Then $X\in O_G$.
\end{thm}

\begin{proof}
By Theorem \ref{thm:bounded-pants-graph-O_G}, it is enough to show that the pants tree $\mathcal{G}$ is recurrent. The finite (isolated) ends correspond to punctures of $X$, and there are, at most, countably, many such ends. An infinite surface has at least one infinite (non-isolated) end. Given that $X$ has finitely many infinite ends, it follows that $X$ has countably many punctures that accumulate to all of its finitely many infinite ends.

We first note that the space of infinite ends of $X$ (i.e., the ends different than punctures of $X$) is homeomorphic to the Gromov boundary of $\mathcal{G}$ when $\mathcal{G}$ is considered as a metric space where all the edges have lengths $1$. Indeed, an end of $X$ is a nested sequence of components $\{ U_n\}_n$ of the complements of a compact exhaustion $\{ K_n\}_n$ of $X$. An infinite end is not a puncture; therefore, each $U_n$ contains infinitely many punctures. Since the pants decomposition is locally finite, only finitely many cuffs intersect $K_n$. Let $\alpha_n$ be a cuff in $U_n$. Then $\alpha_n$ divides $U_n$ into two components; one component contains $U_k\subset U_n$ for all $k\geq k_0(n)$, and the other component is a finite surface with punctures. If $\alpha_n$ does not separate $K_n\cup\alpha_n$ from the infinite end, then  
we replace $\alpha_n$ with a cuff $\beta_n\subset U_n$ that separates $K_n$ from the end of $X$ corresponding to $\{ U_n\}$. Starting from this $\beta_n$, there is a sequence of adjacent pairs of pants in $U_n$ going to infinity. This sequence gives a sequence of adjacent vertices in $\mathcal{G}$, which defines a point in the Gromov boundary of $\mathcal{G}$. A finite tree is attached to each vertex of this sequence because the complement of $\alpha_n$ has a finite topology. Therefore, to each infinite end of $X$, there corresponds an infinite subtree with one Gromov boundary point obtained by an infinite path of edges with attached finite trees at each vertex. The sizes of the attached finite trees may go to infinity.

Since the Gromov boundary of $\mathcal{G}$ consists of finitely many points, a unique infinite edge path (without backtracking) connects the fixed vertex $x$ to each end of $\mathcal{G}$. The complement of the union of these infinite edge paths has connected components consisting of finitely many edges. If $u:E\to\mathbb{R}$ satisfies $\nabla^*u(y)=0$ for all $y\neq x$, it follows that $u$ must be zero on each edge that has a vertex of degree one. Then $u$ is zero on all edges in each component of the complement of the infinite paths connecting $x$ to the ends of $\mathcal{G}$. 

Next, consider a single infinite path from $x$ to an end of $\mathcal{G}$. After finitely many edges away from $x$, the path does not share an edge with other infinite paths. In this case, we may assume that this path has only vertices of degree two (since $u$ is zero on the components of the complement). Thus, if $u$ is non-zero on a single edge, then it is equal to the same value (up to sign change $\pm$) on the rest of the path going to infinity. That contradicts $u\in\ell^2(E)$ unless $u\equiv 0$. This is true for all infinite paths. Theorem \ref{thm:bounded-pants-graph-O_G} implies $X\in O_G$.
\end{proof}

The above Theorem implies that a flute surface (i.e., an infinite planar surface with infinitely many punctures and one topological end) with a bounded pants decomposition (with an arbitrary topological configuration of the pants decomposition) is always in $O_G$. We also point out that a flute surface with bounded geometry is not necessarily in $O_G$ by an example of Kinjo (see \cite{Kinjo2010}).

Even when a surface has a bounded pants decomposition, but it is not planar, the type of the surface depends on the topological configuration of the pants decomposition. 
To see this, consider the infinite Loch Ness monster surface $N$, which has an infinite genus and only one topological end. The surface $N$ has many pants decompositions, and, most importantly, the pants graph is not a tree. For example, the infinite Loch Ness monster surface $N$ can be obtained by taking the surface of the $\epsilon$-neighborhood $\mathbb{Z}^2$-grid in $\mathbb{R}^3$, where $\mathbb{Z}^2$-grid is realized in the $xy$-plane. We form a bounded pants decomposition of $N$ as follows. Consider the geodesic representatives of all the isotopy classes of curves, which are the circles on $N$ (in the Euclidean metric of $\mathbb{R}^3$) with centers the midpoints of the edges of the $\mathbb{Z}^2$-grid. The complementary regions are homeomorphic to a sphere minus four disks, and the isometry group of $N$ maps any complementary region to any other complementary region. We choose a simple closed geodesic in one complementary region to make a pants decomposition of the region and propagate this simple closed geodesic to all other complementary regions by elements of the isometry group. In this fashion, we obtain a bounded pants decomposition of $N$. It is well-known that the (simple) random walk on the $\mathbb{Z}^2$-grid is recurrent (\cite{Woess}). By Theorem \ref{thm:bounded-pants-graph-O_G}, the geodesic flow on the Riemann surface $N$ obtained from the above construction is ergodic if and only if the simple random walk on the pants graph $\mathcal{G}$ is recurrent. The pants graph $\mathcal{G}$ is obtained from the $\mathbb{Z}^2$-grid by introducing two vertices for each element of $\mathbb{Z}^2$ (i.e., vertex of the $\mathbb{Z}^2$-grid), of the four edges meeting each old vertex we distribute two edges to each new vertex and we connect the two new vertices by an extra edge (see Figure 12).

\begin{figure}[h]
\leavevmode \SetLabels
\endSetLabels
\begin{center}
\AffixLabels{\centerline{\epsfig{file =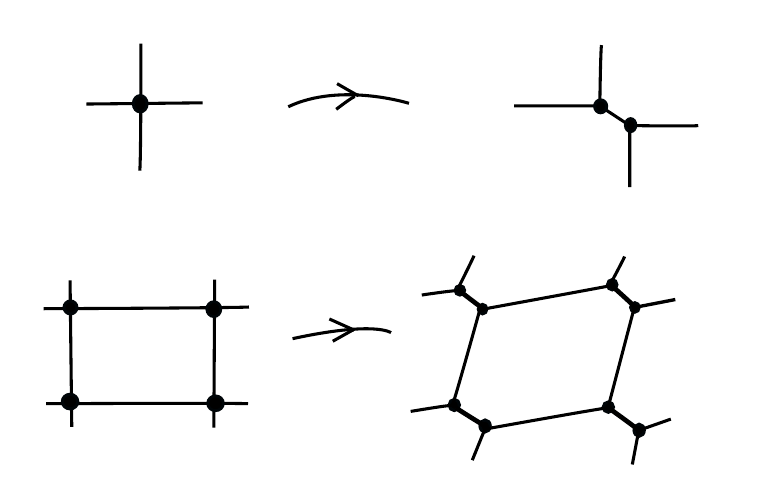,bb= 0 0 390 220,width=10.0cm,angle=0} }}
\vspace{-20pt}
\end{center}
\caption{Replacing a vertex of the $\mathbb{Z}^2$-grid with two vertices and an edge to obtain the pants graph $\mathcal{G}$ of the infinite Loch Ness monster.} 
\end{figure}

We claim the random walk on $\mathcal{G}$ is recurrent. Assume, on the contrary, that the random walk is not recurrent. Then there exists a non-trivial function $u\in \ell^2(E)$ and a vertex $x\in V$ such that $\nabla^*u(y)=\delta_x(y)$ for all $y\in V$ (see \cite[Theorem 2.12]{Woess}). Starting from $u\in \ell^2(E)$, we construct a function $\bar{u}\in\ell^2(\mathbb{Z}^2)$ which satisfies $\nabla^*\bar{u}(\bar{y})=\delta_{\bar{x}}(\bar{y})$ for a fixed $\bar{x}\in \mathbb{Z}^2$ and all $\bar{y}\in \mathbb{Z}^2$.
If $\bar{e}$ is the edge of the $\mathbb{Z}^2$, then there is a unique edge $e$ of $\mathcal{G}$ corresponding to it. We set $\bar{u}(\bar{e}):=u(e)$. It is immediate that $\bar{u}\in\ell^2(\mathbb{Z}^2)$. Let $\bar{x}\in\mathbb{Z}^2$ corresponds to two vertices of $\mathcal{G}$ one of which is $x$. There are several different possible orientations of edges of $\mathcal{G}$ at the vertex $x$, and we will consider the one possible orientation as in Figure 13. Then we have $\nabla^*u(y)=u(e_1)+u(e_2)-u(e)=\delta_x(y)=0$ which implies $u(e)=u(e_1)+u(e_2)$. In addition, we have $1=\nabla^*u(x)=u(e_1)+u(e_2)+u(e)=u(e_1)+u(e_2)+u(e_3)+u(e_4)=\bar{u}(\bar{e}_1)+\bar{u}(\bar{e}_2)+\bar{u}(\bar{e}_3)+\bar{u}(\bar{e}_4)=\nabla^*\bar{u}(\bar{x})$. In an analogous fashion we obtain $\nabla^*\bar{u}(\bar{y})=0$ for $\bar{y}\neq\bar{x}$. This contradicts the fact that the random walk on the $\mathbb{Z}^2$-grid is recurrent. Therefore the random walk on $\mathcal{G}$ is recurrent and the geodesic flow on $N$ is ergodic.

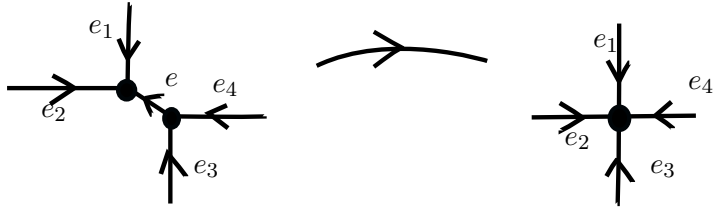
\begin{figure}[h]
\leavevmode \SetLabels
\L(.2*.7) $e_1$\\
\L(.15*.4) $e_2$\\
\L(.31*.2) $e_3$\\
\L(.33*.49) $e_4$\\
\L(.28*.51) $e$\\
\L(.73*.65) $e_1$\\
\L(.7*.3) $e_2$\\
\L(.79*.2) $e_3$\\
\L(.83*.5) $e_4$\\
\endSetLabels
\begin{center}
\AffixLabels{\centerline{\epsfig{file =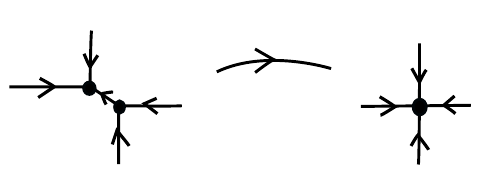,bb= 0 0 242 90,width=10.0cm,angle=0} }}
\vspace{-20pt}
\end{center}
\caption{From $\mathcal{G}$ to the $\mathbb{Z}^2$-grid.} 
\end{figure}

Let $N_1$ be the infinite Loch Ness monster surface obtained by taking the surface of the $\epsilon$-neighborhood of the $\mathbb{Z}^3$-grid in $\mathbb{R}^3$. Analogous to the construction of $N$, the family of simple closed geodesics corresponding to the edges of the $\mathbb{Z}^3$-grid decomposes $N_1$ into components homeomorphic to the sphere minus six disks, each corresponding to the element of $\mathbb{Z}^3$. Each component is decomposed into three pairs of pants by adding two simple closed geodesics that are disjoint from each other. We choose those geodesics to be propagated by elements of the isometry group of $N_1$ so that the obtained pants decomposition of $N_1$ has cuff lengths bounded between two positive constants. The pants graph $\mathcal{G}$ is obtained from the $\mathbb{Z}^3$-grid by replacing each vertex with four vertices and adding edges, as in Figure 14. The pants graph is more complicated than the construction with the $\mathbb{Z}^2$-grid.

\begin{figure}[h]
\leavevmode \SetLabels
\L(.23*.85) $\bar{e}_1$\\
\L(.31*.8) $\bar{e}_2$\\
\L(.32*.2) $\bar{e}_3$\\
\L(.23*.1) $\bar{e}_4$\\
\L(.15*.2) $\bar{e}_5$\\
\L(.13*.8) $\bar{e}_6$\\
\L(.45*.17) $e_5$\\
\L(.45*.63) $e_6$\\
\L(.55*.5) $f$\\
\L(.6*.11) $e_4$\\
\L(.65*.5) $g$\\
\L(.74*.14) $e_3$\\
\L(.72*.5) $h$\\
\L(.84*.6) $e_2$\\
\L(.77*.68) $e_1$\\
\L(.51*.31) $x_1$\\
\L(.61*.31) $x_2$\\
\L(.67*.31) $x_3$\\
\L(.77*.31) $x_4$\\
\endSetLabels
\begin{center}
\AffixLabels{\centerline{\epsfig{file =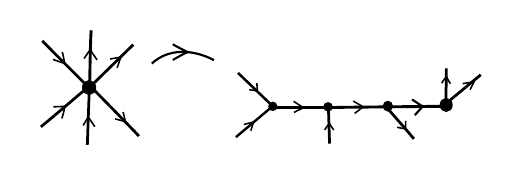,bb= 0 0 247 90,width=10.0cm,angle=0} }}
\vspace{-20pt}
\end{center}
\caption{From the $\mathbb{Z}^3$-grid to the pants graph $\mathcal{G}$ of $N_1$.} 
\end{figure}
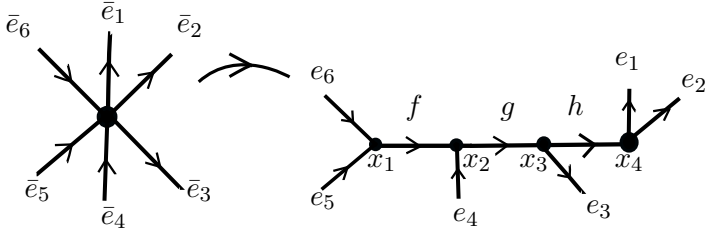

It is a standard fact that the random walk on the $\mathbb{Z}^3$-grid is transient. Therefore, there exists a function $\bar{u}:\mathbb{Z}^3\to\mathbb{R}$ with $\bar{u}\in\ell^2(\mathbb{Z}^3)$ and $\nabla^*\bar{u}(\bar{y})=\delta_{\bar{x}}(\bar{y})$ for some fixed $\bar{x}\in\mathbb{Z}^3$ and all $\bar{y}\in\mathbb{Z}^3$. We construct a function $u:E\to\mathbb{R}$ that has the same properties as $\bar{u}$, where $E$ is the edge set of $\mathcal{G}$. This will imply that the pants graph $\mathcal{G}$ is transient, which implies that $N_1\notin O_G$. Mori \cite{Mori} and Rees \cite{Rees} (for higher order groups) obtained this result using different methods. 

Let $\bar{y}$ be a vertex of the $\mathbb{Z}^3$-grid with edges $\bar{e}_i$ for $1\leq i\leq 6$. We assume that they are oriented, as in Figure 14. We replace the configuration of the $\mathbb{Z}^3$-grid at the vertex $\bar{y}$ with the configuration on the right of Figure 14 and repeat this procedure at each vertex. Denote by $e_i$ the new edges on the right of Figure 14 that correspond to $\bar{e}_i$. Define $u(e_i):=\bar{u}(\bar{e}_i)$. Further, let
\begin{equation*}
\begin{array}l
u(f):=\bar{u}(e_6)+\bar{u}(e_5),\\
u(g):=u(f)+\bar{u}(e_4)=\bar{u}(e_6)+\bar{u}(e_5)+\bar{u}(e_4),\\
u(h):={u}(g)-\bar{u}(e_3)=\bar{u}(e_6)+\bar{u}(e_5)+\bar{u}(e_4)-\bar{u}(e_3).
\end{array}
\end{equation*}

Using the above definitions, we get
\begin{equation*}
\begin{array}l
\nabla^*u(x_1)=0,\\
\nabla^*u(x_2)=-u(g)+u(f)+u(e_4)=0,\\
\nabla^*u(x_3)=u(g)-u(e_1)-u(h)=0\ \mathrm{and}\\
\nabla^*u(x_4)=u(h)-u(e_1)-u(e_2)=\\
\hskip 1.5 cm -u(e_1)-u(e_2)-u(e_3)+u(e_4)+u(e_5)+u(e_6)=\delta_x(y).
\end{array}
\end{equation*}
Therefore $X\notin O_G$.

We also prove that planar surfaces with bounded pants decompositions and countably many ends are parabolic as well.

\begin{thm}
\label{thm:countable-ends-planar}
Let $X$ be an infinite planar surface with a bounded pants decomposition and countably many ends. Then $X\in O_G$.
\end{thm}

\begin{proof}
Consider the graph $\mathcal{G}$, which is the dual graph to the pants graph of $X$ for the fixed bounded pants decomposition. By Theorem \ref{thm:bounded-pants-graph-O_G}, it is enough to prove that the random walk on $\mathcal{G}$ is recurrent. Note that the graph $\mathcal{G}$ is a tree since $X$ is planar.

By the transfinite induction, the space of ends $\mathcal{E}(X)$ of $X$ will become empty after applying countably many Cantor-Benedixson derivatives. Each puncture of $X$ is an isolated end of $X$ corresponding to either a bivalent or an univalent vertex of $\mathcal{G}$.

Taking the first derivative $\mathcal{E}(X)'$ on the space of ends of $X$, we obtain all infinite ends, i.e.-ends accumulated by other ends. Consider an isolated end ${\bf e}$ in $\mathcal{E}(X)'$. There exists a compact set $K_{\bf e}\subset X$ such that a single component of its complement $K_{\bf e}^C$ only approaches the end ${\bf e}$. This single component necessarily contains infinitely many cusps, and an infinite sequence of consecutive pairs of pants goes off to ${\bf e}$. This corresponds to a geodesic in $\mathcal{G}$, that is, an infinite edge path without backtracking that converges to a unique point on the Gromov boundary of $\mathcal{G}$. In addition to the edges of this path, at each vertex, there is attached, at most, a finite graph whose sizes may go to infinity. In particular, there are infinitely many cuffs on $X$, each separating $X$ into two parts, one of which approaches ${\bf e}$ and no other infinite ends of $X$. Among these cuffs, there is a cuff corresponding to the first edge of the above path in $\mathcal{G}$, and we truncate $X$ at this cuff, discarding the part approaching the end ${\bf e}$. We perform this truncation for each isolated end of $\mathcal{E}(X)'$ and denote the truncated surface by $X_1$. Note that the isolated ends of $X_1$ are the cuffs where we truncated and the remaining punctures. The graph dual to the pants graph of $X_1$ is connected by the edges corresponding to the cuffs where we performed the truncations to the dual pants graphs of the truncated parts of $X$.  

We perform the same process on $X_1$ for the isolated ends of $\mathcal{E}(X_1)'$. In this fashion, we obtain an edge path without backtracking in the dual graph of the pants graph of $X_1$ corresponding to the isolated ends of $\mathcal{E}(X_1)'$. To the vertices of each of these edge paths, we add finite trees if the attached pieces (to the cuffs) are finite surfaces, or we attach infinite paths corresponding to the isolated points of $\mathcal{E}(X)'$ from the previous step. We continue this process of taking the Cantor-Benedixson derivatives of $\mathcal{E}(X)$ until we arrive at an empty set. We conclude that to each infinite end in $\mathcal{E}(X)$, there corresponds an infinite path without backtracking in $\mathcal{G}$. For any two different ends, the corresponding paths are not bounded distance away in $\mathcal{G}$; therefore, they correspond to distinct points in the Gromov boundary. 

There is a hierarchy on these paths of $\mathcal{G}$ corresponding to infinite ends of $X$. First, there are paths corresponding to infinite ends accumulated only by the punctures, followed by the paths corresponding to infinite ends accumulated by the ends of the first level, and so on.
Assume now that $u:E\to\mathbb{R}$ satisfies $u\in\ell^2(E)$ and $\nabla^*u(y)=\delta_x(y)$. To see that $X\in O_G$, it is enough to prove that $u(y)\equiv 0$. Let $\{ e_i\}_{i=1}^{\infty}$ be an infinite edge path in $\mathcal{G}$ at the first level in the hierarchy. Let $x_1$ be the vertex of $e_1$ not shared by $e_2$. Consider the edge path 
$\{ e_i\}_{i=1}^{\infty}$ and all edges of $\mathcal{G}$ to this path except the ones that are connected through the vertex $x_1$. This subtree of $\mathcal{G}$ consists of $\{ e_i\}_{i=1}^{\infty}$ to which we add finite trees attached to its vertices except for the first vertex $x_1$. This subtree is isomorphic to a tree arising from the dual graph of the pants graph of a flute surface. 
When we restrict $u$ to the subtree that contains the edge path 
$\{ e_i\}_{i=1}^{\infty}$, we notice that the restricted function $u_1$ satisfies $\nabla^*u_1(y)=0$ if $y\neq x_1$ because this is inherited from $u$ and that $\nabla^*u_1(x_1)\neq 0$ because we erased at least one edge at $x_1$. Then the proof of Theorem \ref{thm:planar-bounded-finite-ends} gives that $u_1$ is constantly zero on these edges. We establish this for all paths in the first level of the hierarchy. The argument is then repeated on the paths of the second level since we can truncate the first level path (the function $u$ is constantly zero there). Indeed, after truncating, the second level path and the attached edges are again exactly as in the first level. After repeating this process, we conclude that $u$ is zero on all edge paths corresponding to the ends. The remaining edge set is finite, and we conclude that $u\equiv 0$. Thus $X\in O_G$.
\end{proof}

\subsection{Non-parabolic flute surfaces with bounded geometry} \label{subsec:nonPflutes}
Recall that an infinite flute Riemann surface is a planar surface with countably many punctures that accumulate to a single non-isolated topological end. Therefore, a flute Riemann surface is conformal to either a complex plane $\mathbb{C}$ minus a discrete, countable subset $A$ or to a disk $\mathbb{D}$ minus a discrete, countable subset $A$ of the disk. In the former case, the surface $X$ is parabolic, and in the latter case, the surface $X$ is not parabolic. When $X=\mathbb{D}\setminus A$, there are two possibilities: $A$ accumulates to all points of the unit circle $S^1$, or $A$ does not accumulate to an open arc of $S^1$. In the former case, the covering group of $X$ is of the first kind (see \cite{BS}). 

Sullivan posed a problem of finding examples of $\mathbb{D}\setminus A$ such that $A$ accumulates to all points of $S^1$ in terms of the hyperbolic invariants. We give a description of such surfaces in terms of the shear coordinates to an ideal triangulation decomposition of the surface. 

We start with a countable planar graph with bounded valence and one topological end whose simple random walk is transient. Such planar graphs are constructed by Geyer and Merenkov \cite{GM} (see also \cite{BMS}). In a private communication, Merenkov suggested using the graph obtained by embedding a trivalent tree in the plane and attaching a half-plane square grid to each complementary component. Further, each complementary component can be triangulated by adding finitely many edges between the vertices of the component. The obtained graph called $\mathcal{H}$, still has bounded valence and one topological end, and it is transient.

We draw a small circle around each vertex so that the circles over all vertices are disjoint. The complementary part of each triangle with respect to the union of the disks is topologically a hexagon, and it will correspond to a right-angled hexagon on the constructed surface. We change the graph $\mathcal{H}$ by adding a vertex inside each triangle and connecting this vertex by edges to the vertices of the triangle and to the new vertices of the adjacent triangles. We erase the old edges of the triangle. The new graph $\mathcal{G}$ is also transient because we can construct a non-zero potential function leaving to infinity by the methods above.

We form the surface $X$ whose decomposition into pairs of pants and right-angled hexagons corresponds to the graph $\mathcal{G}$. First, we realize the surface topologically in the plane. Indeed, to each vertex of $\mathcal{H}$, we take a pair of pants with one cuff and two punctures, which is realized by taking a small circle around the vertex and two punctures inside the disk bounded by the circle. To an edge of $\mathbb{H}$ connecting two vertices $v_1$ and $v_2$, we associate an ideal geodesic connecting one puncture in the pair of pants corresponding to $v_1$ to one puncture in the pair of pants corresponding to $v_2$. We arrange for all geodesics that we have drawn to be disjoint in the plane. Finally, all punctures in the pairs of pants are connected by ideal geodesics, and then the rest of the surface is triangulated by only making the punctures to be the vertices of the triangulation $\mathcal{T}$.

We form the flute surface $X$ by taking an ideal geodesic triangle for each triangle of $\mathcal{T}$ and gluing adjacent ideal triangles by zero shears (for the definition of shears, see, for example, \cite{saric}). The surface $X$ has the covering Fuchsian group of the first kind since zero shears guarantee that the lift of the ideal geodesic triangulation to $\mathbb{H}$ is the Farey triangulation (see also \cite[Theorem C]{saric}). 

We introduce cuffs (simple closed geodesics) corresponding to the small circles around the vertices of $\mathcal{H}$. Each cuff separates the two punctures from the rest of the surface and forms the boundary of a geodesic pair of pants with two punctures. We partition the complement of the union of pairs of pants into right-angled hexagons.
If two cuffs are connected by ideal geodesics from the triangulation, then we draw a unique common orthogonal between the two cuffs. These orthogonals are pairwise disjoint and divide $X$ into right-angled hexagons since $\mathcal{T}$ is a triangulation. By choice of zero shears, the lengths of cuffs and the lengths of the common orthogonal to the cuffs depend on the valence at the vertex of $\mathcal{H}$ and the combinatorics of the ideal triangles that we have drawn to define $X$. Since the valence is bounded and the number of combinatorial possibilities is finite, the surface $X$ has bounded geometry. Finally, the graph dual to the decomposition of $X$ into pairs of pants and right-angled hexagons with sides bounded away from $0$ and $\infty$ is isomorphic to $\mathcal{G}$. Since $\mathcal{G}$ is transient, we conclude by Theorem \ref{thm:bounded-pants-graph-O_G} that $X\notin O_G$. We obtained

\begin{thm}
\label{thm:nonparabolic-bdd-geom}
There exists a bounded geometry flute surface $X$ that is not parabolic. The surface $X$ can be chosen to have a locally finite ideal geodesic triangulation with zero shears.
\end{thm}

\section{Appendix}

\noindent {\it Proof of Theorem 7.1}
Assume the random walk on $\mathcal{G}$ is transient. Let $u\in\ell^2(E)$ satisfy (\ref{eq:transient_walk}) for $i_0=-m(x)$. The function $u$ induces a new orientation of the edges $E$ by keeping the old orientation on $e\in E$ if $u(e)>0$ and by reversing the orientation of $e\in E$ if $u(e)<0$. An edge $e$ with an endpoint $x\in V$ is said to be {\it incoming at} $x$ if $e^+=x$, it is {\it outgoing at} $x$ if $e^-=x$. After changing the orientation on the edges, the function $u(e)$ is changed to be the absolute value of the old value at $e$, and we keep the same notation $u(e)$ for the new function. The new function $u\in\ell^2(E)$ is non-negative and it satisfies $\nabla^*u(y)=\delta_x(y)$.  This condition implies that the sum of the values of $u$ on the incoming edges equals the sum of the values of $u$ on the outgoing edges at each vertex of $\mathcal{G}$ except at the vertex $x$. At the vertex $x$, the difference between the sum of the values of $u$ on the incoming edges and the sum of the values of $u$ on the outgoing edges is one. We use this property of $u$ to construct a partial measured foliation on $X$ with finite Dirichlet integral and transient trajectories. By \cite[Theorem 1.1]{Saric23}, this will imply that $X\notin O_G$. 

The vertex $x$ with the above property is arbitrary, and we choose it to correspond to a pair of pants  $P_x$. As in the section above, we form a collection of annuli around the cuffs and rectangles around the edges of the hexagons not lying on the cuffs (i.e., connecting two cuffs).  Further, we form rectangles in the pairs of pants connecting pairs of distinct cuffs depending on the function $u$ (or, more precisely, depending on the orientation of the edges of $\mathcal{G}$).

Consider a pair of pants $P_y$ different from $P_x$, where $y\in V$. Let $\alpha_1$, $\alpha_2$, and $\alpha_3$ be the cuffs of $P$. Each cuff can be a boundary of another pair of pants, or one cuff is equal to another (which happens when the closure of $P$ in $X$ is a torus with a geodesic boundary), or a cuff meets finitely many hexagons from the other side.

If all three cuffs $\alpha_1$, $\alpha_2$, and $\alpha_3$ meet pairs of pants on the other side, then the vertex $y$ is trivalent in $\mathcal{G}$. Two edges are incoming, and one is outgoing, or one edge is incoming, and two are outgoing. Assume we are in the former case, and that the edges $e_1$ and $e_2$ corresponding to the cuffs $\alpha_1$ and $\alpha_2$ are incoming and the edge $e_3$ corresponding to the cuff $\alpha_3$ is outgoing. Then we have $u(e_1)+u(e_2)=u(e_3)$. We form two oriented rectangles that start at the annuli around $\alpha_1$ and $\alpha_2$ end on the annulus around $\alpha_3$. When two edges are outgoing and one incoming, the choice of the oriented rectangles is analogous. 

If $\alpha_1=\alpha_2$, and if $\alpha_3$ meets another pair of pants on the other side, then the vertex $y$ has a loop $e$ (corresponding to $\alpha_1=\alpha_2$) and an edge $e_3$ corresponding to $\alpha_3$. The loop $e$ is both incoming and outgoing at $y$, and the edge $e_3$ is either outgoing or incoming at $y$. This implies that $u(e)+u(e_3)=u(e)$ which gives $u(e_3)=0$. Therefore, we do not need a rectangle representing $e_3$. We can also set $u(e)=0$ because $e$ does not connect to the other edges of $\mathcal{G}$ and (\ref{eq:transient_walk}) remains valid after this modification. 

When at least one of the cuffs of $P_y$ meets hexagons on the other side, then for each such cuff, we consider the sum of the values of the function on the edges of $\mathcal{G}$ that correspond to the sides of the hexagons meeting the cuff. At the vertex $y$ in the graph $\mathcal{G}$, edges are divided into three families, each corresponding to a cuff of $P_y$ (see Figure 15). If the sum of the values of $u$ on the edges corresponding to a cuff is positive, then we will treat that as a single incoming edge at $y$, and if the sum is negative, then we will treat this family as a single outgoing edge at $y$. The total sum of all values of $u$ over the three families is zero. Then, we add two rectangles in $P_y$, connecting two pairs of cuffs $\alpha_1$ and $\alpha_2$ to the third cuff $\alpha_3$ if the families of edges corresponding to $\alpha_1$ and $\alpha_2$ are both outgoing or both incoming. The other possibilities are treated accordingly, and the rectangles are appropriately oriented.

\begin{figure}[h]
\leavevmode \SetLabels
\L(.32*.13) $e_1^{\alpha_1}$\\
\L(.45*.1) $e_2^{\alpha_1}$\\
\L(.55*.12) $e_3^{\alpha_1}$\\
\L(.31*.78) $e_1^{\alpha_2}$\\
\L(.4*.79) $e_2^{\alpha_2}$\\
\L(.6*.82) $e_3^{\alpha_3}$\\
\L(.637*.7) $e_2^{\alpha_3}$\\
\L(.69*.57) $e_1^{\alpha_3}$\\
\endSetLabels
\begin{center}
\AffixLabels{\centerline{\epsfig{file =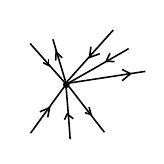,bb= 0 0 83 80,width=6.0cm,angle=0} }}
\vspace{-20pt}
\end{center}
\label{fig12}
\caption{The partition of the edges of $\mathcal{G}$ at a vertex $y$ corresponding to a pair of pants with cuffs $\alpha_1$, $\alpha_2$ and $\alpha_3$.} 
\end{figure}
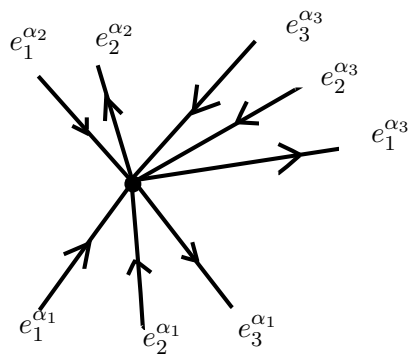

Assume that a vertex $y$ of the graph $\mathcal{G}$ corresponds to a hexagon $H$. There are at most $3M^2+3$ edges at $y$. The sum of the values of $u$ on the incoming edges at $y$ equals the sum of the values of $u$ of the outgoing edges at $y$. We form a partial foliation corresponding to this configuration supported in the union of the rectangles corresponding to the sides of the hexagon $H$ that do not lie on the cuffs and the annuli around the cuffs. We realize the vertex $y$ of $\mathcal{G}$ as a point in $H$ and the edges at $y$ as geodesic arcs $e_i$ from $y$ that cross the corresponding sides of $H$ (see Figure 16). The geodesic arcs are given orientation such that they finish at $y$ or start at $y$ to agree with the orientation of the corresponding edges in $\mathcal{G}$. When two adjacent arcs $e_i$ and $e_{i+1}$ have opposite orientations, they make an oriented arc that can be homotoped modulo its endpoints to the boundary of the hexagon. If $u(e_i)<u(e_{i+1})$ then erase $e_i$ and replace it with a thin rectangle (with an appropriate orientation) going from the endpoint of $e_i$ to the endpoint of $e_{i+1}$ inside the prescribed neighborhood of the identity. When $u(e_i)=u(e_{i+1})$ then we erase both arcs. By continuing this process, after finitely many steps, we erase all geodesic arcs representing edges at $y$ and replace them with rectangles inside the neighborhood of the boundary of the hexagon $H$. The rectangles start and end at the place where the edges are replaced. They are oriented and have assigned weights from the values of the (positive) function $u$ (see Figure 16).

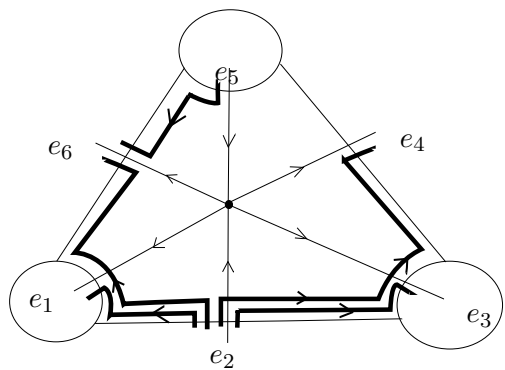
\begin{figure}[h]
\leavevmode \SetLabels
\L(.26*.18) $e_1$\\
\L(.45*.03) $e_2$\\
\L(.72*.14) $e_3$\\
\L(.65*.61) $e_4$\\
\L(.455*.79) $e_5$\\
\L(.28*.59) $e_6$\\
\endSetLabels
\begin{center}
\AffixLabels{\centerline{\epsfig{file =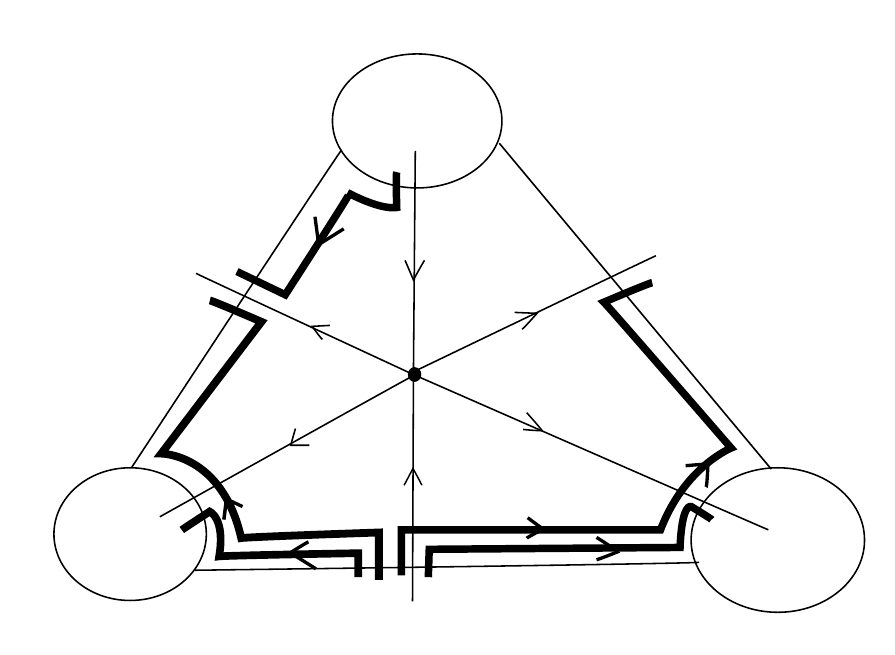,bb= -20 0 455 295,width=8.0cm,angle=0} }}
\vspace{-20pt}
\end{center}
\label{fig13}
\caption{The homotopy of the edges of $\mathcal{G}$ at a vertex $y$ corresponding to a hexagon $H$. The thick lines are rectangles obtained by the homotopy.} 
\end{figure}

Adding the rectangles in the pairs of pants and hexagons above produces a family of oriented rectangles with positive weights. A cuff meets several rectangles from both sides with different orientations. If a cuff $\alpha$ on the boundary of a pair of pants, then $\alpha$ meets either one or two rectangles from that pair of pants. If a cuff is on the boundary of a planar part, then $\alpha$ meets rectangles defined as neighborhoods of the sides of the hexagons that are orthogonal to $\alpha$. The cuff $\alpha$ is on the boundary of two pairs of pants or a pair of pants and a planar part or on the boundary of two planar parts. 

\begin{figure}[h]
\leavevmode \SetLabels
\L(.33*.3) $\alpha$\\
\endSetLabels
\begin{center}
\AffixLabels{\centerline{\epsfig{file =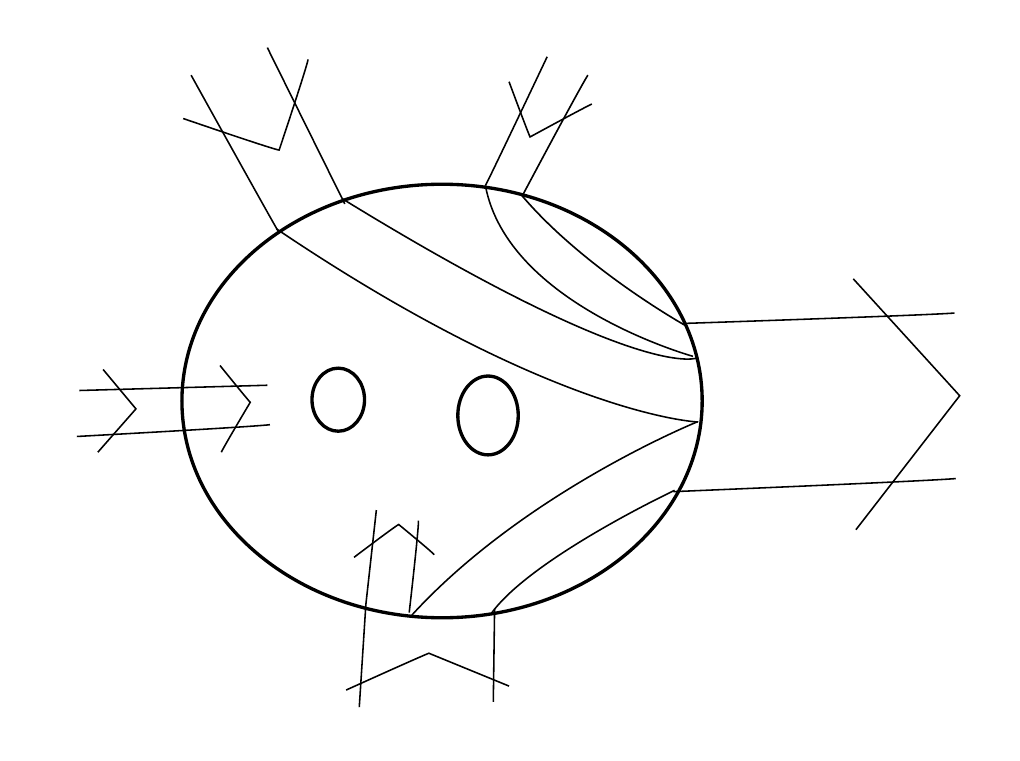,bb= 0 0 390 390,width=6.0cm,angle=0} }}
\vspace{-20pt}
\end{center}
\label{fig13}
\caption{The rectangles with opposite orientations on a cuff $\alpha$.} 
\end{figure}

When $\alpha$ is on the boundary of a planar part, the edges of $\mathcal{G}$ that correspond to $\alpha$ can have different orientations. We connect (inside the annulus around $\alpha$) the neighboring rectangles corresponding to the edges and assign the width given by a smaller of the two weights. Then, the smaller weight decreases the rectangle with the larger weight and crosses the cuff $\alpha$, and the rectangle with the smaller weight does not cross $\alpha$ (see Figure 17). We repeat this process until only rectangles with the same orientation remain. This process is repeated on both sides of all cuffs. The total $u$ value of the rectangles on one side that cross $\alpha$ equals the total $u$ value of the rectangles from the other side that cross $\alpha$. We connect the rectangles from one side to the rectangles from the other side inside the annulus around $\alpha$ so that the weights on the connected rectangles agree at their connections with the $u$ values and that the rectangles are mutually disjoint (see Figure 18).

\begin{figure}[h]
\leavevmode \SetLabels
\endSetLabels
\begin{center}
\AffixLabels{\centerline{\epsfig{file =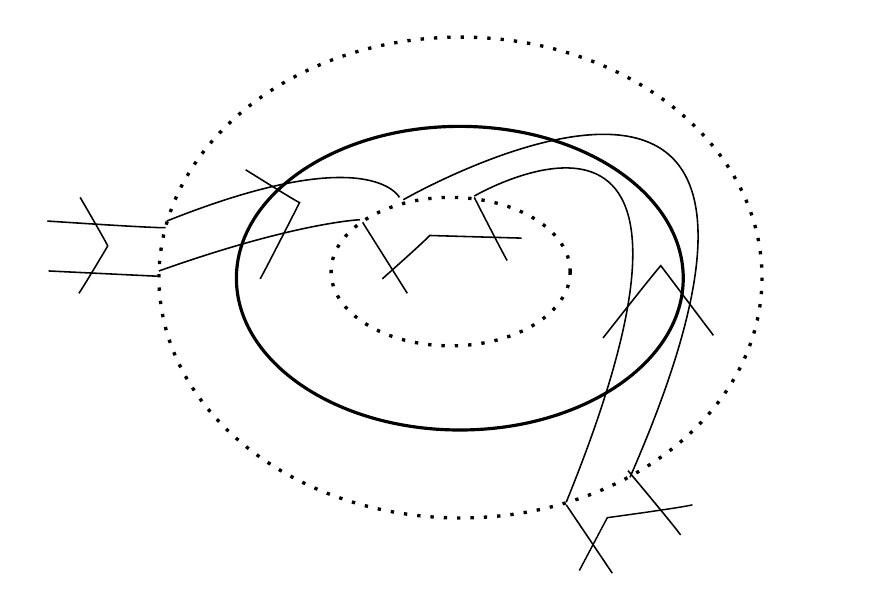,bb= 0 0 390 390,width=6.0cm,angle=0} }}
\vspace{-20pt}
\end{center}
\label{fig13}
\caption{Connecting the rectangles from the opposite side of $\alpha$.} 
\end{figure}

Therefore, we obtained a set of oriented rectangles that connect while respecting the orientation. The weights on the rectangles are obtained from the function $u:\mathcal{G}\to\mathbb{R}$. The weight on each rectangle is a linear function of the $u$ values of a bounded number of edges of $\mathcal{G}$ that are neighbors of either the edge to which the rectangle corresponds or edges adjacent to a cuff whose annulus contains the rectangle. Moreover, the coefficients of the linear function are $\pm 1$, and therefore, the weights are $\ell^2$ functions on the set of rectangles. We introduce a measured partial foliation $\mathcal{F}_u$ by foliating the rectangles and assigning the transverse measure given by the weight. Let $\mu$ be the measured geodesic lamination homotopic to $\mathcal{F}_u$. Note that 
$$
\sum_n [i(\alpha_n,\mu )^2+i(\beta_n,\mu )^2]+\sum_s i(\beta_s,\mu )^2 <\infty
$$
because the number of rectangles that each of $\alpha_n$, $\beta_n$ and $\beta_s$ can intersect is bounded.
By Theorem \ref{thm:holq-ml-bdd}, the corresponding partial foliation $\mathcal{F}_u$ has a finite Dirichlet integral. Then, there exists a unique finite-area holomorphic quadratic differential $\varphi_u$ whose horizontal foliation is equivalent to $\mathcal{F}_u$ (see Theorem \ref{thm:main} and \cite[Theorem 5]{Saric23}). 

To prove that $X\notin O_G$, it is enough to prove that the set of transient (i.e., leaving every compact subset) horizontal leaves of $\varphi_u$ has a non-zero transverse measure. Consider a finite-area geodesic subsurface $Y$ of $X\setminus P_x$ whose boundary consists of finitely many cuffs in $\{ P_j\}_j$. The restriction of the horizontal foliation of $\varphi_u$ can be made transverse to the boundary by slightly modifying the finite surface $Y$. We consider the set of the horizontal leaves of $\varphi_u$ starting from the cuff(s) of $P_x$  corresponding to the incoming edges of $\mathcal{G}$. The Poincar\'e recurrence theorem for finite surfaces with boundary (see \cite[Theorem 5.2]{FLP}) says that almost every leaf ends at a boundary of the finite surface. By (\ref{eq:transient_walk}), the total transverse measure of the horizontal leaves of $\varphi_u$ on the cuff(s) of $P_x$ corresponding to the incoming edge(s) at $x$ is larger by $1$ from the transverse measure of the outgoing horizontal leaves of $\varphi_u$ cuff(s) of $P_x$. Since the horizontal leaves of $\varphi_u$ are oriented in $X\setminus P_x$, it follows that the leaves with transverse measure at least $1$ have to end up on the boundary sides of the finite surface $Y$ that are not cuffs of $P_x$.

Consider exhaustion of $X\setminus P_x$ by finite bordered surfaces $Y_n$ and, on each finite surface, the set of leaves going from $\partial P_x$ to other boundary components.
These sets of leaves (when considering the full extension of the leaves) are descending for the inclusion along the exhaustion of $X$ by finite subsurfaces. Therefore, the intersection (of the sets of fully extended leaves) is a set of transient leaves of the horizontal foliation of $\varphi_u$ of $X\setminus P_x$ with a transverse measure of at least $1$. Thus $X\notin O_G$.

Assume that $X\notin O_G$. In \cite[Theorem 4.2]{Saric23}, a finite area holomorphic quadratic differential is constructed so that almost all the leaves connect a single geodesic pair of pants with infinity. We only consider the first part of this construction. Namely, let $P_x$ be the fixed pair of pants of the pants decomposition that corresponds to a fixed vertex $x\in V$ of the pants graph $\mathcal{G}$. Then there exists a harmonic function $g:X\setminus P_x$ such that $g|_{\partial P_x}=0$ and $0< g(z)<1$ for all $z\in X\setminus \bar{P}_x$. Let $g^*(z)$ denote the local harmonic conjugate of $g$ away from the isolated zeroes. Then $d(g+ig^*)$ is an Abelian differential on $X\setminus \bar{P}_x$, and its horizontal foliation $\mathcal{F}_g$ consists of curves $g^*(z)=const$. The function $g(z)$ monotonically increases along each horizontal trajectory starting at $\partial P_x$. This gives a global orientation on the horizontal leaves of $\mathcal{F}_g$ that start at $\partial P_x$ and leave every finite area subset of $X\setminus P_x$ (we remove recurrent leaves). The transverse measure to the horizontal foliation $\mathcal{F}_g$ is given by integrating $dg^*$ along transverse arcs. Let $\mu$ be measured lamination obtained by extending $\mathcal{F}_g$ to a proper partial measured foliation on $X$, straightening the leaves and pushing forward the measure. The leaves of the support of $\mu$ in $X\setminus P_x$ inherit the orientation from the leaves of the horizontal foliation of $\mathcal{F}_g$.

We first choose an arbitrary orientation of all edges of $\mathcal{G}$. An edge $e$ connects $e^-=y_1$ to $e^+=y_2$, where $y_1$ and $y_2$ can be pairs of pants or hexagons or one of each kind. Let $a$ be the arc of the intersection of the closures of $y_1$ and $y_2$. We define $u(e)$ to be the transverse measure of the geodesics of the support of $\mu$ that intersect $a$ and have the same orientation as $e$ minus the transverse measure of the geodesics of the support of $\mu$ that intersect $a$ and have opposite orientation to $e$. Note that two arcs $a_1$ and $a_2$ may intersect at an endpoint. Since $\mu$ has no atoms, this does not induce overcounting or ambiguities. For every vertex $y\neq x$, the total transverse measure of the geodesics of the support of $\mu$ that enter the corresponding pair of pants or hexagon is equal to the transverse measure of the geodesics exiting it. Thus, we have that (\ref{eq:transient_walk}) holds for $y\neq x$, and it remains to consider the edges with one endpoint at $x$. Orient all edges to start at $x$ and define $u$ on each such edge to be the transverse measure of the geodesics of the support of $\mu$ that intersect corresponding arcs. Then (\ref{eq:transient_walk}) holds for $y= x$ with an appropriate constant. 
Since $\mu$ is in $ML_{int}(X)$ it satisfies 
$$
\sum_n [i(\alpha_n,\mu )^2+i(\beta_n,\mu )^2]+\sum_s i(\beta_s,\mu )^2 <\infty.
$$
The values $u(e)$ are bounded functions of the above intersection numbers of nearby geodesics. Thus
$u\in\ell^2(E)$. Therefore $\mathcal{G}$ is transient. This finishes the proof.
$\Box$

\vskip .4 cm
\noindent
{\it Proof of Theorem \ref{thm:realization-arbitrary}}.
We consider a Riemann surface $X=\mathbb{H}/\Gamma$ with $\Gamma$ infinite but not necessarily of the first kind. Recall that $X$ has at most countably many border components (also called visible ends in \cite{BS}). A border component is the simple closed curve on the boundary of a funnel or an open curve on the boundary of an attached half-plane to the open geodesic component of the boundary of the convex core of $X$ (see \cite{AR}, \cite{BS}).

Let $\mathcal{F}$ be a proper partial measured foliation on $X$ that has finite Dirichlet integral $D_X(\mathcal{F})$ which realizes $\mu\in ML_{int}(X)$. Our goal is to construct $\varphi\in A(X)$ whose horizontal foliation is homotopic to $\mathcal{F}$. Note that the leaves of $\mathcal{F}$ can have their endpoints on the border components of $X$ and that also the leaves can go off to infinity of $X$ without reaching the border. 

We form a double Riemann surface $\widehat{X}$ of $X$ by taking a mirror image of $X$ across the border and gluing along the border. The proper partial foliation $\mathcal{F}$ extends by doubling to a partial foliation $\widehat{\mathcal{F}}$ of $\widehat{X}$. The foliation $\widehat{\mathcal{F}}$ on $\widehat{X}$ is not proper in general since a leaf of $\mathcal{F}$ can be homotopic to a single border component of $X$ and the double of this leaf is a closed leaf homotopic to a point on $\widehat{X}$. 

We form a new surface $\widehat{X}_n$ for any integer $n>1$ as follows. Each border component of $X$ is an interior curve of $\widehat{X}$, which is invariant under the reflection symmetry. Define $A_1$ to be a set that has exactly one point in each border component of $X$. Assume that $A_n$ is defined such that each border component of $X$ contains finitely many points of $A_n$. Each border component of $X$ is divided into finitely many subarcs by $A_n$. The set
 $A_{n+1}$ is obtained by adding to $A_n$ one point in each subarc. The choices are made such that $\cup_nA_n$ is dense in each component of the border of $X$. We define
$$
\widehat{X}_n=\widehat{X}\setminus A_n.
$$

We form a partial measured foliation $\widehat{\mathcal{F}}_n$ of $\widehat{X}_n$ by first erasing at most countably many leaves of the foliation $\widehat{\mathcal{F}}$ that contain points in $A_n$. Then we erase all leaves that are either homotopic to a point in $\widehat{X}_n$ or to a puncture in $\widehat{X}_n$, i.e., homotopic to a single point in $A_n$. Except for the countably many leaves containing points in $\cup_nA_n$, any leaf of $\widehat{\mathcal{F}}$ is in $\widehat{\mathcal{F}}_n$ for $n$ large enough. Denote by $\mathcal{F}_n$ the restriction of $\widehat{\mathcal{F}}_n$ to $X$ and note that it is a proper partial measured foliation with the Dirichlet integral bounded by $D_X(\mathcal{F})$. 

We claim that the double Riemann surface $\widehat{X}$ has a covering Fuchsian group of the first kind. Indeed, the boundary components corresponding to the funnels of $X$ are closed geodesics in $\widehat{X}$, and $X$ is doubled across them. Therefore, all funnels disappear in $\widehat{X}$, and no new funnels are created by doubling. The boundary component of $X$ corresponding to the attached half-plane becomes an open bi-infinite geodesic in $\widehat{X}$. By \cite{AR} and \cite{BS}, in order to have a half-plane inside a Riemann surface, it is necessary that the boundary geodesic of the half-plane does not intersect any simple closed curve. The boundary component of $X$ that became an open bi-infinite geodesic $g$ in $\widehat{X}$ necessarily intersects any closed curve whose one half goes from the geodesic $g$ back to the geodesics $g$ inside $X$ and the second half is its mirror image in $\widehat{X}\setminus X$. Therefore, $g$ is not on the boundary of a geodesic half-plane in $\widehat{X}$. Any other open bi-infinite geodesic in $\widehat{X}$ either stays in $X$ or stays in its complement, or it intersects a boundary geodesic of $\widehat{X}$. If a bi-infinite geodesic stays in $X$, then it intersects a closed geodesic of $X$ (since it is not on the boundary of a half-plane in $X$). A similar argument works for geodesics in the mirror image of $X$. If a bi-infinite geodesic intersects a boundary of $X$, then it also intersects a closed geodesic in $\widehat{X}$ since it enters the convex core of $X$. Therefore, $\widehat{X}$ contains no funnels or geodesic half-planes; its covering group is of the first kind.

Since the covering group of $\widehat{X}$ is of the first kind, it follows that the covering group of $\widehat{X}_n$ is also of the first kind. Therefore, 
by \cite[Theorem 1.6]{Saric23} (see also Theorem \ref{thm:par-ch-cross-cuts}), there exists a finite-area holomorphic quadratic differential $\widehat{\varphi}_n$ on $\widehat{X}_n$ that realizes $\widehat{\mathcal{F}}_n$.  Since $\widehat{\mathcal{F}}_n$ is invariant under the mirror symmetry, by the uniqueness, it follows that $\widehat{\varphi}_n$ is invariant under the mirror symmetry as well. 

Let $\varphi_n$ be the restriction of $\widehat{\varphi}_n$ to $X$. If a ray of a horizontal trajectory of $\varphi_n$ meets the boundary of $X$, then it has a single accumulation point on the boundary because $\varphi_n$  is a finite-area quadratic differential (see \cite{MardenStrebel}). The sequence $\{ {\varphi}_n\}_n$ on $X$ has a subsequence which converges locally uniformly because $\int_{{X}} |{\varphi}_n|\leq D_X(\mathcal{F})$. 

Denote the convergent subsequence by ${\varphi}_n$ and its limit by ${\varphi}$. Then ${\varphi}$ is a finite-area holomorphic quadratic differential on $X$. Let $\gamma$ be a simple closed (hyperbolic) geodesic on $X$. By \cite[Lemma 3.8]{Saric23}, we have that
$$
h_{\varphi}(\gamma )=i(\gamma ,\mathcal{F}).
$$

In our case, the Riemann surface $X=\mathbb{H}/\Gamma$ has a border obtained by taking the quotient (with respect to $\Gamma$) of the intervals of discontinuity of $\Gamma$. Thus, the equality of the $\varphi$-heights of the closed curves and the intersection of the closed curves with the foliation $\mathcal{F}$ does not guarantee that the horizontal foliation of $\varphi$ is equivalent to the measured foliation $\mathcal{F}$. To prove that they are equivalent, we need to consider cross-cut arcs (i.e., arcs with both endpoints on border components of $X$) and their heights. If a trajectory of $\mathcal{F}$ connects two different components of the border of $X$, then it intersects a simple closed geodesic of $X$ (since it crosses the convex hull of $X$ by \cite{BS}). Therefore, it is enough to consider the cross-cuts connecting the same component of the border that do not essentially intersect any closed geodesic, or equivalently, the cross-cuts that can be homotoped to a single border component modulo their endpoints.

Let $\delta$ be a geodesic arc on $X$ (not intersecting any simple closed geodesic) with both endpoints at the same component $b$ of the border of $X$, i.e., the geodesic arc $\delta$ converges at each end to a single border (ideal) point of $X$ on the same component. Let $p$ be a short geodesic arc orthogonal to $\delta$ such that every geodesic orthogonal to $p$ embeds in $X$ and has both of its endpoints in $d$. Consider the intersection number of the measured foliation $\mathcal{F}$ with each geodesic orthogonal to $p$. Then, for almost every geodesic, the intersection number is finite since $\mathcal{F}$ has a finite Dirichlet integral. Therefore, among all geodesics homotopic to an arc in $b$ modulo their endpoints, there exists a dense set of geodesics that have a finite intersection number with $\mathcal{F}$. From now on, we consider geodesics $\delta$ with $$i(\delta ,\mathcal{F})<\infty.$$

Given $\epsilon >0$, there exists a simple differentiable arc $\delta_1$ in $X$ which is homotopic to $\delta$ modulo the endpoints such that
\begin{equation}
\label{eq:delta_1}
h_{\varphi}(\delta_1)>\int_{\delta_1}|dv(x)|-\epsilon ,
\end{equation}
where $\varphi =u+iv$ and, if $\varphi \equiv 0$, then we set $h_{\varphi}(\delta_1)=\int_{\delta_1}|dv(x)|=0$.

  The arc $\delta_1$ divides $X$ into two components. One component is simply connected, and its boundary consists of 
$\delta_1$ and an open subarc $b_1$ of $b$. The boundary component $b$ is a hyperbolic geodesic of $\widehat{X}$, and it is divided into subarcs 
$\{ d_i^n\}$ by adding punctures $A_n$ for each $n$. 

For any $b_1'\subset b_1$, denote by $\mathcal{F}_n(b_1',\delta_1 )$ the set of horizontal trajectories of $\varphi_n$ that have one endpoint in $b_1'$ and essentially cross $\delta_1$, i.e., trajectories that cannot be homotoped modulo endpoints to a curve that does not intersect $\delta_1$. 
The total transverse measure of $\mathcal{F}_n(b_1',\delta_1 )$ is given by $i(\delta , \mathcal{F}_n(b_1',\delta_1 ))$.

Given $\epsilon >0$, there exists a compact subarc $b_1'$ of the open arc $b_1$ such that the transverse measure of $\mathcal{F}_n(b_1',\delta_1 )$ is greater than $i(\delta ,\mathcal{F}_n)-\epsilon$ for all $n$ large enough, that is,
\begin{equation}
\label{eq:c_1'-est}
i(\delta ,\mathcal{F}_n(b_1',\delta_1 ))\geq i(\delta ,\mathcal{F}_n)-\epsilon .
\end{equation}
The set of trajectories $\mathcal{F}_n(b_1',\delta_1 )$ of $\varphi_n$ is in a one-to-one correspondence with the trajectories of $\mathcal{F}$ that have one endpoint in $b_1'$, essentially intersect $\delta$ and do not separate a single puncture (at the level $n$)  from the rest of surface. Therefore, for $n$ large enough depending on $b_1'$, the family $\mathcal{F}_n(b_1',\delta_1 )$ stabilizes (with the exception of the countably many leaves meeting the union of all $A_n$).

To prove the above inequality, assume that it is not true. Then there exist an $\epsilon_0>0$, an increasing sequence $\{c_k'\}_k$ of closed subintervals of $b_1$ and an increasing sequence of integers $\{ n_k\}_k$ with $\mathcal{F}_{n_k}(c_k',\delta_1 )$ stable such that (up to countably many leaves)
$$\mathcal{F}(b_1,\delta_1 )=\cup_k \mathcal{F}_{n_k}(c_k',\delta_1 )$$
and
$$
i(\delta ,\mathcal{F}_{n_k}(c_k',\delta_1 ))\leq i(\delta ,\mathcal{F})-\epsilon_0 .
$$ 
Since $ i(\delta ,\mathcal{F})<\infty$ and $\{\mathcal{F}_{n_k}(c_k',\delta_1 )\}_k$ is an increasing sequence of sets, it follows that $i(\delta ,\mathcal{F})=\lim_{k\to\infty} i(\delta ,\mathcal{F}_{n_k}(c_k',\delta_1 ))\leq i(\delta ,\mathcal{F})-\epsilon_0$ which is a contradiction. This proves (\ref{eq:c_1'-est}).

For a subarc $\delta_1'$ of $\delta_1$, let $\mathcal{F}_{\varphi_n}(c_1',\delta_1')$ denote the horizontal trajectories of $\varphi_n$ that start at $c_1'$ and intersect $\delta_1'$. We note that $\mathcal{F}_{\varphi_n}(c_1',\delta_1')$ depends on the size and position of $\delta_1'$ (not only on the homotopy classes) while each of its leaves is homotopic to a leaf of $\mathcal{F}_n$. 

Given $\epsilon >0$, there is a compact subarc $\delta_1'$ of $\delta_1$ such that the family of all curves connecting $c_1'$ to $\delta_1\setminus \delta_1'$ has modulus bounded above by $\epsilon^2$. This follows by noting that the curve family through a single point has zero modulus and that $\delta_1\setminus\delta_1'$ can be made arbitrarily small.
  
We estimate the transverse measure of the leaves in $\mathcal{F}_{\varphi_n}(c_1',\delta_1\setminus\delta_1')$, where $\mathcal{F}_{\varphi_n}(c_1',\delta_1\setminus\delta_1')$ consists of the horizontal trajectories of $\varphi_n$ that have one endpoint in $c_1'$ and intersect $\delta_1\setminus\delta_1'$. Let $ I_{c_1',\delta_1\setminus\delta_1'}$ be a countable union of vertical arcs for $\varphi_n$ that intersects each trajectory of $\mathcal{F}_{\varphi_n}(c_1',\delta_1\setminus\delta_1')$ in exactly one point and does not intersect any other trajectory. Note that the extremal metric for the family $\mathcal{F}_{\varphi_n}(c_1',\delta_1\setminus\delta_1')$ is $\frac{|\sqrt{\varphi_n}(z)dz|}{\ell (z)}$, where $\ell (z)$ is the $\varphi_n$-length of the horizontal trajectory passing through $z$ (for example, see \cite{HakobyanSaric}). Let $\varphi_n=u_n+iv_n$.  Then, by an application of the Cauchy-Schwarz inequality, we have
$$
\int_{I_{c_1',\delta_1\setminus\delta_1'}}|dv_n(x)|\leq \Big{[}\int_{I_{c_1',\delta_1\setminus\delta_1'}}\frac{1}{\ell_n (x)}|dv_n(x)|\Big{]}^{\frac{1}{2}}\cdot \Big{[}\int_{I_{c_1',\delta_1\setminus\delta_1'}}{\ell_n (x)}|dv_n(x)|\Big{]}^{\frac{1}{2}}.
$$
Since $\int_{I_{c_1',\delta_1\setminus\delta_1'}}\frac{1}{\ell_n (x)}|dv_n(x)|$ is the modulus of 
$\mathcal{F}_{\varphi_n}(c_1',\delta_1\setminus\delta_1')$ (see \cite{HakobyanSaric}) and since 
$\int_{I_{c_1',\delta_1\setminus\delta_1'}}\ell_n (x)|dv_n(x)|$ is the $\varphi_n$-area of the set covered by $\mathcal{F}_n(c_1',\delta_1\setminus\delta_1')$, it follows that the transverse measure to $\mathcal{F}_{\varphi_n}(c_1',\delta_1\setminus\delta_1')$ is bounded above by $D_X(\mathcal{F})^{\frac{1}{2}}\cdot\epsilon$. In other words,
\begin{equation}
\label{eq:c_1',delta-delta'}
\int_{I_{c_1',\delta_1\setminus\delta_1'}}|dv_n(x)|\leq D_X(\mathcal{F})^{\frac{1}{2}}\cdot\epsilon .
\end{equation}  
 Thus we have
 \begin{equation}
\label{eq:c_1'-delta'}
i(\delta ,\mathcal{F}_{\varphi_n}(c_1',\delta_1'))\geq i(\delta ,\mathcal{F}_{\varphi_n}(c_1',\delta_1))-D_X(\mathcal{F})^{\frac{1}{2}}\cdot\epsilon .
\end{equation}

 By (\ref{eq:c_1'-est}), (\ref{eq:c_1'-delta'}) and $i(\delta ,\mathcal{F}_{\varphi_n})=i(\delta ,\mathcal{F}_n)$, we get
\begin{equation}
\label{eq:est-1}
\int_{\delta_1'}|dv_n(x)|\geq i(\delta ,\mathcal{F}_{\varphi_n}(c_1',\delta_1'))\geq i(\delta ,\mathcal{F}_n)-(1+D_X(\mathcal{F})^{\frac{1}{2}})\epsilon
\end{equation}

Since $h_{\varphi_n}(\delta )= i(\delta ,\mathcal{F}_n)$ and $\varphi_n\to\varphi$ uniformly on the compact arc $\delta_1'$, the inequalities (\ref{eq:est-1}) and (\ref{eq:delta_1}) give
\begin{equation}
\label{eq:est-2}
\limsup_{n\to\infty}h_{\varphi_n}(\delta )-(1+D_X(\mathcal{F})^{\frac{1}{2}})\epsilon\leq \int_{\delta_1'}|dv(x)|\leq
\int_{\delta_1}|dv(x)|<h_{\varphi}(\delta )+\epsilon .
\end{equation}

By letting $\epsilon\to 0$ and noting that $\limsup_{n\to\infty}h_{\varphi_n}(\delta )=i(\delta ,\mathcal{F})$ by the definition on $\varphi_n$, we obtain
\begin{equation}
\label{eq:est-3-ineq}
i(\delta ,\mathcal{F})\leq h_{\varphi}(\delta ).
\end{equation}
In particular, $\varphi$ is a non-trivial holomorphic quadratic differential.

\begin{rem}
Strebel \cite{Strebel2} noted that it is not true that $\limsup_{n\to\infty}h_{\varphi_n}(\delta )\leq h_{\varphi}(\delta )$ for general $\varphi_n\to\varphi$ with locally uniform convergence. Our sequence $\varphi_n$ is special in the sense that the homotopy classes $\mathcal{F}_n$ of the horizontal foliations of $\varphi_n$ is an increasing sequence whose union (and the limit) is the foliation $\mathcal{F}$.
\end{rem}

It remains to prove that $h_{\varphi}(\delta )\leq i(\delta ,\mathcal{F})$. Since $\lim_{n\to\infty}h_{\varphi_n}(\delta )=i(\delta ,\mathcal{F})$, it is enough to prove that $\liminf_{n\to\infty}h_{\varphi_n}(\delta )\geq h_{\varphi}(\delta )$. Since $\int_X|\varphi_n|\leq D_X(\mathcal{F})$, this follows by Strebel \cite[Corollary 4.3]{Strebel2} in the case when $X$ is the unit disk. 

Assume $X$ is not the unit disk. Any regular trajectory $\alpha$ of $\varphi$ that essentially intersects $\delta$ has one endpoint in $b_1$. By the classification of regular trajectories of finite-area holomorphic quadratic differentials on an arbitrary Riemann surface $X$ \cite[\S 13]{Strebel}, it follows that almost every regular trajectory $\alpha$ that has one endpoint in $b_1$ also has the other endpoint on the border of $X$. The other border component on which $\alpha$ has an endpoint may be the same $b$ or a different component. Since we consider the heights, our statements will be up to measure zero. Thus, we can assume that that all trajectories $\alpha$ that essentially intersect $\delta$ have one endpoint in $b_1$ and the other endpoint also at a border component in the complement of the closure of $b_1$. 

Denote by $\tilde{\varphi}$, $\tilde{\varphi}_n$ the lifts of $\varphi$, $\varphi_n$ to the universal covering (identified with the unit disk $\mathbb{D}$). Denote by $\tilde{\alpha}$, $\tilde{b}_1\subset \tilde{b}$ and $\tilde{\delta}$ single 
components of the lifts of $\alpha$, $b_1$, $b$ and $\delta$. 
The other endpoint of 
$\tilde{\alpha}$ may belong to a component of the lift of $b_1$, but it cannot belong to $\tilde{b}_1$.

The pre-image on the unit circle of the border components of $X$ consists of countably many open pairwise disjoint arcs $\{ c_j\}_j$. We decompose the family of all regular trajectories of $\tilde{\varphi}$ that essentially intersect $\tilde{\delta}$ into a countable disjoint union of families $\mathcal{F}_{c_j}$ connecting the fixed subarc $\tilde{b}_1$ with $c_j$, and possibly the family $\mathcal{F}_{\tilde{b}\setminus \tilde{b}_1}$ connecting $\tilde{b}_1$ with $\tilde{b}\setminus \tilde{b}_1$, where the decomposition is up to the set of measure zero.

Note that $i(\tilde{\delta} ,\mathcal{F}_{\tilde{\varphi}})=i(\delta ,\mathcal{F}_{\varphi})<\infty$. Thus, given $\epsilon >0$ there exists finitely many $c_1,\ldots ,c_k$ such that
$$
i(\tilde{\delta} ,\mathcal{F}_{\tilde{\varphi}})-\sum_{j=1}^k i(\tilde{\delta} ,\mathcal{F}_{c_j}) -i(\tilde{\delta} ,\mathcal{F}_{\tilde{b}\setminus \tilde{b}_1})<\epsilon .
$$
Since the transverse measure of the trajectories of $\tilde{\varphi}$ with a common endpoint is zero, it follows that there exist compact subarcs $c_j'$ of $c_j$, and a compact subarc $\tilde{b}_1'$ of the interior of $\tilde{b}\setminus \tilde{b}_1$ such that
$$
i(\tilde{\delta} ,\mathcal{F}_{c_j})-i(\tilde{\delta} ,\mathcal{F}_{c_j'})<\frac{\epsilon}{k},
$$
for $j=1,\ldots ,k$, where $\mathcal{F}_{c_j'}$ is the family of horizontal trajectories of $\tilde{\varphi}$ that connect $\tilde{b}_1$ and $c_j'$,
and
 $$
i(\tilde{\delta} ,\mathcal{F}_{\tilde{b}\setminus \tilde{b}_1})-i(\tilde{\delta} ,\mathcal{F}_{\tilde{b}_1'})<{\epsilon},
$$
where $\mathcal{F}_{\tilde{b}_1'}$ is the family of horizontal trajectories of $\tilde{\varphi}$ that connect $\tilde{b}_1$ and $\tilde{b}_1'$.

Denote by $v_j^l$ and $v_j^r$ the two horizontal trajectories of $\tilde{\varphi}$ in 
$\mathcal{F}_{c_j'}$ such that all other trajectories of $\mathcal{F}_{c_j'}$ are between these two. If either of $v_j^l$ and $v_j^r$ is not a regular trajectory, we replace it with a regular trajectory connecting $\tilde{b}_1$ with $c_j'$ such that the transverse measure of the trajectories connecting $\tilde{b}_1$ with $c_j'$ but not between $v_j^l$ and $v_j^r$ is less than $\epsilon /k$. 

The part of $\mathbb{D}$ between $v_j^l$ and $v_j^r$ is a simply connected domain $D_j$, and there exists $N>0$ such that the number of points in the preimage in $D_j$ of each point of $X$ is at most $N$. Therefore, the integrals 
$\int_{D_j}|\tilde{\varphi}_n|$ and $\int_{D_j}|\tilde{\varphi}|$ are bounded above by $N\cdot D_{X}(\mathcal{F})$. We replace $v_j^l$ and $v_j^r$ by totally regular trajectories $(v_j^l)'$ and $(v_j^r)'$ for $\tilde{\varphi}$ in $D_j$ such that the transverse measure of the trajectories separating $v_j^l$ and $(v_j^l)'$, as well as the trajectories separating $v_j^r$ and $(v_j^r)'$ is less than $\epsilon$.

By Strebel \cite[Theorem 3.2.2]{Strebel-disk}, there exist two totally regular trajectories $t_{j,n}^l$ and $t_{j,n}^r$ of $\tilde{\varphi}_n$ that are in the prescribed Euclidean neighborhoods of $(v_j^l)'$ and $(v_j^r)'$ inside $D_j$. We choose these neighborhoods to be disjoint from $v_j^l$ and $v_j^r$. It follows that the $\tilde{\varphi}$-height between $(v_j^l)'$ and $(v_j^r)'$ inside $D_j$ is equal to the limit infimum of the $\tilde{\varphi}_n$-heights between $t_{j,n}^l$ and $t_{j,n}^r$ as $n\to\infty$ by \cite[Theorem 4.3]{Strebel2}. This implies that
$$
\liminf_{n\to\infty}h_{\tilde{\varphi}_n}(\delta )\geq h_{\tilde{\varphi}}(\delta )-6\epsilon .
$$
By letting $\epsilon\to 0$ and noting that $h_{\tilde{\varphi}_n}(\tilde{\delta} )=h_{{\varphi}_n}(\delta )$ and $h_{\tilde{\varphi}}(\tilde{\delta} )=h_{{\varphi}}(\delta )$, we obtain
$$
\liminf_{n\to\infty}h_{{\varphi}_n}(\delta )\geq h_{{\varphi}}(\delta )
$$
which finishes the proof.
$\Box$


\begin{thebibliography}{Thua}

\vskip .5cm

\bibitem{Ahlfors}  L. V. Ahlfors, {\it Conformal invariants: topics in geometric function theory}, McGraw-Hill Series in Higher Mathematics. McGraw-Hill Book Co., New York-D\"usseldorf-Johannesburg, 1973. 

\bibitem{Ahlfors1} L. V. Ahlfors, {\it Remarks on the classification of open Riemann surfaces}, Ann. Acad. Sci. Fennicae Ser. A. I. Math.-Phys. 1951 (1951), no. 87, 8 pp.

\bibitem{AhlforsSario} L. V. Ahlfors and L. Sario, {\it Riemann surfaces}, 
Princeton Mathematical Series, No. 26 Princeton University Press, Princeton, N.J. 1960.



\bibitem{ALPS}  D. Alessandrini, L. Liu, A. Papadopoulos and W. Su, {\it On local comparison between various metrics on Teichm\" uller spaces}, Geom. Dedicata 157 (2012), 91-110.






\bibitem{AR}   V.  Alvarez and J.M. Rodriguez, {\it Structure
Theorems for Riemann and topological surfaces}, J. London Math. Soc. (2) 69
(2004), 153-168.

\bibitem{Astala-Zinsmeister}
K. Astala and M. Zinsmeister, {\it{ Mostow rigidity and Fuchsian groups.} }
C. R. Acad. Sci. Paris S\'er. I Math. 311 (1990), no. 6, 301-306.

\bibitem{Bas} A. Basmajian, {\it  Hyperbolic structures for surfaces of infinite type}, Trans. Amer. Math. Soc. 336, no. 1, March 1993, 421-444.



\bibitem{BHS} A. Basmajian, H. Hakobyan and D. \v Sari\' c, {\it The type problem for Riemann surfaces via Fenchel-Nielsen parameters}, Proc. Lond. Math. Soc. (3) 125 (2022), no. 3, 568-625. 

\bibitem{BPV} A. Basmajian, H. Parlier and N. Vlamis, {\it Bounded geometry with no bounded pants decomposition}, Israel J. Math. 260 (2024), no. 1, 235-260. 

\bibitem{BS} A. Basmajian and D. \v Sari\' c, {\it Geodesically complete hyperbolic structures}, Math. Proc. Cambridge Philos. Soc. 166 (2019), no. 2, 219-242.

\bibitem{Bear} A. Beardon, {\it The geometry of discrete groups},  
Graduate Texts in Mathematics, 91. Springer-Verlag, New York, 1983.

\bibitem{BMS} I. Benjamini, S. Merenkov and O. Schramm, {\it A negative answer to Nevanlinna's type question and a parabolic surface with a lot of negative curvature}, Proc. Amer. Math. Soc. 132 (2004), no. 3, 641-647. 

\bibitem{bers}  L. Bers,  {\it Universal Teichm\"uller space}, Analytic methods in mathematical physics (Sympos., Indiana Univ., Bloomington, Ind., 1968), pp. 65-83. Gordon and Breach, New York, 1970.

\bibitem{Bishop} C. Bishop, \textit{Divergence groups have the Bowen property}, Ann. of Math. (2) 154 (2001), no. 1, 205-217.

\bibitem{Bonahon} F. Bonahon, {\it Closed curves on surfaces}, manuscript.
 

\bibitem{Buser} P. Buser,  {\it Geometry and spectra of compact Riemann surfaces},
Progress in Mathematics, 106. Birkhuser Boston, Inc., Boston, MA,  1992.

\bibitem{FLP} A. Fathi, F. Laudenbach and V. Po\'enaru, {\it
Thurston's work on surfaces},
Translated from the 1979 French original by Djun M. Kim and Dan Margalit. Mathematical Notes, 48. Princeton University Press, Princeton, NJ, 2012. 

\bibitem{Fernandez-Melian}  J. L. Fern\'andez and M. V. Meli\'an, {\it Escaping geodesics of Riemannian surfaces.} Acta Math. 187 (2001), no. 2, 213-236.


\bibitem{GardinerLakic}  F. Gardiner and N. Lakic, {\it Quasiconformal Teichm\" uller theory}, Mathematical Surveys and Monographs, 76. American Mathematical Society, Providence, RI, 2000. 

\bibitem{GM}  L. Geyer and S. Merenkov, {\it A hyperbolic surface with a square grid net}, J. Anal. Math. 96 (2005), 357-367.

\bibitem{HPS} H. Hakobyan, M. Pandazis and D. \v Sari\' c, {\it How twisting induces ergodicity on infinite Riemann surfaces}, in preparation.

\bibitem{HakobyanSaric} H. Hakobyan and D. \v Sari\' c, {\it Limits of Teichm\" uller geodesics in the universal Teichm\" uller space}, Proc. Lond. Math. Soc. (3) 116 (2018), no. 6, 1599-1628. 

\bibitem{hopf} E. Hopf, {\it  Ergodic theory and the geodesic flow on surfaces of constant negative curvature},  Bull. Amer. Math. Soc. 77 (1971), 863-877. 


 \bibitem{HubbardMasur} J. Hubbard and H. Masur, {\it Quadratic differentials and foliations}, Acta Math. 142 (1979), no. 3-4,  221-274.
 
\bibitem{Jenkins} J. A. Jenkins, {\it On the existence of certain general extremal metrics}, Ann. of Math.,
66 (1957), 440-453.

\bibitem{Kaim} V. A. Kaimanovich, {\it Ergodic properties of the horocycle flow and classification of Fuchsian groups}, J. Dynam. Control Systems 6 (2000), no. 1, 21-56. 

\bibitem{kanai}  M. Kanai, {\it Rough isometries and the parabolicity of Riemannian manifolds}, J. Math. Soc. Japan 38 (1986), no. 2, 227-238.

\bibitem{Kerckhoff} S. Kerckhoff, {\it The asymptotic geometry of Teichm\" uller space}, Topology 19 (1980), no. 1, 23-41.

\bibitem{Ker} B. K\'er\'ekj\'arto, {\it   Vorlesungen \"uber Topologie}  I, Springer, Berlin, 1923.

\bibitem{Kinjo2010} E. Kinjo, {\it On Teichm\" uller metric and the length spectrums of topologically infinite Riemann surfaces}, Kodai Math. J. 34 (2011), no. 2, 179-190.

\bibitem{Kinjo}  E. Kinjo, {\it On the length spectrum Teichmüller spaces of Riemann surfaces of infinite type}, Conform. Geom. Dyn. 22 (2018), 1-14.

\bibitem{HPS} H. Hakobyan, M. Pandazis and D. \v Sari\' c, {\it How twisting produces ergodicity on infinite Riemann surfaces}, in preparation.

\bibitem{Levitt}  G. Levitt, {\it Foliations and laminations on hyperbolic surfaces}, Topology 22 (1983), no. 2, 119-135. 

\bibitem{Lyons}  T. Lyons, {\it Instability of the Liouville property for quasi-isometric Riemannian manifolds and reversible Markov chains}, J. Differential Geom. 26 (1987), no. 1, 33-66.

\bibitem{LMcK} T. Lyons and H. McKean, {\it
Winding of the plane Brownian motion}, 
Adv. in Math. 51 (1984), no. 3, 212-225. 



\bibitem{LyonsSullivan} T. Lyons and D. Sullivan, {\it 
Function theory, random paths and covering spaces}, 
J. Differential Geom. 19 (1984), no. 2, 299-323. 

\bibitem{MardenStrebel} A. Marden and K. Strebel, {\it On the ends of trajectories},  Differential geometry and complex analysis, 195-204, Springer, Berlin, 1985.

\bibitem{MardenStrebel1} A. Marden and K. Strebel, {\it The heights theorem for quadratic differentials on Riemann surfaces}, Acta Math. 153 (1984), no. 3-4, 153-211.


\bibitem{markovic}  V. Markovic, {\it Biholomorphic maps between Teichm\"uller spaces}, Duke Math. J. 120 (2003), no. 2, 405-431.


\bibitem{Maskit} B. Maskit, {\it Comparison of hyperbolic and extremal lengths}, 
Ann. Acad. Sci. Fenn. Ser. A I Math. 10 (1985), 381-386. 

\bibitem{masur-cyl} H. Masur, {\it 
The Jenkins-Strebel differentials with one cylinder are dense}, 
Comment. Math. Helv. 54 (1979), no. 2, 179-184.

\bibitem{McKS} H. McKean and D. Sullivan, {\it Brownian motion and harmonic functions on the class surface of the thrice punctured sphere},
Adv. in Math. 51 (1984), no. 3, 203-211. 


\bibitem{mcmullen} C. McMullen, {\it Hausdorff dimension and conformal dynamics. III. Computation of dimension}, Amer. J. Math. 120 (1998), no. 4, 691-721. 

\bibitem{Merenkov}  S. Merenkov, {\it Rapidly growing entire functions with three singular values}, Illinois J. Math. 52 (2008), no. 2, 473-491.

\bibitem{Milnor} J. Milnor, {\it  On deciding whether a surface is parabolic or hyperbolic}, Amer. Math. Monthly 84 (1977), no. 1, 43-46.



\bibitem{Mori} A. Mori, \textit{A note on unramified abelian covering surfaces of a closed Riemann surface}, J. Math. Soc. Japan \textbf{6} (1954), 162-176.

\bibitem{Nevanlinna:criterion} R. Nevanlinna, \textit{\"Uber die Existenz von beschr\"ankten Potentialfunktionen auf Fl\"achen von unendlichem Geschlecht.} (German)
Math. Z. 52 (1950), 599-604.

\bibitem{Nicholls1} P. Nicholls, {\it The ergodic theory of discrete groups}, London Mathematical Society Lecture Note Series, 143. Cambridge University Press, Cambridge, 1989.

\bibitem{Pandazis} M. Pandazis, {\it Non-ergodicity of the geodesic flow on a special class of Cantor tree surfaces}, to appear in Proc. Amer. Math. Soc.

\bibitem{PandazisSaric} M. Pandazis and D. \v Sari\' c, {\it Ergodicity of the geodesic flow on symmetric surfaces}, Trans. Amer. Math. Soc. 376 (2023), no. 10, 7013-7043.

\bibitem{PennerHarer} R. Penner and J. Harer, 
{\it Combinatorics of train tracks}.
Annals of Mathematics Studies, 125. Princeton University Press, Princeton, NJ, 1992.


\bibitem{Rees} M. Rees, {\it Checking ergodicity of some geodesic flows with infinite Gibbs measure}, Ergodic Theory Dynam. Systems 1 (1981), no. 1, 107-133.

\bibitem{reichstrebel} E. Reich and K. Strebel, {\it Teichm\"uller mappings which keep the boundary point-wise fixed},  1971 Advances in the Theory of Riemann Surfaces (Proc. Conf., Stony Brook, N.Y., 1969) pp. 365-367 Princeton Univ. Press, Princeton, N.J.

\bibitem{Sarictt} D. \v Sari\' c, {\it Train tracks and measured laminations on infinite surfaces}, Trans. Amer. Math. Soc. 374 (2021), no. 12, 8903-8947.

\bibitem{Saric-heights} D. \v Sari\' c, {\it The heights theorem for infinite Riemann surfaces}, Geom. Dedicata 216 (2022), no. 3, Paper No. 33, 23 pp. 

\bibitem{Saric23} D. \v Sari\' c, {\it Quadratic differentials and foliations on infinite Riemann surfaces}, to appear in Duke Math. J.

\bibitem{saric} D. \v Sari\' c, {\it Circle homeomorphisms and shears}, Geom. Topol. 14 (2010), no. 4, 2405-2430.

\bibitem{SarioNakai}  L. Sario and M. Nakai, {\it Classification theory of Riemann surfaces}, Die Grundlehren der mathematischen Wissenschaften, Band 164. Springer-Verlag, New York-Berlin, 1970.

\bibitem{Shiga} H. Shiga, {\it On a distance defined by the length spectrum of Teichm\"uller space}, Ann. Acad. Sci. Fenn. Math. 28 (2003), no. 2, 315-326. 

\bibitem{Strebel1} K. Strebel, {\it  \"Uber quadratische differentials mit geschlossen trajectorien and extremale
quasikonforme abbildungen}, In Festband zum 70, Geburstag von Rolf Nevanlinna,
Springer-Verlag, Berlin, 1966.

\bibitem{Strebel} K. Strebel, {\it Quadratic differentials}, Ergebnisse der Mathematik und ihrer Grenzgebiete (3) [Results in Mathematics and Related Areas (3)], 5. Springer-Verlag, Berlin, 1984. 


\bibitem{Strebel2} K. Strebel, {\it The mapping by heights for quadratic differentials in the disk}, Ann. Acad. Sci. Fenn. Ser. A I Math. 18 (1993), no. 1, 155-190.

\bibitem{Strebel-disk} K. Strebel, {\it On the geometry of quadratic differentials in the disk}, Results Math. 22 (1992), no. 3-4, 799-816. 


\bibitem{Sullivan} D. Sullivan, {\it On the ergodic theory at infinity of an arbitrary discrete group of hyperbolic motions.} Riemann surfaces and related topics: Proceedings of the 1978 Stony Brook Conference (State Univ. New York, Stony Brook, N.Y., 1978), pp. 465-496, Ann. of Math. Stud., 97, Princeton Univ. Press, Princeton, N.J., 1981. 

\bibitem{Thurston} W. Thurston, {\it Hyperbolic geometry and 3-manifolds}, Reprint of Low-dimensional topology (Bangor, 1979), 1982, Cambridge Univ. Press, Cambridge-New York, 9-25. Collected works of William P. Thurston with commentary. Vol. II. 3-manifolds, complexity and geometric group theory, 153-169, Amer. Math. Soc., Providence, RI, 2022.

\bibitem{Toki} Y. T\^oki, {\it On the examples in the classification of open Riemann surfaces. I.} Osaka Math. J. 5 (1953), 267-280. 

\bibitem{Woess}  W. Woess, {\it Random walks on infinite graphs and groups.} Cambridge Tracts in Mathematics, 138. Cambridge University Press, Cambridge, 2000.


\end{thebibliography}
\end{document}